\documentclass[reqno, oneside]{amsart}

\usepackage{amsmath,amsthm}
\usepackage{comment}
\usepackage{amsfonts, fullpage, fancyhdr, mathtools, qtree, amsmath, tipa, amssymb, hyperref, url, amsthm, subfigure, xy, tikz-cd, verbatim, bbm, textcomp, enumitem, mathdots, etex, extpfeil, extarrows, cleveref}

\usepackage[a4paper]{geometry}
\tikzset{
    labl/.style={anchor=south, rotate=90, inner sep=.5mm}
}

\bgroup
\makeatletter
\catcode`&=\active
\gdef\spandiagram#1{%
    \let\spandiagram@rowone\empty
    \let\spandiagram@rowtwo\empty
    \spandiagram@start#1\separator-\null<\null\nullll
}

\gdef\spandiagram@start#1<{%
    \pgfutil@in@-{#1}%
    \ifpgfutil@in@
        \def\next{\spandiagram@start@top#1<}%
    \else
        \pgfutil@g@addto@macro\spandiagram@rowtwo{#1}%
        \let\next\spandiagram@top
    \fi
    \next
}

\gdef\spandiagram@start@top#1-{%
    \pgfutil@g@addto@macro\spandiagram@rowone{#1}%
    \spandiagram@forward
}

\gdef\spandiagram@top#1-#2-{%
    \pgfutil@g@addto@macro\spandiagram@rowone{& \dlar["#1"']#2}%
    \pgfutil@g@addto@macro\spandiagram@rowtwo{&}%
    \@ifnextchar\null{%
       \spandiagram@trimseparator\spandiagram@rowone
       \spandiagram@finish
    }{\spandiagram@forward}%
}

\gdef\spandiagram@forward#1>#2<{%
    \pgfutil@g@addto@macro\spandiagram@rowone{\drar["#1"]&}%
    \pgfutil@g@addto@macro\spandiagram@rowtwo{&#2}%
    \@ifnextchar\null{%
        \spandiagram@trimseparator\spandiagram@rowtwo
        \spandiagram@finish
    }{\spandiagram@top}%
}

\gdef\spandiagram@trimseparator#1{%
    \edef#1{\unexpanded\expandafter\expandafter\expandafter{\expandafter\spandiagram@trimseparator@helper#1\nulll}}%
}
\gdef\spandiagram@trimseparator@helper#1\separator#2\nulll{#1}

\gdef\spandiagram@finish{%
    \begin{tikzcd}[row sep=small,column sep=small]
    \spandiagram@rowone\\
    \spandiagram@rowtwo\\
    \end{tikzcd}%
    \spandiagram@gobbletonullll
}

\gdef\spandiagram@gobbletonullll#1\nullll{}
\egroup
\theoremstyle{remark}
\newtheorem*{thmArough}{Theorem A}
\newtheorem*{thmBrough}{Theorem B}
\newtheorem*{thmCrough}{Theorem C}

\theoremstyle{definition}
\newtheorem{nul}{}[section]
\newtheorem{defn}[nul]{Definition}
\newtheorem{defnprop}[nul]{Definition/Proposition}

\newtheorem{rmk}[nul]{Remark}

\newtheorem{cnstr}[nul]{Construction}

\newtheorem{notation}[nul]{Notation}
\newtheorem{exm}[nul]{Example}
\newtheorem{cnv}[nul]{Convention}

\newtheorem{obs}[nul]{Observation}

\newtheorem{rec}[nul]{Recollection}

\newtheorem{warn}[nul]{Warning}
\newtheorem{qst}[nul]{Question}
\newtheorem*{defn*}{Definition}
\newtheorem*{axm*}{Axiom}
\newtheorem*{notation*}{Notation}
\newtheorem*{exm*}{Example}
\newtheorem*{exr*}{Exercise}
\newtheorem*{int*}{Intuition}
\newtheorem*{qst*}{Question}
\newtheorem*{obs*}{Observation}
\newtheorem*{rmk*}{Remark}
\newtheorem*{cnv*}{Convention}

\newtheorem{thm}[nul]{Theorem}
\newtheorem{prop}[nul]{Proposition}
\newtheorem{lem}[nul]{Lemma}
\newtheorem{var}[nul]{Variant}

\newtheorem{cnj}[nul]{Conjecture}
\newtheorem{cor}{Corollary}[nul]
\newtheorem*{thm*}{Theorem}
\newtheorem*{thmA}{Theorem A}
\newtheorem*{thmB}{Theorem B}
\newtheorem*{thmC}{Theorem C}
\newtheorem*{thmA'}{Theorem A'}
\newtheorem*{thmA-}{Theorem A-}
\newtheorem*{thmAnat}{Theorem $\text{A}^{\natural}$}
\newtheorem*{thmAcirc}{Theorem $\text{A}^{\circ}$}
\newtheorem*{prop*}{Proposition}
\newtheorem*{cor*}{Corollary}
\newtheorem*{claim*}{Claim}
\newtheorem*{claim1*}{Claim 1}
\newtheorem*{claim2*}{Claim 2}
\newtheorem*{lem*}{Lemma}
\newtheorem*{conj*}{Conjecture}

    \newtheoremstyle{TheoremNum}
        {8pt}{\topsep}              
        {}
        {}                              
        {\bfseries}                     
        {.}                             
        { }                             
        {\thmname{#1}\thmnote{ \bfseries #3}}
    \theoremstyle{TheoremNum}
    \newtheorem{thmn}{Theorem}

\DeclareMathOperator*{\colim}{\mathrm{colim}}
\DeclareMathOperator{\Map}{\mathrm{Map}}

\DeclareMathOperator{\fib}{\mathrm{fib}}
\DeclareMathOperator{\Hom}{\mathrm{Hom}}

\newcommand{\sslash}{\mathbin{/\mkern-6mu/}}

\DeclareMathOperator{\A}{\mathcal{A}}
\DeclareMathOperator{\C}{\mathcal{C}}
\DeclareMathOperator{\D}{\mathcal{D}}

\DeclareMathOperator{\Z}{\mathbb{Z}}
\DeclareMathOperator{\bS}{\mathbb{S}}
\DeclareMathOperator{\E}{\mathbb{E}}
\DeclareMathOperator{\mE}{\mathcal{E}}
\DeclareMathOperator{\Q}{\mathbb{Q}}

\DeclareMathOperator{\F}{\mathbb{F}}

\newcommand{\pS}{{\mathbb{S}^{\wedge}_p}}
\newcommand{\pSbar}{{\overline{\mathbb{S}}^{\wedge}_p}}
\newcommand{\pSff}{({\mathbb{S}^{\wedge}_p})_{\varphi=1}}
\newcommand{\bSff}{\mathbb{S}_{\varphi=1}}

\DeclareMathOperator{\Span}{\mathrm{Span}}
\DeclareMathOperator{\Gpd}{\mathrm{Gpd}}
\DeclareMathOperator{\Mod}{\mathrm{Mod}}
\DeclareMathOperator{\Mon}{\mathrm{Mon}}
\DeclareMathOperator{\Cat}{\mathrm{Cat}}
\DeclareMathOperator{\Fin}{\mathrm{Fin}}
\DeclareMathOperator{\red}{\mathrm{red}}

\DeclareMathOperator{\End}{\mathrm{End}}
\DeclareMathOperator{\Fun}{\mathrm{Fun}}
\DeclareMathOperator{\Seg}{\mathrm{Seg}}
\DeclareMathOperator{\CplSeg}{\mathrm{CplSeg}}
\DeclareMathOperator{\Alg}{\mathrm{Alg}}

\DeclareMathOperator{\CAlg}{\mathrm{CAlg}}
\DeclareMathOperator{\can}{\mathrm{can}}

\newcommand{\QVect}{{Q\mathrm{Vect}_{\mathbb{F}_p}}}
\newcommand{\CAlgperfp}{{\CAlg^{\mathrm{perf}}_p}}
\newcommand{\CAlgperf}{{\CAlg^{\mathrm{perf}}}}
\DeclareMathOperator{\Sp}{\mathrm{Sp}}
\DeclareMathOperator{\Spaces}{\mathcal{S}}
\newcommand{\Spacesfnp}{{\mathcal{S}^{\mathrm{fsc}}_p}}
\newcommand{\Spacesfn}{{\mathcal{S}^{\mathrm{fsc}}}}
\newcommand{\Spacesf}{{\mathcal{S}^{\mathrm{f}}}}
\newcommand{\CAlgffp}{{\CAlg^{\varphi =1}_p}}
\newcommand{\CAlgff}{{\CAlg^{\varphi =1}}}
\newcommand{\Fpbar}{{\overline{\F}_p}}
\newcommand{\ff}{{\varphi =1}}
\DeclareMathOperator{\Set}{\mathrm{Set}}

\DeclareMathOperator{\id}{\mathrm{id}}
\DeclareMathOperator{\res}{\mathrm{res}}

\DeclareMathOperator{\sect}{\mathrm{sect}}
\DeclareMathOperator{\op}{\mathrm{op}}
\DeclareMathOperator{\triv}{\mathrm{triv}}
\DeclareMathOperator{\Bor}{\mathrm{Bor}}
\DeclareMathOperator{\Glo}{\mathrm{Glo}}
\DeclareMathOperator{\Glop}{\mathrm{Glo}^+}

\DeclareMathOperator{\GloSp}{\mathrm{GloSp}}
\DeclareMathOperator{\GlopSp}{\mathrm{Glo}^+\mathrm{Sp}}
\DeclareMathOperator{\GloSpBor}{\mathrm{Glo}\mathrm{Sp}_{\mathrm{Bor}}}

\DeclareMathOperator{\Sq}{\mathrm{Sq}}

\usepackage{tikz}
\usetikzlibrary{matrix,arrows,decorations}
\usepackage{tikz-cd}

\usepackage{adjustbox}

\begin{document}

\begin{abstract}
We give a fully faithful integral model for simply connected finite complexes in terms of $\E_{\infty}$-ring spectra and the Nikolaus--Scholze Frobenius.  The key technical input is the development of a homotopy coherent Frobenius action on a certain subcategory of $p$-complete $\E_{\infty}$-rings for each prime $p$.  Using this, we show that the data of a simply connected finite complex $X$ is the data of its Spanier-Whitehead dual, as an $\E_{\infty}$-ring, together with a trivialization of the Frobenius action after completion at each prime.  

In producing the above Frobenius action, we explore two ideas which may be of independent interest.  The first is a more general action of Frobenius in equivariant homotopy theory; we show that a version of Quillen's $Q$-construction acts on the $\infty$-category of $\E_{\infty}$-rings with ``genuine equivariant multiplication," which we call global algebras.  The second is a ``pre-group-completed" variant of algebraic $K$-theory which we call \emph{partial $K$-theory}.  We develop the notion of partial $K$-theory and give a computation of the partial $K$-theory of $\F_p$ up to $p$-completion.  

\end{abstract}

\title{Integral models for spaces via the higher Frobenius}
\author{Allen Yuan}


\setcounter{tocdepth}{1}
\maketitle

\tableofcontents


\section{Introduction}


In algebra, one often studies questions over the integers by understanding them over $\Q$ and after completion at each prime $p$.  The analog of this idea in topology, due to Sullivan \cite{Sullivan}, is that any space $X$ can be approximated by its rationalization $X_{\Q}$ and its $p$-completions $X^{\wedge}_p$.  These approximations, in turn, can be understood in terms of their algebras of cochains.  Work of Sullivan (and Quillen in a dual setting) shows that for sufficiently nice $X$, the rationalization $X_{\Q}$ is captured completely by a commutative differential graded algebra (cdga) which is quasi-isomorphic to $C^*(X,\Q)$.  More precisely:
\begin{thm}[Sullivan \cite{SullivanQ}, Quillen \cite{QuillenRatl}]\label{thm:RHT}
The assignment $X\mapsto C^*(X,\Q)$ determines a fully faithful functor from the $\infty$-category of simply connected rational spaces of finite type to the $\infty$-category $\CAlg_{\Q}$ of rational cdgas.  
\end{thm}
In the $p$-adic case, Mandell proves an analogous result with cdgas over $\Q$ replaced by $\E_{\infty}$-algebras over $\Fpbar$:
\begin{thm}[Mandell \cite{Mandell}] \label{thm:Mandell}
The assignment $X \mapsto C^*(X, \overline{\F}_p)$ determines a fully faithful functor from the $\infty$-category of simply connected $p$-complete spaces of finite type to the $\infty$-category $\CAlg_{\overline{\F}_p}$ of $\E_{\infty}$-algebras over $\overline{\F}_p$.  
\end{thm}

In the above theorems, we say a rational (resp. $p$-complete) space is of finite type if each homotopy group is finitely generated over $\Q$ (resp. $\Z_p)$.  
The next natural question is whether a similar cochain model for spaces exists \emph{integrally}.  

\begin{warn}
Mandell \cite{MandellZ} shows that the assignment $X \mapsto C^*(X;\Z)$ determines a functor from spaces to $\E_{\infty}$-algebras over $\Z$ that is faithful but \emph{not} full.  

Let us illustrate the difficulty in producing such an integral model: Sullivan showed that any sufficiently nice space $X$ can be recovered from its rationalization and $p$-completions together with the additional data of a map $$X_{\Q} \to \big(\prod_p X^{\wedge}_p\big)_{\Q}$$ via a homotopy pullback square
\begin{equation*}
\begin{tikzcd}
X \arrow[r]\arrow[d] &\displaystyle{ \prod_p X^{\wedge}_p }\arrow[d]\\
X_{\Q} \arrow[r] &\displaystyle{\big(\prod_p X^{\wedge}_p\big)_{\Q}}.
\end{tikzcd}
\end{equation*}
One might hope to construct the desired integral model for spaces by assembling the above cochain models via an analogous procedure.  However, it is unclear how this assembly would work: for instance, one might hope to relate the algebra models -- $C^*(X;\Q)$ for $X_{\Q}$ and $C^*(X;\Fpbar)$ for $X^{\wedge}_p$ -- via some comparison over $\Q_p$; but it does not seem obvious how to obtain a $\Q_p$ algebra from $C^*(X;\Fpbar)$.  \end{warn}

\subsection{Summary of results}

In this paper, we give an \emph{integral} cochain model for simply connected finite spaces.  We accomplish this by producing a Frobenius action on certain $\E_{\infty}$-ring spectra and modeling spaces in terms of the fixed points of this action. 
We now state our main theorems informally, with more refined statements to follow later in this introduction.  

\begin{thmArough}[Frobenius Action] For each prime $p$, the Nikolaus--Scholze Frobenius determines an action of the monoidal category $B\Z_{\geq 0}$ on the $\infty$-category of $p$-complete $\E_{\infty}$-rings which are finite over the $p$-complete sphere.
\end{thmArough}

In fact, as we explain shortly, the action exists on a larger $\infty$-category of $\E_\infty$-rings which we call $F_p$-stable.  We may then consider the homotopy fixed points for this Frobenius action, which we call \emph{$p$-Frobenius fixed $\E_{\infty}$-rings}.  We will explain that a $p$-Frobenius fixed $\E_{\infty}$-ring is the data of an $\E_{\infty}$-ring $A$ equipped with a particular sequence of homotopies, the first of which is a trivialization of its Frobenius endomorphism (cf. Warning \ref{warn:Fptriv}).  Using this notion, we  give a $p$-adic model for spaces:

\begin{thmBrough}[$p$-adic Model]
Let $p$ be a prime and let $X$ be a simply connected finite complex.  Then:
\begin{itemize}
\item The $\E_{\infty}$-algebra $(\pS)^X$ of cochains on $X$ with values in the $p$-complete sphere is naturally a $p$-Frobenius fixed $\E_{\infty}$-ring.
\item The data of the space $X^{\wedge}_p$ is captured completely by $(\pS)^X$ as a $p$-Frobenius fixed $\E_{\infty}$-ring.
\end{itemize}
\end{thmBrough}
These $p$-adic models are amenable to assembly into an integral model: in particular, Theorem B implies that the $\E_{\infty}$-ring $\bS^X$ of sphere-valued cochains on $X$ (i.e., its Spanier-Whitehead dual) is what we will call \emph{Frobenius fixed}, meaning that for each prime $p$, its $p$-completion is $p$-Frobenius fixed.  These Frobenius fixed $\E_{\infty}$-rings are our integral model for spaces:

\begin{thmCrough}[Integral Model]
The data of a simply connected finite complex $X$ is captured completely by its Spanier-Whitehead dual $\bS^X$ as a Frobenius fixed $\E_{\infty}$-ring.
\end{thmCrough}

We now proceed to a more detailed outline of the paper.

\subsection{Integral models for spaces and the Frobenius}

Let $p$ be a prime and $X$ be a simply connected $p$-complete space of finite type.  Mandell's theorem (Theorem \ref{thm:Mandell}) asserts that the functor $X\mapsto C^*(X;\Fpbar)$ is fully faithful, so $X$ can be recovered as the mapping space $$X\simeq \CAlg_{\overline{\F}_p}(C^*(X,\overline{\F}_p), \Fpbar).$$  On the other hand, the assignment $X\mapsto C^*(X, \F_p)$ fails to be fully faithful: one has that 
\begin{align*}
\CAlg_{\F_p}(C^*(X,\F_p), \F_p) &\simeq \CAlg_{\F_p}\big(C^*(X,\F_p), \Fpbar\big)^{h\Z} \\
&\simeq \CAlg_{\Fpbar}(C^*(X,\Fpbar) , \Fpbar)^{h\Z} \simeq X^{h\Z} \simeq \mathcal{L}X
\end{align*}
with $\Z$ acting on $\Fpbar$ by Frobenius \cite{Mandell}.   If this functor to $\CAlg_{\F_p}$ were fully faithful, the result of this calculation would be $X$; we see that the extra free loop space has to do with a failure to account for the Frobenius.  In analogy to classical algebra, one might hope that any $\E_\infty$-algebra over $\F_p$ has a natural Frobenius map, and that the $S^1$-action on $\mathcal{L}X = \Hom(S^1, X)$ arises intrinsically on the left-hand side from an $S^1$-action on the $\infty$-category $\CAlg_{\F_p}$ whose monodromy on any object is Frobenius.  Then, taking fixed points for $S^1$ would ``undo" the free loop space that appears in the above calculation.  

In fact, there is a candidate for this Frobenius map, first defined by Nikolaus-Scholze \cite{NS}.  For any $\E_{\infty}$-ring $A$ and prime $p$, they constructed a natural map of $\E_\infty$-rings $\varphi_A: A\to A^{tC_p}$ which we will call the \emph{$\E_{\infty}$-Frobenius} (or simply the \emph{Frobenius}).   This Frobenius map $\varphi_A$ is \emph{not} generally an endomorphism of $A$: it takes values in the $C_p$-Tate cohomology of $A$ (taken with trivial action).  Thus, $\varphi_A$ can be regarded as an $\E_{\infty}$ analog of the $p$th power map $R\to R/p$ for an ordinary ring $R$ (cf. Example \ref{exm:frobdiscretering}).  In particular, the $\E_{\infty}$-Frobenius is not an endomorphism for the $\E_{\infty}$-ring $\F_p$ because $\F_p^{tC_p}$ is not equivalent to $\F_p$, so the conjectural picture above is not correct as stated.  However, the main theme of this paper is that one \emph{can} realize this picture by working not with $\F_p$-algebras, but with algebras over the $p$-complete sphere.

\begin{exm}\label{exm:introlin}
Any $\E_{\infty}$-ring spectrum $A$ admits another canonical ring map $\can : A \to A^{tC_p}$ by the composite  $$A\to A^{hC_p} \to A^{tC_p}$$ of restriction to group cohomology of $C_p$ followed by projection to Tate cohomology.  
When $A=\pS$, the $p$-complete sphere, the theorems of Lin \cite{Lin} ($p=2$) and Gunawardena \cite{Gunawardena} ($p$ odd) assert that this canonical map is an equivalence.  Thus, the Frobenius \emph{can} be regarded as an endomorphism of $\pS$.  
\end{exm}

Motivated by this example, we extract a full subcategory $\CAlg^F_p$ of $p$-complete $\E_{\infty}$-rings which we call \emph{$F_p$-stable}: these are characterized by the requirement that the natural map $\can: A \to A^{tC_p}$, as well as certain generalized versions of it, $\can^V:  A\to A^{\tau V}$ for elementary abelian $p$-groups $V$, is an equivalence (cf. Definition \ref{defn:frobstable}).  Most relevantly, we will see that any $\E_{\infty}$-ring which is finite over the $p$-complete sphere is $F_p$-stable by the Segal conjecture.

By design, we may think of the Frobenius as a natural \emph{endomorphism} of any $F_p$-stable ring $A$ (given by $\can^{-1}\circ \varphi$).  This determines a functor 
$$\Phi: B\Z_{\geq 0} \to \Fun(\CAlg^F_p, \CAlg^F_p)$$ 
sending the unique object of $B\Z_{\geq 0}$ to the identity functor and the morphism $1\in \Z_{\geq 0}$ to the Frobenius endomorphism, considered as a natural transformation $\mathrm{id}\to \mathrm{id}$. 

\begin{qst*}
Does $\Phi$ lift to a monoidal functor and thus define an action of $B\Z_{\geq 0}$ on the $\infty$-category $\CAlg^F_p$?  
\end{qst*}

\begin{rmk}
This question is due to Jacob Lurie and Thomas Nikolaus, who formulated it in the dual setting of $p$-complete bounded below $\E_{\infty}$-coalgebras.  We address this closely related case in upcoming work \cite{AYCoalg}. 
\end{rmk}

Our first main theorem answers this question to the affirmative:

\begin{thmA}[Frobenius Action]
The $\infty$-category $\CAlg^F_p$ of $F_p$-stable $\E_{\infty}$-rings admits an action of $B\Z_{\geq 0}$ for which $1\in \Z_{\geq 0}$ acts by the Frobenius.  
\end{thmA}

We will elaborate on Theorem A and its proof starting in \S \ref{subsect:introproof} below.  For now, we return to the application in mind with a straightforward corollary of Theorem A.  

\begin{thmAcirc}
Let $\CAlgperfp \subset \CAlg^F_p$ denote the full subcategory of $F_p$-stable $\E_{\infty}$-rings for which the Frobenius acts invertibly.  Then the $\infty$-category $\CAlgperfp$ admits an action of $S^1$ whose monodromy induces the Frobenius \emph{automorphism} on each object.
\end{thmAcirc}

We refer to such $\E_{\infty}$-rings $A \in \CAlgperfp$ as \emph{$p$-perfect}.

\begin{exm}
The $p$-complete sphere $\pS$ is $p$-perfect because the Frobenius is a ring map and $\pS$ is the initial $p$-complete $\E_{\infty}$-ring.  Since $\CAlgperfp \subset \CAlg$ is closed under finite limits, it follows that for any finite space $X$, the $\E_{\infty}$-algebra of cochains $(\pS)^X$ is $p$-perfect.  \end{exm}

In fact, the Frobenius is better than just an equivalence for the $p$-complete sphere -- since $\pS$ is initial among $p$-complete $\E_\infty$-rings, the Frobenius is the \emph{identity} (analogous to how the classical Frobenius is the identity on $\F_p$).  One can capture this idea as follows:

\begin{defn}
Let $\CAlgffp := (\CAlgperfp)^{hS^1}$ denote the $\infty$-category of homotopy fixed points under the $S^1$-action of Theorem $\text{A}^{\circ}$. 
We will refer to objects of $\CAlgffp$ as \emph{$p$-Frobenius fixed} $\E_{\infty}$-rings, and we will sometimes refer to a lift $A_{\varphi=1} \in \CAlgffp$ of a $p$-perfect $\E_{\infty}$-ring $A$ as providing a $p$-Frobenius fixed structure on $A$.
\end{defn}

\begin{exm}\label{exm:pSff}
Since $\pS$ is initial in $\CAlgperfp$, it admits a canonical lift to $\CAlgffp$ which we denote by $(\pS)_{\varphi=1}$.  
\end{exm}

Since $\CAlgffp$ admits finite limits which are computed on the underlying $\E_{\infty}$-ring, we see that for any finite complex $X$, the cochain algebra $(\pS)^X$ admits a canonical lift $(\pS)_{\varphi=1}^X \in \CAlgffp$ to a $p$-Frobenius fixed $\E_{\infty}$-ring.  This is essentially our $p$-adic model for spaces; to state it more precisely, we will need the following finiteness condition:

\begin{notation}
We say a simply connected space $X$ is \emph{$p$-complete finite} if $X$ is $p$-complete and $\bigoplus_{i\in \Z} H^i(X;\F_p)$ is a finite abelian group.  We let $\Spacesfnp$ denote the full subcategory of spaces which are simply connected and $p$-complete finite.
\end{notation}

\begin{rmk}
By a standard Hurewicz theorem argument, a simply connected $p$-complete space $X$ is $p$-complete finite if and only if $X$ is built out of a finite number of $p$-complete spheres via a (finite) sequence of cell attachments.  We remark that a (simply connected) $p$-complete finite space need not be the $p$-completion of a finite complex -- we thank Robert Burklund for clarifying this point for us and refer the reader to \cite[Counterexample 3.4]{BelfiWilkerson} for a discussion.   
\end{rmk}

 The following theorem gives a model for $\Spacesfnp$ in terms of $\E_{\infty}$-algebras and the Frobenius action of Theorem A.  

\begin{thmB}[$p$-adic Model]
The cochain functor $(\Spacesfnp)^{\op} \to \CAlg_{\pS}$ given by $X \mapsto (\pS)^X$ lifts to a fully faithful functor 
$$(\pS)^{(-)}_{\varphi=1}: (\Spacesfnp)^{\op} \to \CAlgffp.$$
The essential image of the functor $(\pS)^{(-)}_{\varphi=1}$ consists of those $p$-Frobenius fixed $\E_{\infty}$-rings whose underlying $\E_{\infty}$-ring $A$ satisfies the following two conditions:
\begin{enumerate}
\item $A$ is finite as a module over $\pS$.  
\item The unit map $\F_p \to \F_p\otimes A$ induces an isomorphism on $\pi_k$ for $k\geq -1$.  
\end{enumerate}
\end{thmB}

Informally, Theorem B asserts that for any simply connected $p$-complete finite space $X$, the algebra of cochains $(\pS)^X$ admits a canonical $p$-Frobenius fixed structure, and that $X$ is captured completely by the $\E_{\infty}$-ring $(\pS)^X$ together with this $p$-Frobenius fixed structure.  

\begin{warn}\label{warn:Fptriv}
A $p$-Frobenius fixed structure is \emph{not} simply a trivialization of the $\E_{\infty}$-Frobenius map.
To illustrate this, note first that the $S^1$ action of Theorem $\text{A}^{\circ}$ allows one to construct a fibration $q: (\CAlg^{\mathrm{perf},\simeq}_p)_{hS^1} \to BS^1$ of spaces (the \emph{action groupoid}). This begets the outer pullback square in the diagram:
\begin{equation*}
\begin{tikzcd}
\CAlg^{\mathrm{perf},\simeq}_p \arrow[d] \arrow[rrr] & & & (\CAlg^{\mathrm{perf},\simeq}_p)_{hS^1}\arrow[d,"q"] \\
* \arrow[r, hook] &\mathbb{C}P^1 \arrow[rru, dashed]\arrow[r,hook]& \cdots \arrow[r,hook] &BS^1\simeq \mathbb{C}P^{\infty}\\
\end{tikzcd}
\end{equation*}

The data of a $p$-Frobenius fixed algebra is the data of a section of $q$.  One can view such a section as being built cell-by-cell along the cell decomposition of $\mathbb{C}P^{\infty}$.  A section of $q$ over the point is the data of a $p$-perfect $\E_{\infty}$-ring $A$.  Extending this section to a section over $\mathbb{C}P^1$ \emph{is} just the data of a trivialization of the Frobenius map on $A$.  However, to promote this to a $p$-Frobenius fixed structure on $A$, one needs to further extend this section over all of $\mathbb{C}P^{\infty}$; this requires one to specify an additional homotopy for each additional cell of $\mathbb{C} P^{\infty}$.  
\end{warn}
As promised, this model of $p$-complete spaces is more compatible with the rational model of Theorem \ref{thm:RHT}; unlike $\CAlg_{\Fpbar}$, the $\infty$-category $\CAlgffp$ of $p$-Frobenius fixed algebras admits an obvious functor to $\CAlg_{\Q_p}$ by forgetting the $p$-Frobenius fixed structure and extending scalars along the map $\pS \to \Q_p$.  This allows our $p$-adic model to be compared to Sullivan's rational model for spaces.  We now describe the resulting integral model for spaces.

\begin{defn}\label{defn:assembly}
We will say that an $\E_{\infty}$-ring $A$ is \emph{perfect} if for each prime $p$, its $p$-completion $A^{\wedge}_p$ is $p$-perfect.  Let $\CAlgperf \subset \CAlg$ be the full subcategory of perfect $\E_{\infty}$-rings.  Additionally, let the $\infty$-category $\CAlg^{\varphi = 1}$ of \emph{Frobenius fixed} $\E_{\infty}$-rings be defined by the pullback square
\begin{equation*}
\begin{tikzcd}
\CAlg^{\varphi = 1} \arrow[d]\arrow[rr] \arrow[drr, phantom, "\ulcorner", very near start] & & \displaystyle{\prod_{p} \CAlg^{\varphi = 1}_p }\arrow[d]\\
\CAlgperf \arrow[rr, "\prod (-)^{\wedge}_p"] & &\displaystyle{ \prod_p \CAlgperfp } .
\end{tikzcd}
\end{equation*}
Informally, a Frobenius fixed $\E_{\infty}$-ring is a perfect $\E_{\infty}$-ring $A$ equipped with a $p$-Frobenius fixed structure on its $p$-completion $A^{\wedge}_p$ for each prime $p$.  
\end{defn}

By \Cref{exm:pSff}, the sphere spectrum $\bS$ admits a canonical lift to a Frobenius fixed $\E_{\infty}$-ring, which we denote by $\bS_{\varphi=1}\in \CAlgff$.  
Let $\Spacesf$ denote the full subcategory of spaces which are homotopy equivalent to a finite complex.  Then, since finite limits exist in $\CAlgff$ and are computed on the underlying $\E_{\infty}$-ring, $\bS_{\varphi=1}$ determines an essentially unique functor 
\[
\bS^{(-)}_{\varphi = 1} : (\Spacesf)^{\op} \to \CAlg^{\varphi = 1}
\]
which preserves finite limits and whose value on the point is $\bS_{\varphi=1}$.  Our final theorem, which is a corollary of Theorem B and Theorem \ref{thm:RHT}, is that this functor is an integral model for simply connected finite complexes:

\begin{thmC}[Integral Model]
The restriction of the functor $\bS^{(-)}_{\varphi=1}$ to the full subcategory $\Spacesfn \subset \Spacesf$ of simply connected finite complexes is fully faithful.  
\end{thmC}

In other words, for a simply connected finite complex $X$, the algebra $\bS^X$ of spherical cochains can be canonically promoted to a Frobenius fixed $\E_{\infty}$-ring $\bS^X_{\varphi=1}$, and the data of $X$ is completely captured by $\bS^X_{\varphi=1}.$   In contrast to the $p$-adic case, we have not succeeded in determining the essential image of $\bS^{(-)}_{\varphi=1}$ (cf. \Cref{qst:Zessim}).  

\begin{rmk}
Unlike Theorem \ref{thm:Mandell}, this spherical model only works for \emph{finite} simply connected spaces (see also \S \ref{subsect:extensions}).  Of course, since any simply connected space is a filtered colimit of finite ones,  one could model the whole $\infty$-category of simply connected spaces by pro-objects in Frobenius fixed $\E_{\infty}$-rings.  
\end{rmk}
  

\subsection{The Frobenius action}\label{subsect:introproof}
We return to the main technical focus of this paper, which is the development of the Frobenius action of Theorem A.  

\begin{rmk}\label{rmk:classical}
The analog of Theorem A in classical algebra is easy.  Letting $\CAlg^{\heartsuit}_{\F_p}$ denote the category of discrete commutative $\F_p$-algebras, we can define a functor $$\Phi^{\heartsuit}:B\Z_{\geq 0} \to \Fun(\CAlg^{\heartsuit}_{\F_p},\CAlg^{\heartsuit}_{\F_p})$$ by sending $1\in \Z_{\geq 0}$ to the Frobenius endomorphism for discrete $\F_p$-algebras.  This functor is monoidal if and only the corresponding map on centers (i.e., endomorphisms of the unit object) $$\Z_{\geq 0} \to \End (\mathrm{id}_{\CAlg^{\heartsuit}_{\F_p}})$$ is a map of commutative monoids.  This is automatic because the category of commutative monoids is a full subcategory of the category of monoids; in algebra, commutativity is a property.

This argument breaks down in homotopy theory.  The space $\End(\mathrm{id}_{\CAlg^F_p})$ is naturally $\E_2$-monoidal because it arises as endomorphisms of the unit object in a monoidal $\infty$-category.  The Frobenius $\varphi \in \End(\mathrm{id}_{\CAlg^F_p})$ determines a map of \emph{$\E_1$-spaces}
 $$\Z_{\geq 0} \to \End(\mathrm{id}_{\CAlg^F_p})$$ because $\Z_{\geq 0}$ is free as an $\E_1$-monoid, and one needs to show that it can be promoted to a map of $\E_2$-spaces.  This is no longer automatic and the content of Theorem A is that this is possible.  
 
Theorem A has the following concrete consequence: for any $\psi_1, \psi_2 \in \End(\mathrm{id}_{\CAlg^F_p}),$ the $\E_2$-structure provides a homotopy $\psi_1 \psi_2 \sim \psi_2 \psi_1 $.  This induces an action of the \emph{braid group on $n$ strands} on the natural transformation $\varphi^n : \mathrm{id}_{\CAlg^F_p}\to \mathrm{id}_{\CAlg^F_p}.$  Theorem A asserts that this action is \emph{trivial}, which can be thought of heuristically as saying that the Frobenius ``commutes with itself."
\end{rmk}

The majority of the paper is dedicated to proving Theorem A.  The proof proceeds in two steps,  which we believe may be of independent interest and which we outline in \S \ref{subsect:introglobal} and \S \ref{subsect:intropartK}, respectively.  Roughly, the first step is to use equivariant stable homotopy theory to produce a ``lax" version of the Frobenius action which exists for \emph{all} $\E_{\infty}$-rings, without $p$-completing or restricting to $\CAlg^F_p$ (Theorem $\text{A}^{\natural}$).  The second step is to use a ``pre-group-completed" variant of algebraic $K$-theory which we call \emph{partial $K$-theory} to descend this to an action of $B\Z_{\geq 0}$ on the full subcategory $\CAlg^F_p \subset \CAlg$.  

\subsection{Frobenius for genuine globally equivariant rings}\label{subsect:introglobal}
The Frobenius action of Theorem A arises by observing that one can associate a Frobenius operator to every finite group $G$ and then carefully studying the interaction between these operators.  We describe this using equivariant stable homotopy theory.  In \S \ref{sect:algact}, we construct an $\infty$-category $\CAlg^{\Glo}$ of \emph{global algebras} (Definition \ref{defn:gloalg}), which are roughly $\E_{\infty}$-rings $A$ together with the additional structure, for each finite group $G$, of a genuine $G$-equivariant multiplication map $N^G A \to \triv^G A$ from the norm of $A$ to the ring $A$ with trivial $G$-action.  The genuine multiplications on a global algebra can be thought of as lifting various composites of the $\E_{\infty}$-Frobenius to endomorphisms of $A$ (cf. Example \ref{exm:Cpequivariantmult}).  As such, global algebras admit natural Frobenius \emph{endomorphisms} for each $G$, and we denote the associated natural transformation by $\varphi^G: \id \to \id$.

The endomorphisms $\varphi^G$ exhibit interesting functorialities in the group $G$: we observe that the various Frobenius operators assemble into an action of a symmetric monoidal 1-category $\mathcal{Q}$ (roughly, Quillen's $Q$-construction on finite groups, cf. Definition~\ref{defn:Q}), whose objects are finite groups and whose morphisms from $H$ to $G$ are isomorphism classes of spans $(H \twoheadleftarrow K \hookrightarrow G)$.   
Namely, in Theorem \ref{thm:Qactgloalg}, we show that $\mathcal{Q}$ acts on the $\infty$-category $\CAlg^{\Glo}$ in a way such that each object in $\mathcal{Q}$ acts trivially and the span $(* \twoheadleftarrow * \hookrightarrow G)$ is sent to $\varphi^G$.  

The action of $\mathcal{Q}$ on the $\infty$-category of global algebras does not directly induce an action on the $\infty$-category of $\E_{\infty}$-rings.  A global algebra is, in general, more data than an $\E_{\infty}$-ring.  However, we explain that the $\infty$-categories $\CAlg$ and $\CAlg^{\Glo}$ are related via the notion of a \emph{Borel global algebra} (Definition \ref{defn:borgloalg}, Theorem \ref{thm:globalalgs}).  We then deduce that the $\E_{\infty}$-Frobenius $A\to A^{tC_p}$ extends to an \emph{oplax} action of $\mathcal{Q}$ on the $\infty$-category $\CAlg$ of $\E_{\infty}$-ring spectra; this oplax action is unwound explicitly in \S \ref{sub:genfrob} and is the content of Theorem \ref{thm:mainalgaction}, stated here in rough form:

 
\begin{thmAnat}[Integral Frobenius Action]
There is an oplax monoidal functor $$\mathcal{Q} \to \Fun(\CAlg, \CAlg)$$ extending the $\E_{\infty}$-Frobenius in the sense that $C_p\in \mathcal{Q}$ is sent to $(-)^{tC_p}$ and the span $(* \twoheadleftarrow * \hookrightarrow C_p)$ acts by the Frobenius $\id \to (-)^{tC_p}$.  
\end{thmAnat}

We emphasize that Theorem $\text{A}^{\natural}$ applies to \emph{all} $\E_{\infty}$-rings and finite groups, and in particular does not require $p$-completion or passing to a subcategory of rings.  It can be seen as describing the relationship between various ``stable power operations" on $\E_{\infty}$-rings (cf. Remark \ref{rmk:adem}).

\subsection{Partial Algebraic $K$-theory}\label{subsect:intropartK}

Recall that the \emph{$Q$-construction} is a device introduced by Quillen in order to define higher algebraic $K$-theory.  In this light, the Frobenius acting through $\mathcal{Q}$ can be seen as articulating that the action of Frobenius is ``$K$-theoretic" or ``additive in exact sequences" in the sense that for a short exact sequence $G' \to G \to G''$ of groups, the Frobenius for $G$ is equivalent to the Frobenius for $G'$ followed by the Frobenius for $G''$.  Theorem A follows from Theorem $\text{A}^{\natural}$ by making this idea precise.  

Fixing a prime $p$, we restrict the \emph{oplax} action of Theorem $\text{A}^{\natural}$ to the full subcategories $\CAlg^F_p\subset \CAlg$ and $\QVect \subset \mathcal{Q}$ (spanned by the elementary abelian $p$-groups) to obtain a (\emph{strong}) action of $\QVect$ on $\CAlg^F_p$.  Recall the following computation of Quillen regarding $\QVect$:

\begin{thm}[Quillen \cite{Quillen}]\label{thm:introquillen}
The natural map $$\Omega |\QVect| =: K(\mathbb{F}_p) \to \pi_0 K(\mathbb{F}_p) \simeq \Z$$ induces an isomorphism in $\F_p$-homology.  
\end{thm}

Motivated by this computation, one might hope that our action of $\QVect$ extends to an action of the underlying $\E_{\infty}$-space $|\QVect |$.  Since $|\QVect| \simeq BK(\F_p)$ is $p$-adically equivalent to $B\Z \simeq S^1$, the resulting $S^1$ action would restrict to the desired action of $B\Z_{\geq 0}$.  Unfortunately, this is impossible because $S^1$ is group complete, and the Frobenius on $\CAlg^F_p$ does not necessarily act by equivalences.  We overcome this difficulty in Section \ref{sect:partK} by introducing a non-group-complete variant of algebraic $K$-theory which we call \emph{partial $K$-theory} (Definition \ref{defn:Kpart}). This construction takes a Waldhausen ($\infty$-)category
$\C$ and produces an $\E_{\infty}$-space $K^{\mathrm{part}}(\C)$ such that:
\begin{itemize}
\item There is a canonical equivalence of $\E_{\infty}$-spaces $K^{\mathrm{part}}(\C)^{\mathrm{gp}} \simeq K(\C)$.
\item The monoid $\pi_0 (K^{\mathrm{part}}(\C))$ is freely generated by the objects of $\mathcal{C}$ subject to the relation  $[A]+[C] = [B]$ for every short exact sequence $0\to A\to B\to C\to 0.$
\end{itemize} 
We show that many statements in algebraic $K$-theory have analogs in partial $K$-theory.  For instance, while the definition of $K^{\mathrm{part}}(\C)$ uses Waldhausen's $S_{\bullet}$-construction, we give an alternate construction in the case of an exact category $\C$ involving Quillen's $Q$-construction and show that the two definitions coincide (Theorem \ref{thm:QequalsS}).  

It follows from this latter construction of partial $K$-theory that the action of $\QVect$ on $\CAlg^F_p$ descends to an action of the monoidal $\infty$-category $BK^{\mathrm{part}}(\F_p)$ on $\CAlg^F_p$.  The $B\Z_{\geq 0}$ action of Theorem A arises by combining this action of $BK^{\mathrm{part}}(\F_p)$ with the following partial $K$-theory analog of Theorem \ref{thm:introquillen}:

\begin{thmn}[\ref{thm:KpartFp}]
The natural map $$K^{\mathrm{part}}(\F_p) \to \pi_0 K^{\mathrm{part}}(\F_p) \simeq \Z_{\geq 0}$$ induces an isomorphism in $\F_p$-homology.  
\end{thmn}

In addition to proving Theorem A, our methods also imply the following assertions about global algebras, which may be of independent interest:
\begin{thmn}[\ref{thm:gloalgpartK}]
\leavevmode
\begin{enumerate}
\item The $\infty$-category $\CAlg^{\Glo}$ of global algebras admits an action of the monoidal $\infty$-category $BK^{\mathrm{part}}(\Z)$ for which an abelian group $G$ acts by the Frobenius $\varphi_G: \id \to \id$.  
\item The full subcategory $\CAlg^{\Glo}_p\subset \CAlg^{\Glo}$ of global algebras with $p$-complete underlying $\E_{\infty}$-ring admits an action of $B\Z_{\geq 0}$ for which $1\in \Z_{\geq 0}$ acts by the Frobenius (for the group $C_p$).  
\end{enumerate}
\end{thmn}
\begin{rmk}
The notion of a global algebra is not new; they have appeared in the literature as the normed algebras of \cite{BachHoy} and are closely related to Schwede's ultra-commutative monoids \cite{Schwede}.  
\end{rmk}

\subsection{Relationship to other work}

There has been much previous work on understanding spaces via chain and cochain functors.  The rational case is due to Quillen \cite{QuillenRatl} and Sullivan \cite{SullivanQ}.  Recent work of Heuts \cite{Gijsvn} and Behrens-Rezk \cite{BehrensRezk} extends this rational picture to the setting of $v_n$-periodic spaces.  In another vein, Goerss \cite{Goerss} gives a model for $p$-local spaces in terms of simplicial coalgebras, and Kriz \cite{Kriz} for $p$-complete spaces in terms of cosimplicial algebras.  

Our work is closest in spirit to work of Mandell \cite{Mandell} and Dwyer-Hopkins (unpublished, cf. \cite[Appendix C, Theorem C.1]{Mandell}), who study $p$-complete spaces via $\E_{\infty}$-algebras over $\overline{\F}_p$.  In \cite{MandellZ}, Mandell further explains that the functor of \emph{integral} cochains from (nice enough) spaces to $\E_{\infty}$-algebras over $\Z$ is faithful but not full.  This paper essentially realizes a program due to Thomas Nikolaus, in the setting of coalgebras, for making a fully faithful model.  The idea for this starts with unpublished work of Mandell which lifts $p$-adic homotopy theory from $\overline{\F}_p$ to the spherical Witt vectors of $\overline{\F}_p$.  Then, it is an insight of Nikolaus that having a homotopy coherent Frobenius action allows one to descend to a statement over the sphere.  Our contribution is to produce this Frobenius action (in the dual setting of algebras).

We also acknowledge that Nikolaus has concurrent work in the dual setting of $p$-complete perfect $\E_{\infty}$-coalgebras.  In forthcoming work \cite{AYCoalg}, we show how our methods adapt to the setting of coalgebras to produce a $B\Z_{\geq 0}$-action on all $p$-complete bounded below $\E_\infty$-coalgebras.  We will also (at least conjecturally) describe the Goodwillie filtration in that setting, relating our model to the Tate coalgebras of Heuts \cite[\S 6.4]{Heuts}.

\subsection{Organization}
In \S \ref{sect:ESHT}, we record the requisite preliminaries on equivariant stable homotopy theory. Then in \S \ref{sect:algact}, we introduce the $\E_\infty$-Frobenius and use the results of \S \ref{sect:ESHT} to prove Theorem $\text{A}^{\natural}$ (as Theorem \ref{thm:mainalgaction}).  

In \S \ref{sect:partK}, we introduce and develop basic results about partial $K$-theory.  Then in \S \ref{sect:KpartFp}, we present our results on the partial $K$-theory of $\F_p$ (Theorem \ref{thm:KpartFp}).  

In \S \ref{sect:frobstable}, we combine the previous results to prove Theorem A.  Then in \S \ref{sect:padic}, we apply Theorem A to prove Theorems B and C about $\E_{\infty}$-algebra models for spaces.  

Finally, in the appendix, we settle certain technical constructions needed in \S \ref{sect:algact}.


\subsection{Acknowledgments}
The author would like to thank Clark Barwick for a crucial suggestion, Robert Burklund, Jeremy Hahn, Gijs Heuts, Akhil Mathew, Denis Nardin, and Arpon Raksit for numerous enlightening conversations, and Robert Burklund, Elden Elmanto and Jeremy Hahn for comments on a draft; in addition, he would like to thank Maria Yakerson, Milton Lin, and Denis Nardin for pointing out mistakes or omissions in a previous version.   He is also grateful to the anonymous referee for numerous clarifying comments, corrections, and simplifications.
He would like to extend a very special thanks to Thomas Nikolaus, whose work made this project possible and who generously shared his ideas and encouraged the author's work.  Most importantly, the author would like to thank his PhD advisor, Jacob Lurie, for his consistent encouragement and countless insights; many of the ideas in this paper originate from his suggestions.  The author was supported by the NSF under Grant DGE-1122374.

\subsection{Notation and conventions}
We use the following notations and conventions throughout:
\begin{itemize}
\item For an abelian group $A$, we denote the corresponding Eilenberg-MacLane spectrum by $A$. 
\item We let $\Spaces$ denote the $\infty$-category of spaces and $\Sp$ denote the $\infty$-category of spectra.
\item We will call a space $p$-complete if it is Bousfield local with respect to $\F_p$-homology, and we will call a spectrum $p$-complete if it is Bousfield local with respect to the Moore spectrum $\bS/p$ \cite{BousSpaces, BousSpectra}.  
\end{itemize}

\section{Equivariant stable homotopy theory}\label{sect:ESHT}

In this section, we briefly review notions in equivariant stable homotopy theory that feature heavily in this paper.  We will use a variant of the framework of \cite[\S 9]{BachHoy}, but we refer the reader to \cite{LMS} for a classical treatment of the subject and \cite{MNN} for another helpful account.  The material in this section is not new, with the possible exception of Theorem \ref{thm:equivfibinf}, which the author thanks Jacob Lurie for suggesting. 

In \S \ref{subsect:globalequiv}, we will set notation by giving examples and motivation for the formal constructions related to global equivariance that follow in \S \ref{subsect:formalglo}.  Then in \S \ref{subsect:bortate}, we review the notion of Borel equivariant spectra and its relationship to Tate constructions.  

\subsection{Global equivariance}\label{subsect:globalequiv}

For a finite group $G$, let $\Sp^G$ denote the $\infty$-category of genuine $G$-spectra. This construction is contravariantly functorial in the group $G$, so for any group homomorphism $f:H\to K$, one has a functor $f^*:\Sp^K \to \Sp^H$:

\begin{enumerate}
\item For an injection $f: H \hookrightarrow G$, $f^*: \Sp^G \to \Sp^H$ can be thought of as restricting to the subgroup $H$, and so we will sometimes denote this functor by $\res_H^G$.  
\item For a surjection $f: G \twoheadrightarrow K$, $f^*: \Sp^K \to \Sp^G$ can be thought of as giving a spectrum the trivial action on $\mathrm{ker}(f)$.    As such, we will sometimes denote this functor by $\triv_K^G$.  
\end{enumerate}

In general, the functor $f^*$ can be computed by factoring $f:H\to K$ as $H\xtwoheadrightarrow{f'} G \xhookrightarrow{f''}K$ and setting $f^* = f'^* \circ f''^*$.  

One also has a covariant functoriality in the group $G$ by multiplicative transfer:

\begin{enumerate}\setcounter{enumi}{2}
\item For an injection $g:H\hookrightarrow G$,  one has a \emph{norm} functor $g_{\otimes} : \Sp^H \to \Sp^G$, due to \cite{HHR}.  We will sometimes denote this functor by $N_H^G$.  
\item For a surjection $g:G \twoheadrightarrow K$, one has a \emph{geometric fixed point} functor $g_{\otimes}: \Sp^G \to \Sp^K$, which we sometimes denote by $\Phi^{\mathrm{ker}(g)}$.  
\end{enumerate}

Analogously to the restriction functors, these are ``part of the same functoriality"; one can functorially associate a multiplicative transfer map $g_{\otimes}$ to any group homomorphism $g$ in a way that recovers the norm and geometric fixed point functors in the above cases. 

\begin{exm}
For the composite $e \xrightarrow{g'} G \xrightarrow{g''} e$, we can check that $g''_{\otimes} \circ g'_{\otimes} = \id = \id_{\otimes}$ because the geometric fixed points of the norm is the identity.  
\end{exm}

\begin{var}
The $\infty$-categories $\Sp^G$ and the functors above depend only on the groupoid $BG$ and not the group $G$; in fact, one can make sense of all the above notions with groupoids in place of groups.  Let $\Gpd$ denote the 2-category of groupoids $X$ such that $\pi_0(X)$ and $\pi_1(X)$ are finite.  We will refer to objects of $\Gpd$ simply as groupoids.  

For $X\in \Gpd$, we let:
\begin{itemize}
\item $\Set_X$ denote the category of $X$-sets.
\item $\Fin_X$ denote the category of finite $X$-sets.
\item $\Sp_X$ denote the $\infty$-category of genuine $X$-spectra (\cite[\S 9]{BachHoy}).  
\end{itemize}
\begin{exm}
When $X = \coprod BG_i$, the $\infty$-category $\Sp_X$ is equivalent to $\prod \Sp^{G_i}$.  
\end{exm}


For a map $f:X\to Y$ of groupoids, one has functors $f^*: \Sp_Y \to \Sp_X$ and $f_{\otimes}: \Sp_X \to \Sp_Y$ which coincide with the similarly notated functors above when $f$ is induced by a group homomorphism.  These two opposing functorialities interact as follows: let $\Span(\Gpd)$ denote the $\infty$-category of spans of groupoids \cite[\S 5]{BarK}.  Then, by \cite[\S 9]{BachHoy}, there is a functor $$\Span(\Gpd) \to \Cat_{\infty}$$ which sends a groupoid $X$ to the $\infty$-category $\Sp_X$ and sends a span $X\xleftarrow{f} M\xrightarrow{g} Y$ to the functor $g_{\otimes} f^*: \Sp_X \to \Sp_Y$.  
\end{var}



\begin{exm}
When $g= \coprod g_i: \coprod X_i \to Y$ is a map of groupoids, $g_{\otimes}$ is given by the tensor product of the $(g_i)_{\otimes}$.  
\end{exm}

We use a variant of this framework which arises from modifying the $\infty$-category $\Span(\Gpd)$:
\begin{var}\label{var:glop}
In Definition \ref{defn:glo}, we construct a $(2,1)$-category $\Glo^+$ which can be described informally as follows:
\begin{itemize}
\item The objects of $\Glo^+$ are groupoids $X\in \Gpd$.  
\item For two objects $X,Y \in \Glo^+$, the morphism groupoid $\Hom(X,Y)$ is the category of finite coverings $M \to X\times Y$ and isomorphisms of coverings.  We will think of a morphism as a span $X\leftarrow M \rightarrow Y$.  
\item For two composable morphisms $X\leftarrow M \rightarrow Y$ and $Y\leftarrow N \rightarrow Z$, the composite is the span $X\leftarrow T \rightarrow  Z$ such that the map $T \to X\times Z$ fits into the essentially unique factorization $M\times_{Y} N \to T \to X \times Z$  of $M\times_Y N\to X\times Z$ into a map with connected fibers followed by a finite cover.
\end{itemize}
In contrast, maps from $X$ to $Y$ in $\Span(\Gpd)$ are given by arbitrary groupoids $M \to X\times Y$.  We will see (\Cref{prop:spangpd} and \Cref{cnv:spangpd}) that there is a functor $\Span(\Gpd) \to \Glo^+$ which is ``the identity on objects" and can be thought of as sending the morphism $M$ above to its $0$-truncation relative to $X\times Y$.  
\end{var}

\begin{rmk}\label{rmk:fininterp}
For groupoids $BH,BG, BK \in \Glop$, the morphism groupoid $\Glop(BH,BG)$ can be identified with the groupoid of $(G,H)$-bisets and isomorphisms, and the composition $$\Glop(BG,BK)\times \Glop(BH,BG)\to \Glop(BH,BK)$$ sends a $(K,G)$-biset $T$ and a $(G,H)$-biset $U$ to the $(K,H)$-biset $T\times_{G} U$.  
\end{rmk}

As promised, $\Glop$ parametrizes the various equivariant stable homotopy categories and associated functorialities:

\begin{thm}\label{thm:equivfibinf}
There is a functor $$\Psi^+_{\mathrm{fun}}:\Glo^+ \to \Cat_{\infty}$$ sending a groupoid $X$ to the $\infty$-category $\Sp_X$ of genuine $X$-spectra, and sending a span $X\xleftarrow{f} M \xrightarrow{g}Y$ to the functor $g_{\otimes}f^*: \Sp_X \to \Sp_Y$.  
\end{thm}

We will prove \Cref{thm:equivfibinf} in the following section as \Cref{thm:mainfib}.  

\begin{exm}\label{exm:spanfininside}
The full subcategory of $\Glop$ spanned by discrete groupoids is equivalent to the $(2,1)$-category $\Span(\Fin)$ of spans of finite sets.  The restricted functor $\Span(\Fin) \to \Cat_{\infty}$ sends a finite set $T$ to $\Fun(T,\Sp)$ and exhibits $\Sp$ as a symmetric monoidal $\infty$-category.  
\end{exm}

\begin{exm}
Inside $\Glo^+$, we have the composite of morphisms $$ (BC_p \leftarrow * \rightarrow *)\circ (* \leftarrow * \rightarrow BC_p) = (* \leftarrow C_p \rightarrow *).$$  Applying $\Psi^+_{\mathrm{fun}}$, we recover the fact that the composite of the norm $N_{e}^{C_p}:\Sp \to \Sp^{C_p}$ and the restriction $\res^{C_p}_{e}:\Sp^{C_p} \to \Sp$ is the $p$-fold tensor product functor $(-)^{\otimes p}: \Sp \to \Sp$.
\end{exm}

\begin{exm}\label{exm:trivgfp}
A more subtle example comes from composing the spans $$(BC_p \leftarrow BC_p \rightarrow *) \circ (*\leftarrow BC_p \rightarrow BC_p).$$  The composite as spans of groupoids would be $(* \leftarrow BC_p \rightarrow *).$  However, by the composition law of $\Glo^+$, this is replaced by $(* \leftarrow * \rightarrow *).$  This expresses the fact that for a spectrum $E$, there is a canonical equivalence $\Phi^{C_p}\triv^{C_p} E \simeq E$.  These equivalences $\Phi^G \triv^G E \simeq E$ are exactly the additional data captured by $\Psi_{\mathrm{fun}}: \Glop \to \Cat_{\infty}$ but not by the corresponding functor from $\Span(\Gpd)$ (see Remark \ref{rmk:mapfromspangpd}).   
\end{exm}

\begin{exm}\label{exm:geonorm}
Let $H\subset G$ be finite groups.  Then the composite of spans $$(BG \leftarrow BG \rightarrow *) \circ (BH \leftarrow BH \rightarrow BG) = (BH \leftarrow BH \rightarrow *)$$  expresses the fact that for $E\in \Sp^H$, there is a canonical equivalence $\Phi^G N_H^G E \simeq \Phi^H E$.
\end{exm}

\begin{exm}\label{exm:geotriv}
Let $G \twoheadrightarrow K$ be a surjection of finite groups.  Then the composite in $\Glop$ $$(BG \leftarrow BG \rightarrow *) \circ (BK \leftarrow BG \rightarrow BG) = (BK \leftarrow BK \rightarrow *)$$ expresses the fact that for $E\in \Sp^K$, there is a canonical equivalence $\Phi^G \triv_K^G E \simeq \Phi^K E.$  
\end{exm}

\subsection{Formal constructions around $\Glo^+$}\label{subsect:formalglo}
We now give a formal definition of $\Glo^+$ and prove Theorem \ref{thm:equivfibinf}  (appearing here as Theorem \ref{thm:mainfib}).  

\begin{defn}\label{defn:glo}
Let the $(2,1)$-category $\Glo^+ \subset \Cat_1$ be the subcategory of the $(2,1)$-category of 1-categories whose objects are categories of the form $\Set_X$ for $X\in \Gpd$ and whose morphisms are functors which preserve limits and filtered colimits.  
\end{defn}

We will now unwind this definition of $\Glop$ to see that it admits the description given in Variant \ref{var:glop}.  The following lemma relates Definition \ref{defn:glo} to Remark \ref{rmk:fininterp}.  

\begin{lem}\label{lem:functorsvssets}
Let $H$ and $G$ be finite groups.  There is a fully faithful functor $$(\Fin^{H \times G^{op}})^{\op} \to \Fun(\Set^{H},\Set^{G})$$ given by sending $T\in \Fin^{H \times G^{\op}}$ to the functor $X\mapsto 
\Hom_H(T, X).$  The essential image is the full subcategory $\Fun^{R,\omega}(\Set^H, \Set^G)\subset \Fun(\Set^H,\Set^G)$ of functors which preserve limits and filtered colimits.    
\end{lem}
\begin{proof}
Let $\Fun^{L}(\Set^G,\Set^H)$ denote the full subcategory of left adjoint functors, which coincides with the full subcategory of colimit preserving functors by the adjoint functor theorem.  Since $\Set_G$ is freely generated under colimits by the free transitive $G$-set, we have the equivalences 
\begin{align*}
\Fun^{L}(\Set^G,\Set^H) &\simeq \Fun(BG^{\op},\Set^H)\\
&\simeq \Fun(BH\times BG^{\op}, \Set) \simeq \Set^{H\times G^{\op}}.
\end{align*}  By taking adjoints, we obtain a fully faithful embedding 
$$\Fun^{R,\omega}(\Set^H,\Set^G) \subset \Fun^L(\Set^G,\Set^H)^{\op} \simeq (\Set^{H\times G^{\op}})^{\op}$$ with essential image exactly the finite $G$-$H$-bisets.  
\end{proof}

\begin{rmk}
By the adjoint functor theorem, the opposite category $(\Glo^+)^{\op}$ can be described as the subcategory of $\Cat_1$ spanned by categories of the form $\Set_X$ for $X\in \Gpd$ and colimit preserving functors which send compact objects to compact objects.  This acquires a symmetric monoidal structure under the tensor product of presentable categories (cf. \cite[\S 4.8]{HA}), and so $\Glo^+$ also inherits a symmetric monoidal structure.  On objects, one has $\Set_X \otimes \Set_Y \simeq \Set_{X\times Y}$.
\end{rmk}

We now relate $\Glop$ to the $\infty$-category $\Span(\Gpd)$.  Recall that the latter has the structure of a symmetric monoidal $\infty$-category under the Cartesian product of groupoids \cite[\S 7]{BarSMF2}.

\begin{prop}\label{prop:spangpd}
There is a symmetric monoidal functor $$\pi: \Span(\Gpd) \to \Glop$$ which sends a groupoid $X$ to the category $\Set_X$ and sends a span $(X\xleftarrow{f} M \xrightarrow{g} Y)$ to the functor $g_{\times} f^*: \Set_X \to \Set_Y,$ where $g_{\times}$ denotes the right adjoint to $g^*$.  
\end{prop}

\begin{proof}
The ($\Set$-valued) Yoneda embedding gives a functor $\Gpd^{\op} \to \Glop$ which takes products of groupoids to tensor products of presentable categories.  The restriction functors admit right adjoints satisfying the necessary Beck-Chevalley conditions, so by \cite[Construction 7.6]{BarSMF2}, we obtain the desired symmetric monoidal functor $\pi: \Span(\Gpd) \to \Glop$.  

\end{proof}

\begin{cnv}\label{cnv:spangpd}
It follows from Lemma \ref{lem:functorsvssets} that for $X,Y\in \Glop$, the groupoid $\Glop(X,Y)$ can be identified with the groupoid of finite coverings of $X\times Y$.  Under this identification, the functor $\pi$ of Proposition \ref{prop:spangpd} sends a span $\tilde{M} \to X\times Y$ to the span $M\to X\times Y$ determined by the canonical factorization $\tilde{M} \to M \to X\times Y$ of $\tilde{M} \to X\times Y$ into a map with connected fiber followed by a finite cover.  This justifies the description given in Variant \ref{var:glop}.  We will, therefore, continue to think of the 2-category $\Glo^+$ as groupoids with certain spans between them unless otherwise specified.  
\end{cnv}

We now arrive at the goal of this section:

\begin{thm}\label{thm:mainfib}
There is a coCartesian fibration $\Psi^+: \GlopSp \to \Glop$ such that the fiber over $X\in \Glop$ is $\Sp_X$, and for any span $\sigma = (X\xleftarrow{f} M \xrightarrow{g} Y)$, the associated functor $\sigma_*: \Sp_X \to \Sp_Y$ is given by $g_{\otimes}\circ  f^*$.  
\end{thm}

\begin{proof}
The proof is an imitation of work of Bachmann-Hoyois \cite[\S 9.2]{BachHoy}.  There is a tautological functor $\iota : \Glo^+ \to \Cat_1$ sending an object of $\Glo^+$, thought of as a category, to its subcategory of compact objects.  Thinking of $\Glo^+$ in terms of groupoids, this sends, for instance, $BG$ to $\Fin^G$.  One may then apply the procedure of \cite[\S 9.2]{BachHoy} (to which we refer the reader for details): namely, one notes that $\iota$ determines a functor 
\[
\Fin_+^{\otimes}: \Glo^+ \to \CAlg(\Cat_1)
\]
which sends a groupoid $X$ to the category $\Fin_{X+}$ of finite pointed $X$-sets, given a symmetric monoidal structure under smash product.  Then, applying $P_{\Sigma}$ pointwise, one obtains a functor
\[
\Spaces_+^{\otimes}: \Glo^+ \to \CAlg(\Cat_{\infty}^{\mathrm{sift}}),
\]
which sends a groupoid $X$ to the category $\Spaces_{X+}$ of pointed genuine $X$-spaces.  Here, $\Cat_{\infty}^{\mathrm{sift}}$ denotes the $\infty$-category of $\infty$-categories with sifted colimits, and sifted colimit preserving functors.  For each $X$, one can consider the subset $T(X)$ of objects of $\Spaces_{X+}$ of the form $p_{\otimes}(S^1)$ as $p$ ranges over finite coverings $Y\to X$ (explicitly, $p_{\otimes}$ should be thought of as indexed smash product along the fibers of $Y\to X$).  The key feature of the functor $\Spaces_+^{\otimes}$ is that for any map $X\to Y$ in $\Glo^+$, the induced map $\Spaces_{X+}\to \Spaces_{Y+}$ sends objects in $T(X)$ to objects in $T(Y)$.  Then, by the argument of \cite[\S 6.1]{BachHoy}, this allows one to invert (pointwise) the chosen objects to obtain the desired functor which sends a groupoid $X$ to $\Sp_X$.
\end{proof}


\begin{rmk}\label{rmk:mapfromspangpd}
Theorem \ref{thm:mainfib} is a strengthening of the constructions in \cite[Section 9.2]{BachHoy}, which produces the restriction of the fibration $\Psi^+$ along the functor $\pi: \Span(\Gpd) \to \Glop$ of Proposition \ref{prop:spangpd}.  The extension of this fibration to $\Glop$ encodes exactly the additional fact that geometric fixed points is inverse to giving trivial action, as explained in Example \ref{exm:trivgfp}.  This additional data turns out to power the construction of the Frobenius action in \S \ref{sect:algact}. 
\end{rmk}

We will also be interested in a subcategory of $\Glop$ which restricts the types of multiplicative transfers that are allowed:

\begin{defn} \label{defn:glofib}
Let $\Glo \subset \Glop$ be the wide subcategory where the morphisms are spans $(X\xleftarrow{f} M \xrightarrow{g} Y)$ which have the property that the forward map $g$ has discrete fibers.  Note that $\Glo$ is closed under the monoidal structure on $\Glop$ and therefore inherits a symmetric monoidal structure.  We shall denote the restriction of $\Psi^+$ along the inclusion by $\Psi: \GloSp \to \Glo.$  
\end{defn}

\begin{rmk}\label{rmk:gloinsidespangpd}
Note that $\Glo$ is also a subcategory of $\Span(\Gpd).$ 
Thus, for the purposes of defining $\Glo$, one can simply start with the $\infty$-category $\Span(\Gpd)$.  However, we will see that in order to get the Frobenius action, it will be important that $\Glo$ is embedded inside $\Glop$ rather than just $\Span(\Gpd)$.  
\end{rmk}

\subsection{Borel equivariant spectra and proper Tate constructions}\label{subsect:bortate}
Let $X\in \Gpd$ be a groupoid.  Then for every point in $X$, one obtains a restriction functor $\Sp_X \to \Sp$.  These assemble to a functor $\Sp_X \to \Fun(X,\Sp)$ which can be thought of as taking a genuine equivariant $X$-spectrum to its underlying spectrum (with group action).  This functor admits a fully faithful right adjoint $j_X: \Fun(X,\Sp)\hookrightarrow \Sp_X$ which we think of as cofreely promoting a non-genuine spectrum to a genuine spectrum.  The essential image of this embedding is known as the full subcategory $(\Sp_X)_{\Bor}\subset \Sp_X$ of \emph{Borel} $X$-spectra.  We will also abusively refer to objects of $\Fun(X,\Sp)$ as Borel $X$-spectra.  We write $\beta_X$, or simply $\beta$, for the composite $$\Sp_X \to \Fun(X,\Sp) \xrightarrow{j_X} \Sp_X,$$ so every $E\in \Sp_X$ admits a natural \emph{Borelification} map $E \to \beta E$.  

\begin{rmk}
A genuine $G$-spectrum $E$ is Borel if and only if, for every subgroup $H\subset G$, the canonical map $E^{H} \to E^{hH}$ from genuine $H$-fixed points to homotopy $H$-fixed points is an equivalence.  
\end{rmk}

The condition of being Borel can also be seen in terms of geometric fixed points:

\begin{exm}\label{exm:Cpsp}
The data of a genuine $C_p$-spectrum $E$ can be presented as a triple $(E_0, E_1, f)$ where $E_0\in \Fun(BC_p,\Sp)$ is the underlying spectrum, $E_1 \in \Sp$ is thought of as $\Phi^{C_p}E$, and $f:E_1 \to E_0^{tC_p}$ is a map of spectra.  From this data, one recovers the genuine fixed points of $E$ via the homotopy pullback square
\begin{equation*}
\begin{tikzcd}
E^{C_p} \arrow[r] \arrow[d] & E_1 \arrow[d,"f"]\\
E_0^{hC_p} \arrow[r] & E_0^{tC_p}.
\end{tikzcd}
\end{equation*}
Thus, the genuine $C_p$-spectrum $E$ is Borel if and only if the specified map $f: E_1 \to E_0^{tC_p}$ is an equivalence.  
\end{exm}

For more general finite groups $G$, a variant of the Tate construction becomes relevant.

\begin{defn}[\cite{AMGR}, Definition 2.7]\label{defn:propertate}
For a finite group $G$, let 
$$(-)^{\tau G}: \Sp^G \to \Sp$$ denote the functor given by the formula $E^{\tau G} = \Phi^G(\beta X).$ We will refer to $(-)^{\tau G}$ as the \emph{proper Tate construction for $G$}.  When $G=C_p$, this agrees with the ordinary Tate construction.
\end{defn}

\begin{rmk}\label{rmk:tatelax}
The proper Tate construction is a lax symmetric monoidal functor because both $\Phi^G$ and $\beta$ are.
\end{rmk}

\begin{rmk} \label{rmk:verdier}
Just as the Tate construction can be thought of as universally killing $G$-spectra of the form $G_+\otimes E$, the proper Tate construction can be thought of as universally killing $G$-spectra induced from any proper subgroup (i.e., of the form $(G/H)_+ \otimes E$ for $H\subsetneq G$).  This description is made precise in \cite[Remark 2.16]{AMGR} using Verdier quotients of stable $\infty$-categories. 
\end{rmk}

\begin{rmk}\label{rmk:propertatecofib}
It is clear from the definition that the proper Tate construction depends only on the underlying spectrum with $G$-action.  Unwinding the definition of geometric fixed points, one can give an alternate description of the proper Tate construction analogous to the description of the usual Tate construction as the cofiber of the additive norm map.  Let $G$ be a finite group, let $\mathcal{O}_G$ denote the category of finite transitive $G$-sets, and let $\mathcal{O}_G^-\subset \mathcal{O}_G$ denote the full subcategory spanned by $G$-sets with nontrivial action.  Then, a spectrum $X$ with $G$-action determines a functor $\mathcal{T}: \mathcal{O}_G \to \Sp$ which sends $G/H$ to $X^{hH}$ and sends maps of $G$-sets to the corresponding additive transfer maps.  The spectrum $X^{\tau G}$ is the total cofiber of this diagram; i.e., it fits into the cofiber sequence $$\colim_{\mathcal{O}_G^- }\mathcal{T} \to X^{hG} \to X^{\tau G}.$$  

\end{rmk}

\begin{exm}[Tate diagonals]\label{exm:tatediag}
Let $G$ be a finite group.  There is a canonical lax symmetric monoidal natural transformation $N^G(-) \to \beta N^G(-)$ of functors $\Sp \to \Sp^G$.  Applying geometric fixed points, this yields a natural transformation $$\Delta^G: \mathrm{id} \simeq \Phi^G N^G (-) \to \Phi^G \beta N^G (-) \simeq ((-)^{\otimes G})^{\tau G}$$ which we call the \emph{Tate diagonal for $G$}.   This has been considered previously in various places \cite{Heuts, NS, Klein, AMGR}.  \end{exm}

\section{The integral Frobenius action}\label{sect:algact}

In this section, we prove Theorem $\text{A}^{\natural}$ by building what we call the \emph{integral Frobenius action}.  This is an oplax action of the symmetric monoidal category $\mathcal{Q}$ on the $\infty$-category $\CAlg$ of $\E_{\infty}$-rings which captures the functoriality of the $\E_{\infty}$-Frobenius.  

In \S \ref{subsect:frob}, we begin by recalling the $\E_{\infty}$-Frobenius map from \cite{NS}.  We then generalize this construction in \S \ref{sub:genfrob} to produce certain ``generalized Frobenius maps" associated to inclusions $H\subset G$ of finite groups and use them to give a precise statement of Theorem $\text{A}^{\natural}$ (Theorem \ref{thm:mainalgaction}).  The remainder of the section is dedicated to proving this theorem.  In \S \ref{subsect:globalg}, we construct the $\infty$-category $\CAlg^{\Glo}$ of \emph{global algebras}, which are roughly $\E_{\infty}$-rings $A$ together with ``genuine equivariant multiplications" that we think of as specifying lifts of the generalized Frobenius maps to endomorphisms of $A$.    We show that these Frobenius lifts assemble into an action of $\mathcal{Q}$ on $\CAlg^{\Glo}$ (Theorem \ref{thm:Qactgloalg}).  It then remains to descend this action to $\CAlg$. We accomplish this in \S \ref{subsect:borglobalg} by introducing a structure mediating between $\E_{\infty}$-rings and global algebras: we call these \emph{Borel global algebras} and denote the $\infty$-category of them by $\CAlg^{\Glo}_{\Bor}$.  These Borel global algebras are, on the one hand, close enough to global algebras that one obtains an \emph{oplax} action of $\mathcal{Q}$ on $\CAlg^{\Glo}_{\Bor}$.  On the other hand, the main result of \S \ref{subsect:borglobalg} is that $\CAlg^{\Glo}_{\Bor}$ is actually equivalent to the $\infty$-category of $\E_{\infty}$-rings.   Together, these facts imply Theorem $\text{A}^{\natural}$.

\subsection{The $\E_{\infty}$-Frobenius} \label{subsect:frob}

\begin{exm}\label{exm:classicalcanfrob}
Let $R$ be a discrete commutative ring and let $p$ be a prime.  Then one has two natural ring maps $R \to R/p$: the quotient map and the $p$th power map.  
\end{exm}

There are analogous maps for an $\E_{\infty}$-ring $A$, which we now recall (following \cite{NS}).  

\begin{defn}\label{defn:canmap}
Let $\can_A: A\to A^{tC_p}$ denote the composite $$A \to A^{hC_p} \to A^{tC_p}$$ where the first map is the restriction induced by the map $BC_p \to *$ and the second map is the canonical projection to the Tate construction.  We refer to $\can_A$ as the \emph{canonical map} and remark that $\can_A$ makes sense for any spectrum $A$ but is a ring map if $A$ is a ring.  
\end{defn}
\begin{defn}\label{defn:frobmap}
Let $\varphi_A: A\to A^{tC_p}$ denote the composite $$\varphi_A: A \xrightarrow{\Delta} (A^{\otimes p})^{tC_p} \to A^{tC_p}$$ of the Tate diagonal followed by multiplication. This is a ring map because $\Delta$ and $(-)^{tC_p}$ are lax monoidal (cf. Remark \ref{rmk:tatelax} and Example \ref{exm:tatediag}).  We will refer to this map throughout as the \emph{$\E_{\infty}$-Frobenius}, or simply the Frobenius.  
\end{defn}

\begin{exm}\label{exm:frobdiscretering}
When $A = R$ is a discrete commutative ring, the canonical map $R \to R^{tC_p}$ on $\pi_0$ exhibits $\pi_0(R^{tC_p})$ as the quotient $R/p$.  Under this identification, the Frobenius on $\pi_0$ is given by the $p$th power map $R\to R/p$.  \end{exm}

However, the $\E_{\infty}$-Frobenius for a discrete ring is not simply the classical Frobenius (indeed, it has a different target).  Instead, one sees the Steenrod operations in higher homotopy groups.   

\begin{exm}[\cite{NS}, Theorem IV.1.15]\label{exm:F2Frob}
When $A = \F_2$, we have that $\pi_*A^{tC_2} \simeq \F_2 ((t))$ where $|t| = -1$.  The Frobenius $$\varphi_{\F_2}: \F_2 \to \F_2^{tC_2}$$ is given, as a map of spectra, by the product over $i\geq 0$ of $$\mathrm{Sq}^i: \F_2 \to \Sigma^i \F_2,$$ where we interpret $\mathrm{Sq}^0$ as the identity on $\F_2$.  

In the case $A = \F_p$ for an odd prime $p$, the action of $\F_p^{\times}$ on $C_p$ induces an action of $\F_p^{\times}$ on $\F_p^{tC_p}$, and we have $\pi_*(\F_p^{tC_p})^{h\F_p^{\times}} \cong \F_p((t^{p-1})) \otimes \F_p[e]/e^2$, where $|t^{p-1}| = -2(p-1)$ and $|e|=1$.  
The Frobenius map is equivariant for this action of $\F_p^{\times}$, and the resulting map
\[
\varphi_{\F_p}: \F_p \to (\F_p^{tC_p})^{h\F_p^{\times}}
\]
has components given by 
\[
\F_p \to \Sigma^{k} \F_p =  \begin{cases} P^s & \text{ if } k=2s(p-1), \text{ }s\geq 0 \\
 \beta P^s & \text{ if } k=2s(p-1) +1, \text{ }s\geq 0 \\
0 &\text{ else.}
\end{cases}
\]
Again, we interpret $P^0$ as the identity in the above formulas.  
\end{exm}  

\begin{rmk}\label{rmk:canformula}
If $M$ is an $\F_p$-module spectrum, then the canonical map $\can: M \to M^{tC_p}$ is equivariant for $\F_p^{\times}$, yielding a map $\can: M \to (M^{tC_p})^{h\F_p^{\times}}$.  On homotopy groups, this map can be identified with the natural inclusion
\[
\pi_*(M) \to \pi_*(M)((t)),
\] where $|t|=-1$ for  $p=2$, and 
\[
\pi_*(M)  \to \pi_*(M)((t^{p-1}))[e]/e^2,
\]
where $ |t^{p-1}| = -2(p-1), |e| = 1$ for $p$ odd.
\end{rmk}

\begin{exm}\label{exm:totalsteenrod}
More generally, if $A$ is an $\E_{\infty}$-algebra over $\F_p$, the $\E_{\infty}$-Frobenius determines a map of $\E_{\infty}$-rings $A \to (A^{tC_p})^{h\F_p^{\times}}$.  Under the identification of the homotopy groups of the target given in the previous remark, one can identify this map on homotopy groups by the formula
\[
x \mapsto \sum_{i\in \Z} \mathrm{Sq}^i(x) t^{-i}
\] when $p=2$, and 
\[
x \mapsto \sum_{i\in \Z} P^i(x) t^{-i(p-1)} + \sum_{i\in \Z}\beta P^i(x) t^{-i(p-1)}e
\]
for $p$ odd.  Here, the $\{ \mathrm{Sq}^i \}_{i\in \Z}$ and $\{P^i, \beta P^i\}_{i\in \Z}$ are the \emph{extended} Steenrod operations of May \cite{MayGen} and Mandell  \cite{Mandell}.  For further explanation of this fact, the reader is referred to \cite[\S 3.5]{DWilson} or \cite[\S IV.1]{NS}.
\end{exm}

In general, the Frobenius can be thought of as a ``total degree $p$ stable power operation.''  The basic example that will be important to us is the following:

\begin{exm}\label{exm:frobsphereequiv}
When $A=\bS$, the sphere spectrum, the Frobenius map and the canonical map are both the unique ring map $\bS \to \bS^{tC_p}$ (since $\bS$ is the initial $\E_{\infty}$-ring).  In fact, this map exhibits $\bS^{tC_p}$ as the $p$-completion of $\bS$ by the theorems of Lin \cite{Lin} ($p=2$) and Gunawardena \cite{Gunawardena} ($p$ odd).   
\end{exm}

\subsection{Generalized Frobenius and canonical maps}\label{sub:genfrob}

As explained in Remark \ref{rmk:classical}, Theorem A heuristically says that the Frobenius ``commutes with itself."   To prove this, one has to understand all ways of composing the Frobenius with itself.  We accomplish this by creating a Frobenius composite $\varphi^G:A \to A^{\tau G}$ for every finite group $G$ (and, in fact, a map $\varphi_H^G: A^{\tau H} \to A^{\tau G}$ for any inclusion $H\subset G$ of finite groups), and then describing how these interact.  Roughly, if $|G| = p^k$, one can think of $\varphi^G$ as a version of the $k$-fold composite of the $\E_{\infty}$-Frobenius $\varphi :A \to A^{tC_p}$.  

\begin{cnstr}[Generalized Frobenius maps] \label{cnstr:genfrobmaps} 
Let $H\subset G$ be an inclusion of finite groups and let $A\in \CAlg$ be an $\E_{\infty}$-ring spectrum.  Then, there is a natural map $A^{\otimes G/H} \to A$ of $\E_{\infty}$-rings with $G$-action.  We may regard $A^{\otimes G/H}$ as the underlying spectrum of the genuine $G$-spectrum $N_H^G (\beta_H \triv^H A)$.  By the universal property of Borel $G$-spectra, one obtains a natural map $$N_H^G (\beta_H \triv^H A) \to \beta_G \triv^G A$$ of $\E_{\infty}$-algebras in genuine $G$-spectra.  Applying $\Phi^G$ and using the equivalence of Example \ref{exm:geonorm}, we obtain a generalized Frobenius map $$\varphi_H^G: A^{\tau H} \simeq \Phi^H (\beta _H \triv^H A) \simeq \Phi^G N_H^G (\beta_H \triv^HA ) \to \Phi^G (\beta_G \triv^G A) \simeq A^{\tau G}.$$
\end{cnstr}

\begin{rmk}
In the special case $H= *$, the generalized Frobenius map $\varphi^G$ is given by the composite $$A \xrightarrow{\Delta^G} (A^{\otimes G})^{\tau G} \to A^{\tau G}$$ of the Tate diagonal (cf. Example \ref{exm:tatediag}) with multiplication.  In particular, taking $G= C_p$, we recover the $\E_{\infty}$-Frobenius.  
\end{rmk}

There are corresponding generalizations of the canonical maps. 

\begin{cnstr}[Generalized canonical maps]\label{cnstr:gencanmaps}
Let $G \twoheadrightarrow K$ be a surjection of finite groups and let $A\in \CAlg$ be an $\E_{\infty}$-ring spectrum.  Then one has a genuine $G$-equivariant map $$\triv_K^G (\beta_K \triv^K A) \to \beta_G \triv^G A$$ by the universal property of Borel $G$-spectra.  Applying $\Phi^G$ and using the equivalence of Example \ref{exm:geotriv}, we obtain a generalized canonical map $$\can_K^G: A^{\tau K} \to A^{\tau G}.$$
\end{cnstr}

\begin{rmk}
In the special case $K=*$, the generalized canonical map $\can^G$ is given by the composite $$A \to A^{hG} \to A^{\tau G}.$$  In particular, when $G=C_p$, we recover the canonical map of Definition \ref{defn:canmap}.
\end{rmk}

We see that while the generalized Frobenius maps are covariantly functorial for injections of groups, the generalized canonical maps are \emph{contravariantly} functorial for surjections of groups.  The interaction between these opposing functorialities is captured by a variant of Quillen's $Q$-construction:


\begin{defn}\label{defn:Q}
Let $\mathcal{Q}$ be the symmetric monoidal 1-category defined as follows:
\begin{itemize}
\item The objects of $\mathcal{Q}$ are finite groups.
\item For finite groups $G,H$, $\Hom_{\mathcal{Q}}(H,G)$ is the set of isomorphism classes of spans of finite groups $(H \twoheadleftarrow K \hookrightarrow G)$ where the left morphism is a surjection and the right morphism is an injection.  
\item Composition of morphisms is the usual composition of spans.
\item The symmetric monoidal structure is by Cartesian product of groups.  
\end{itemize}
\end{defn}

\begin{rmk}\label{rmk:Qabelian}
Strictly speaking, the category of finite groups is not an exact category, and so $\mathcal{Q}$ is not an example of Quillen's $Q$-construction.  However, the full subcategory of $\mathcal{Q}$ spanned by the finite \emph{abelian} groups is Quillen's Q-construction \cite{QuillenK} on the exact category of abelian groups.  
\end{rmk}

We may now state Theorem $\text{A}^{\natural}$ more precisely:

\begin{thm}[Integral Frobenius action]\label{thm:mainalgaction}
There is an oplax monoidal functor $$\Theta: \mathcal{Q} \to \Fun(\CAlg,\CAlg)$$ with the following properties: 

\begin{itemize}
\item  The object $G\in \mathcal{Q}$ acts by the functor $(-)^{\tau G}: \CAlg \to \CAlg$.  
\item A left morphism $(K \twoheadleftarrow G \rightarrow G)$ in $\mathcal{Q}$ is sent to the natural transformation $\mathrm{can}_{K}^{G}: (-)^{\tau K} \to (-)^{\tau G}.$
\item A right morphism $(H \leftarrow H \hookrightarrow G)$ in $\mathcal{Q}$ is sent to the natural transformation $\varphi_H^{G}: (-)^{\tau H}\to (-)^{\tau G}$.  
\end{itemize}
\end{thm}

\begin{rmk}
The oplax structure corresponds to a natural map $A^{\tau (G\times H)} \to (A^{\tau G})^{\tau H}$ for any $\E_{\infty}$-ring $A$ and finite groups $G,H$.   The map can be constructed directly by the universal property of the proper Tate construction (cf. Remark \ref{rmk:verdier}), but one may also give a description in the spirit of Constructions \ref{cnstr:genfrobmaps} and \ref{cnstr:gencanmaps} as follows: we observe that the underlying Borel $G$-spectrum of $\Phi^H \beta_{G\times H} \triv^{G\times H} A$ and $\triv^G \Phi^H \beta_H A$ are both $A^{\tau H}$ with the trivial $G$ action.  This yields a canonical genuine $G$-equivariant map $$\Phi^H \beta_{G\times H} \triv^{G\times H}A \to \beta_G \triv^G \Phi^H \beta_H A,$$ 
from which we extract the desired map as the composite $$A^{\tau G\times H} \simeq \Phi^G \Phi^H \beta_{G\times H}\triv^{G\times H}A \to \Phi^G \beta_G \triv^G \Phi^H \beta_H A \simeq (A^{\tau H})^{\tau G}.$$  
\end{rmk}

\begin{rmk}\label{rmk:adem}
Recall from Example \ref{exm:totalsteenrod} that the Frobenius map for an $\F_2$-algebra $A$ captures the total Steenrod operation.  Unwinding the functoriality of Theorem \ref{thm:mainalgaction} on the full subcategory of $\mathcal{Q}$ spanned by the groups $*,C_2$ and $C_2\times C_2$, one recovers the proof of the Adem relations given in \cite{BullMac} and \cite{DWilson} (the functoriality encoding the invariance of the two-fold composite of the total power operation under conjugation inside $\Sigma_4$ or $\Sigma_2\wr \Sigma_2$).   More generally, thinking of the Frobenius as a ``total stable power operation,'' this theorem can be seen as expressing the higher relations that occur when one composes stable power operations.  
\end{rmk}

The remainder of the section is dedicated to proving Theorem \ref{thm:mainalgaction}.

\subsection{Global algebras}\label{subsect:globalg}
Let $G$ be a finite group.  Then any $\E_{\infty}$-ring spectrum $A$ comes equipped with a multiplication map $(A^{\otimes G})_{hG}\to A$.  However, one could ask for a stronger notion of multiplication.  

\begin{qst*}
Does the multiplication on $A$ canonically lift to a genuine $G$-equivariant multiplication map $N^G A \to \triv^G A$?
\end{qst*}

In general, the answer is no -- one needs to specify more data:

\begin{exm}\label{exm:Cpequivariantmult}
Consider the case $G=C_p$.  Recall that the data of a genuine $C_p$-spectrum $X$ is the data of a spectrum $X$ with $C_p$-action, a spectrum $\Phi^{C_p} X$, and a map $\Phi^{C_p}X \to X^{tC_p}$.  To specify a genuine $C_p$-equivariant map $N^{C_p}A \to \triv^{C_p} A$ lifting the $\E_{\infty}$ multiplication on $A$ is to fill in the dotted arrow in the following diagram:
\begin{equation*}
\begin{tikzcd}
\Phi^{C_p}N^{C_p}A \arrow[r,equal] & A \arrow[r,dashed] \arrow[d,"\Delta"] & A \arrow[r,equal] \arrow[d,"\can"] &\Phi^{C_p}\triv^{C_p}A  \\
& (A^{\otimes p})^{tC_p} \arrow[r] & A^{tC_p} & \\
\end{tikzcd}
\end{equation*}
where the bottom arrow is induced by the $\E_{\infty}$ structure.  In other words, one needs a lift of the $\E_{\infty}$-Frobenius map $A\to A^{tC_p}$ to an endomorphism of $A$.  In general, a genuine $G$-equivariant multiplication in the above sense gives $A$ the dotted ``Frobenius lift" in the following diagram:
\begin{equation*}
\begin{tikzcd}
A \arrow[r,dashed] \arrow[rd,"\varphi^G"']& A\arrow[d,"\can^G"]\\
 & A^{\tau G}.
\end{tikzcd}
\end{equation*}
\end{exm}

We now define the notion of a \emph{global algebra}, which is roughly an $\E_{\infty}$-ring spectrum together with coherent choices of genuine $G$-equivariant multiplication for all $G$.  

\begin{defn}\label{defn:gloalg}
A \emph{global algebra} is a section of the fibration $\Psi: \GloSp \to \Glo$ (cf. Definition \ref{defn:glofib}) which is coCartesian over the left morphisms.  We let $\CAlg^{\Glo}$ denote the $\infty$-category of global algebras.  
\end{defn}

\begin{rmk}\label{rmk:concgloalg}
Concretely, such a section $A: \Glo \to \GloSp$ is an assignment to each groupoid $X\in \Glo$ a genuine equivariant spectrum $A(X)\in \Sp_X$ together with certain structure maps.  The condition that the section is coCartesian over the left morphisms implies in particular that $A(BG) \simeq \triv^G A(*)$ (by considering the span $(*\leftarrow BG \rightarrow BG)$), so that the value of the section on a point determines the value of the section on any other groupoid.  The functoriality in the span $(* \leftarrow * \rightarrow BG)$ encodes the genuine multiplication maps $$N^G A(*) \simeq (*\leftarrow * \rightarrow BG)_* A(*) \to A(BG) \simeq \triv^G A.$$  
\end{rmk}

\begin{rmk}\label{rmk:undeinfty}
The restriction of a global algebra $A$ along the inclusion $\Span(\Fin) \to \Glo$ of Example \ref{exm:spanfininside} endows $A(*)$ with the structure of an $\E_{\infty}$-ring spectrum by \cite[Corollary C.2]{BachHoy}.  We will think of $A(*)$ as the underlying $\E_{\infty}$-ring of $A$.  
\end{rmk}

By Example \ref{exm:Cpequivariantmult}, we can think of a global algebra as an $\E_{\infty}$-ring together with ``Frobenius lifts." 
The functoriality of these Frobenius lifts is explained by the following observation:

\begin{obs}\label{obs:twist}
Let $A\in \CAlg^{\Glo}$ and $G$ be a finite group.  Then one can form another global algebra $A^G$ (thought of as $A$ twisted by $G$), which is given on objects via the formula $$A^G(X)  = \sigma_* A(X\times BG),$$ where $\sigma$ denotes the morphism $(X\times BG \leftarrow X\times BG \xrightarrow{\mathrm{proj}_1} X)$ in $\Glo^+$. 
More concretely, we have for any finite group $H$ that $$A^G(BH) = \Phi^G A(BH\times BG) \simeq \Phi^G \triv_H^{H\times G} A(BH).$$    We will construct $A^G$ formally in the proof of \Cref{thm:Qactgloalg} using the results of the appendix.  

These twists $A^G$ have certain functorialities in the finite group $G$:
\begin{enumerate}[leftmargin=*]
\item For an injection $H\to G$, we have a natural map $A^H\to A^G$ given by:
\begin{align*} A^H(X) &= (X\times BH \leftarrow X\times BH \rightarrow X)_* A(X\times BH)\\
&\simeq (X\times BG\leftarrow X\times BG \rightarrow X)_* (X\times BH\leftarrow X\times BH\rightarrow X\times BG)_* A(X\times BH) \\
&\to (X\times BG \leftarrow X\times BG \rightarrow X)_* A(X\times BG) = A^G(X)
\end{align*}
where we have used that $A$ is a section over $\Glo$ and that $X\times BH \to X\times BG$ has discrete fibers.  
\item For a surjection $G\to K$, we have a natural map $A^K \to A^G$ given by:
\begin{align*}
A^K(X) & = (X\times BK \leftarrow X\times BK \rightarrow X)_* A(X\times BK)\\
 &\simeq (X\times BG\leftarrow X\times BG\rightarrow X)_* (X\times BK \leftarrow X\times BG \rightarrow X\times BG)_* A(X\times BK)\\
 &\to (X\times BG \leftarrow X\times BG \rightarrow X)_* A(X\times BG) = A^G(X)
\end{align*}
where we note that we have used critically the definition of the composition law of $\Glo^+$ and the fact that $X\times BG\to X\times BK$ has connected fibers.  
\end{enumerate}
\end{obs}

The main result of this section is that these two opposing functorialities assemble into an action of $\mathcal{Q}$ on global algebras.  In fact, the twists make sense for any section of $\Psi$ -- not just global algebras; we therefore state our theorem more generally:

\begin{thm}\label{thm:Qactgloalg}
There is a monoidal functor $$\mathcal{Q} \to \Fun(\sect(\Psi), \sect(\Psi))$$ sending a group $G$ to the functor $A\mapsto A^G$ and such that maps in $\mathcal{Q}$ are sent to the corresponding maps identified in Observation \ref{obs:twist}.  This action fixes the full subcategory of global algebras, and thus determines an action of $\mathcal{Q}$ on $\CAlg^{\Glo}$.  
\end{thm}

\begin{proof}

By Proposition \ref{prop:mainaction} applied to the coCartesian fibration $q^+ = \Psi^+: \GlopSp \to \Glop$ and the subcategory $i = \iota : \Glo \to \Glop$, we learn that the $\infty$-category $\sect (\Psi)$ admits a natural right action of the monoidal $\infty$-category $\Fun(\Glo,\Glo)\times_{\Fun(\Glo,\Glop)}\Fun(\Glo,\Glop)_{/\iota}.$  Since $\mathcal{Q}$ is symmetric monoidal, left and right actions by $\mathcal{Q}$ coincide, so the proof is completed by the following proposition, which asserts that this action can be restricted to an action of $\mathcal{Q}$ admitting the description given in Observation \ref{obs:twist}.  

\end{proof}

\begin{prop} \label{prop:Qinside}
There is a monoidal functor $$\mathcal{Q} \to \Fun(\Glo,\Glo)\times_{\Fun(\Glo,\Glop)}\Fun(\Glo,\Glop)_{/\iota}$$ sending a group $G$ to the functor $(-)\times BG: \Glo \to \Glo$ given by multiplication by $BG$ together with the natural transformation $\iota(-) \times BG \to \iota(-)$ of functors $\Glo\to \Glo^+$ given on $X\in \Glo$ by the span $(X\times BG\leftarrow X \times BG \rightarrow X).$    
\end{prop}
\begin{proof}
The symmetric monoidal inclusion $\iota: \Glo \to \Glop$ endows $\Glop$ with the structure of a module over the symmetric monoidal $\infty$-category $\Glo$.  The $\infty$-category $\Mod_{\Glo}(\Cat_{\infty})$ of $\Glo$-module $\infty$-categories is naturally tensored over $\Cat_{\infty}$ via the Cartesian product and the forgetful functor $\Mod_{\Glo}(\Cat_{\infty}) \to \Cat_{\infty}$ respects this tensoring.  

Recall the Cartesian fibration $\mathcal{M} \to \Cat_{\infty}$ of Construction \ref{cnstr:slicemonoidal}.  We make the following variant: let $\mathcal{N} \to \Mod_{\Glo}(\Cat_{\infty})$ denote the Cartesian fibration classified by the functor $\Mod_{\Glo}(\Cat_{\infty}) \to \Cat_{\infty}^{\op}$ which sends $\C$ to $\Fun_{\Glo}(\C,\Glop)^{\op}$, where the subscript denotes the $\infty$-category of $\Glo$-module maps.  The forgetful natural transformation $$\Fun_{\Glo}(\C,\Glop)^{\op} \to \Fun(\C,\Glop)^{\op}$$ yields a functor $U: \mathcal{N} \to  \mathcal{M}$.  By construction, the functor $U$ respects the tensoring over $\Cat_{\infty}$.  Moreover, the inclusion $\iota: \Glo \to \Glop$ is a map of $\Glo$-modules, so $(\Glo,\iota)\in\mathcal{M}$ admits a natural lift $(\Glo,\iota)\in \mathcal{N}.$  Consequently, there is a monoidal functor $$\End_{\mathcal{N}}(\Glo,\iota) \to \End_{\mathcal{M}}(\Glo,\iota) = \Fun(\Glo,\Glo)\times_{\Fun(\Glo,\Glop)}\Fun(\Glo,\Glop)_{/\iota}.$$  

It therefore suffices to produce an appropriate monoidal functor $\mathcal{Q} \to \End_{\mathcal{N}}(\Glo, \iota).$  By an analogous argument to Lemma \ref{lem:monendcat}, we may identify the underlying $\infty$-category
\begin{align*}\End_{\mathcal{N}}(\Glo,\iota) &\simeq \Fun^{\Glo}(\Glo,\Glo)\times_{\Fun^{\Glo}(\Glo,\Glop)}\Fun^{\Glo}(\Glo,\Glop)_{/\iota}\\
&\simeq \Glo \times_{\Glop} \Glop_{/*}.\end{align*}

Let us analyze the natural projection $$p: \End_{\mathcal{N}}(\Glo,\iota)\simeq  \Glo \times_{\Glop} \Glop_{/*} \to \Glo,$$ which has the structure of a monoidal functor by construction.  Consider the following subcategories:
\begin{enumerate}
\item Let $\D \subset \Glo \times_{\Glop} \Glop_{/*}$ be the full subcategory spanned by spans of the form $(X = X \rightarrow *)$ such that $X$ is connected.  For ease of notation, we may refer to this object of $\D$ simply as $X$.  
\item Let $\tilde{\mathcal{Q}}\subset \Glo$ be the subcategory spanned by the connected groupoids and morphisms of the form $(X \xleftarrow{f} M \xrightarrow{g} Y)$ where $f$ has connected fibers and $g$ has discrete fibers.  In other words, the spans take the form $BH \leftarrow BG \rightarrow BK$ where $G\to H$ is surjective and $G\to K$ is injective.  
\end{enumerate}
It is immediate that $\D$ and $\tilde{\mathcal{Q}}$ are monoidal subcategories, and that the projection $p$ restricts to a monoidal functor $p:\D \to \tilde{\mathcal{Q}}.$  
\begin{lem}
The restricted functor $p:\D\to \tilde{\mathcal{Q}}$ is an equivalence of monoidal $\infty$-categories.  
\end{lem}
\begin{proof}
Essential surjectivity is obvious so it suffices to show that $p:\D \to \tilde{\mathcal{Q}}$ is fully faithful.  For this, consider two objects of $\D$ corresponding to the groupoids $BH$ and $BK$ together with the natural maps $p_H = (BH = BH \to *)$ and $p_K = (BK = BK \to *)$ in $\Glop$.  We would like to show that the natural map $$\Hom_{\Glo}(BH,BK) \times_{\Hom_{\Glop}(BH,BK)} \Hom_{\Glop_{/*}}(BH,BK) \to \Hom_{\tilde{\mathcal{Q}}}(BH,BK)$$ is an equivalence of groupoids.  We first observe that any map $BH\to BK$ in $\Glo$ which is compatible with the natural maps $BH \to *$ and $BG\to *$ in $\Glop$ is automatically in the subcategory $\tilde{\mathcal{Q}}\subset \Glo$.  Thus, it suffices to show the natural map  \begin{align}\label{eqn:Qglop}
\Hom_{\tilde{\mathcal{Q}}}(BH,BK) \times_{\Hom_{\Glop}(BH,BK)} \Hom_{\Glop_{/*}}(BH,BK) \to \Hom_{\tilde{\mathcal{Q}}}(BH,BK)
\end{align}
 is an equivalence of groupoids.  But we have that 
$$\Hom_{\Glop_{/*}}(BH,BK) \simeq \Hom_{\Glop}(BH,BK)\times_{\Hom_{\Glop}(BH, *)} \{ p_H \}.$$
The lemma then follows by noting that for any $\sigma \in \Hom_{\tilde{\mathcal{Q}}}(BH,BK)$, we have $p_H = p_K \circ \sigma$ so (\ref{eqn:Qglop}) is essentially surjective, and $p_H \in \Hom_{\Glop}(BH, *) \simeq (\Fin_H^{\op})^{\simeq}$ has no automorphisms (since it corresponds to the singleton as an $H$-set) so (\ref{eqn:Qglop}) is fully faithful.
\end{proof}

As a consequence of the lemma, we obtain a monoidal functor $\tilde{\mathcal{Q}} \to \End_{\mathcal{N}}(\Glo,\iota)$.  To complete the proof of Proposition \ref{prop:Qinside} and with it, Theorem \ref{thm:Qactgloalg}, we will show that there is a monoidal functor from $\mathcal{Q}$ (whose objects are groups) to $\tilde{\mathcal{Q}}$ (whose objects are groupoids), which allows us to restrict the action to $\mathcal{Q}$.  To do this, we note that one can consider the $\infty$-category $\Span(\Gpd_*)$ of \emph{pointed} groupoids, which admits a monoidal forgetful functor 
\begin{equation}\label{eqn:forgetgpdpt}
\Span(\Gpd_*) \to \Span(\Gpd).
\end{equation}
Recall that the monoidal $\infty$-category $\tilde{\mathcal{Q}}$ was defined above as a monoidal subcategory of $\Glo$, and thus of $\Span(\Gpd)$ (\Cref{rmk:gloinsidespangpd}).  On the other hand, one can consider the subcategory of $\Span(\Gpd_*)$ spanned by connected groupoids, and where maps from $BH$ to $BK$ are spans $BH \leftarrow BG \rightarrow BK$ where the backward arrow is surjective on $\pi_1$ and the forward arrow is injective on $\pi_1$.  This subcategory is closed under the monoidal structure, and is easily seen to be equivalent to a $1$-category; therefore it is elementary to check that $\Omega$ and $B$ define inverse monoidal equivalences of this subcategory with $\mathcal{Q}$.  Hence, we may restrict (\ref{eqn:forgetgpdpt}) along the inclusion of the subcategory $\mathcal{Q}$ to obtain a monoidal functor $\mathcal{Q} \to \Span(\Gpd)$.  This has image inside the monoidal subcategory $\tilde{\mathcal{Q}}$, and so we may restrict the codomain to obtain the desired monoidal functor $\mathcal{Q} \to \tilde{\mathcal{Q}}$.  \end{proof}

\subsection{Borel global algebras}\label{subsect:borglobalg}
In this section, we complete the proof of Theorem \ref{thm:mainalgaction}, which concerns producing an oplax action of $\mathcal{Q}$ on $\CAlg$.  We will deduce this from Theorem \ref{thm:Qactgloalg} by studying the difference between $\E_{\infty}$-rings and global algebras.  

Let $A\in \CAlg$ be an $\E_{\infty}$-ring spectrum.  We have seen that $A$ does not canonically determine a global algebra because $A$ does not come with a genuine $C_p$-equivariant map $N^{C_p}A \to \triv^{C_p}A$.  However, $A$ does have a weaker version of this structure: the data of a genuine $C_p$-equivariant map $N^{C_p}A \to \beta_{C_p} \triv^{C_p} A$, because this is the same data as a map $A^{\otimes p} \to A$ of spectra with $C_p$-action.  The notion of a Borel global algebra is the analog of a global algebra where one demands this weaker structure.  

After constructing the $\infty$-category $\CAlg^{\Glo}_{\Bor}$ of Borel global algebras, we show that the action of Theorem \ref{thm:Qactgloalg} descends to an oplax action of $\mathcal{Q}$ on $\CAlg^{\Glo}_{\Bor}$ (Proposition \ref{prop:boractionoplax}).  On the other hand, we prove that sending a Borel global algebra to its underlying $\E_{\infty}$-ring induces an equivalence of $\infty$-categories $\CAlg^{\Glo}_{\Bor} \simeq \CAlg$ (Theorem \ref{thm:globalalgs}).  Together, these two facts imply Theorem \ref{thm:mainalgaction}.

\begin{defn}
The total space $\GloSp$ of $\Psi$ can be described as the $\infty$-category of pairs $(E,X)$ where $X\in \Glo$ and $E\in \Sp_X$ is a genuine $X$-spectrum.  For each $X\in \Glo$, recall that there is a full subcategory $j_X: (\Sp_X)_{\mathrm{Bor}} \hookrightarrow \Sp_X$ of Borel $X$-spectra.  Let $\GloSpBor \subset \GloSp$ be the full subcategory on the pairs $(E, X) \in \GloSp$ with the property that $E\in (\Sp_X)_{\Bor}$.  This yields a diagram

\begin{equation}\label{dia:globorglo}
\begin{tikzcd}
\GloSpBor \arrow[rr,"j"] \arrow[dr,"\Psi_{\mathrm{Bor}}", swap] & &\GloSp \arrow[ld,"\Psi"]\\
& \Glo .& \\
\end{tikzcd}
\end{equation}
\end{defn}

\begin{prop}\label{prop:borcocart}
The map $\Psi_{\mathrm{Bor}}$ is a coCartesian fibration.  
\end{prop}
\begin{proof}
The map $\Psi_{\mathrm{Bor}}$ is an inner fibration because it is a full simplicial subset of an inner fibration.  

Let $\sigma =(X \leftarrow S \rightarrow Y)$ be a morphism in $\Glo$ and $E\in (\Sp_X)_{\mathrm{Bor}}$.  Then, $\sigma$ has a coCartesian lift given by the arrow $(E,X) \to (\beta_Y \sigma_* E, Y)$ in $\GloSpBor$ induced by the natural map $\sigma_* E \to \beta \sigma_* E.$  It follows that $\Psi_{\mathrm{Bor}}$ is a locally coCartesian fibration.  But these locally coCartesian arrows are closed under composition: for any $\tau = (Y \leftarrow T\rightarrow Z)$ in $\Glo$, we have that the natural transformation $\beta_Z \tau_* \to \beta_Z \tau_* \beta_Y$ of functors $\Sp_Y \to \Sp_Z$ is an equivalence, so for any $E\in (\Sp_X)_{\Bor}$, the natural map $$\beta_Z (\tau \circ \sigma)_* E \to \beta_Z \tau_* \beta_Y \sigma_* E$$ is an equivalence  (note that it is important that we are working with $\Psi$ and not all of $\Psi^+$ because this may not be true if $\tau$ is in $\Glop$ but not $\Glo$.).  Thus, $\Psi_{\mathrm{Bor}}$ is a coCartesian fibration by \cite[Proposition 2.4.2.8]{HTT}.  
\end{proof}
\begin{warn}
The inclusion $\GloSpBor \subset \GloSp$ does not send coCartesian arrows to coCartesian arrows because the natural map $\sigma_* E\to \beta \sigma_* E$ need not be an equivalence.  
\end{warn}

\begin{lem}
The inclusion $j: \GloSpBor \to \GloSp$ of diagram (\ref{dia:globorglo}) admits a left adjoint relative to $\Glo$ in the sense of \cite[Definition 7.3.2.2]{HA}.  
\end{lem}
\begin{proof}
We apply \cite[Proposition 7.3.2.11]{HA}.  The functors $\Psi$ and $\Psi_{\Bor}$ are coCartesian fibrations by Theorem \ref{thm:mainfib} and Proposition \ref{prop:borcocart}, and therefore they are locally coCartesian categorical fibrations.  For condition (1), we need only recall that for each $X\in \Glo$, the inclusion $j_X: (\Sp_X)_{\mathrm{Bor}} \to \Sp_X$ admits a left adjoint $$\beta_X :\Sp_X \xrightleftharpoons{\quad}(\Sp_X)_{\mathrm{Bor}}: j_X.$$ Condition (2) amounts to the fact that the genuine equivariant norm and restriction maps coincide with the corresponding Borel norm and restriction maps on the underlying Borel equivariant spectrum.  
\end{proof}
\begin{cor}\label{cor:sectboradj}
There is an adjunction
$$\beta^s: \sect(\Psi) \xrightleftharpoons{\quad} \sect(\Psi_{\Bor}): j^s$$ at the level of sections which restricts to the adjunction $$\beta_X :\Sp_X \xrightleftharpoons{\quad}(\Sp_X)_{\mathrm{Bor}}: j_X$$ over each $X\in \Glo.$
\end{cor}

We now make the analog of Definition \ref{defn:gloalg} in this setting.

\begin{defn}\label{defn:borgloalg}
A \emph{Borel global algebra} is a section of the fibration $\Psi_{\Bor}$ which is coCartesian over the left morphisms.  We let $\CAlg^{\Glo}_{\Bor}$ denote the $\infty$-category of Borel global algebras.  
\end{defn}

\begin{rmk}
A Borel global algebra can be thought of as an $\E_\infty$ algebra $A(BG) \in \Fun(BG,\Sp)$ for every groupoid $BG$ together with structure maps corresponding to the various maps in $\Glo$.  As in the case of global algebras,  $A(*)$ acquires the structure of an $\E_{\infty}$-ring spectrum (cf. Remark \ref{rmk:undeinfty}) and determines the value of $A$ at any other $X\in \Glo.$ The right morphisms then encode certain multiplication maps; for instance the span $(* \leftarrow * \rightarrow BC_p)$ encodes a map $$\beta N^{C_p}A(*) \to \beta \triv^{C_p}A(*)$$ in $\Fun(BC_p, \Sp)$.  We saw at the beginning of the section that this map is already part of the $\E_{\infty}$-structure on $A(*)$.  In fact, we will show in Theorem \ref{thm:globalalgs} that the additional structure of a Borel global algebra is specified uniquely by $A(*)$ as an $\E_{\infty}$-ring.  
\end{rmk}

Theorem \ref{thm:mainalgaction} follows immediately from the following two statements about $\CAlg^{\Glo}_{\Bor}$:

\begin{prop}\label{prop:boractionoplax}
There is an oplax monoidal functor $$\mathcal{Q} \to \Fun(\CAlg^{\Glo}_{\Bor}, \CAlg^{\Glo}_{\Bor})$$ which admits the description of Theorem \ref{thm:mainalgaction} on underlying $\E_{\infty}$-rings.  
\end{prop}

\begin{proof}
By Theorem \ref{thm:Qactgloalg} and Lemma \ref{lem:oplax} applied to the adjunction of Corollary \ref{cor:sectboradj}, we obtain an oplax monoidal functor $$\mathcal{Q} \to  \Fun(\sect (\Psi_{\Bor}),\sect (\Psi_{\Bor})). $$

One can describe this action more explicitly as follows.  Let $A\in \sect(\Psi_{\Bor})$ and $G\in \mathcal{Q}$.  By the formula of Observation \ref{obs:twist}, we see that the action of $G$ on $A$ yields a new section $A^G$ whose value on a groupoid $BH$ is given by $$A^G(BH) = \beta_{BH} \Phi^G j_{BH\times BG} A(BH\times BG).$$  If $A$ is in the full subcategory $\CAlg^{\Glo}_{\Bor} \subset \sect(\Psi_{\Bor})$ of Definition \ref{defn:borgloalg}, then this is equivalent to $A^{\tau G}$ with the trivial $H$-action for all $H$; thus, if $A$ is a Borel global algebra, then $A^G$ is as well.  It follows that the oplax action on $\sect(\Psi_{\Bor})$ restricts to an oplax action on $\CAlg^{\Glo}_{\Bor}$.  Unwinding the definitions, we see that the maps in $\mathcal{Q}$ act as described in the statement of Theorem \ref{thm:mainalgaction}.  
\end{proof}

\begin{thm}\label{thm:globalalgs}
The restriction functor $\CAlg^{\mathrm{Glo}}_{\mathrm{Bor}} \to \CAlg$ induces an equivalence of $\infty$-categories.  
\end{thm}
\begin{proof}
The relevant restriction functor is implemented by restricting a section of $\Psi_{\Bor}$ along the inclusion $i: \Span(\Fin) \to \Glo$ (cf. Remark \ref{rmk:undeinfty}).  We show this induces an equivalence by proving the existence of and computing the $\Psi_{\mathrm{Bor}}$-right Kan extension along this inclusion.  

Let $A\in \CAlg$ be an $\E_{\infty}$-ring, and let us first compute the pointwise $\Psi_{\Bor}$-right Kan extension in the diagram
\begin{equation*}
\begin{tikzcd}
\Span(\Fin) \arrow[r,"A"]\arrow[d] & \GloSpBor \arrow[d,"\Psi_{\mathrm{Bor}}"] \\
\Glo \arrow[r,equals]\arrow[ru,dashed] & \Glo
\end{tikzcd}
\end{equation*}
at $X\in \Glo$; this is given by a certain $\Psi_{\Bor}$-limit indexed by the 2-category $\Span(\Fin)\times_{\Glo} \Glo_{X/}.$  

We first examine the indexing category.  An object of $\Span(\Fin)\times_{\Glo} \Glo_{X/}$ is a span of groupoids $(X\leftarrow Y \rightarrow Z)$ such that $Z$ is a finite set and the map $Y\to Z$ has discrete fibers.  It follows that $Y$ must also be a finite set.  We shall use the shorthand $\Fin_{/X}$ for the category $\Fin \times_{\Gpd} \Gpd_{/X}$ whose objects are finite sets $T$ equipped with a map of groupoids $p:T\to X$.  We have the following lemma:
\begin{lem}
The functor $\theta: (\Fin_{/X})^{\op} \to \Span(\Fin)\times_{\Glo} \Glo_{X/}$ defined by $$(p: T\to X) \mapsto (X \xleftarrow{p} T \xrightarrow{=} T)$$ admits a right adjoint.  
\end{lem}
\begin{proof}
The adjoint is given by sending $(X\xleftarrow{q} T \to U)$ to the map $q:T\to X$.  The groupoid $T$ is a finite set by the remarks above and it is immediate that this is a right adjoint. 
\end{proof}

It follows that $\theta$ is coinitial and we have reduced to computing the relative limit over $(\Fin_{/X})^{\op}$.  Unwinding the definitions, a limit of this diagram relative to the fibration $\Psi_{\mathrm{Bor}}$ is the data of $\tilde{A} \in \Fun(X, \Sp)$ equipped with compatible maps $\nu_p: A(T) \to \theta(p)_*\tilde{A}$ for each $p:T\to X \in \Fin_{/X}^{\op}$ such that for any $E\in \Fun(X,\Sp)$, the natural map
\begin{align*}
\Hom_{\Fun(X,\Sp)}(\tilde{A}, E) &\longrightarrow \lim_{p:T\to X \in \Fin_{/X}^{\op}} \Hom_{\Fun(T,\Sp)}(\theta(p)_* \tilde{A}, \theta(p)_* E)\\
&\longrightarrow \lim_{p:T\to X \in \Fin_{/X}^{\op}} \Hom_{\Fun(T,\Sp)}(A(T), \theta(p)_* E)
\end{align*}
is an equivalence.

We claim that $\tilde{A}:= (*\leftarrow X \xrightarrow{=} X)_* A(*) = \triv^X A(*)$ (together with the obvious choice of $\nu_p$) has the desired universal property.  Because everything in sight sends disjoint unions in $X$ to products, it suffices to consider the case when $X$ is connected.  Without loss of generality, let $X=BG$.  
Then we have $\Fin_{/X} \simeq \Fin_G^{\mathrm{free}},$ the category of finite free $G$-sets, and we wish to show that for any $E\in \Fun(BG,\Sp)$, the natural map 
\begin{align*}
\Hom_{\Fun(BG,\Sp)}(\triv^G (A(*)), &E) \\
 &\rightarrow  \lim_{U\in (\Fin_G^{\mathrm{free}})^{\op}} \Hom_{\Fun(U/G,\Sp)}(A(U/G),(BG\leftarrow U/G \rightarrow U/G)_* E)
\end{align*}
  is an equivalence.  Since $A$ was assumed to be an $\E_{\infty}$-algebra, the functor of $U$ on the right-hand side is product preserving.  Since $\Fin_G^{\mathrm{free}}$ is generated freely under coproducts by the full subcategory on the transitive free $G$-set, it follows that the limit diagram is right Kan extended from that subcategory, and so we have  
\begin{align*}
\lim_{U\in (\Fin_G^{\mathrm{free}})^{\op}} \Hom_{\Fun(U/G,\Sp)}(A(U/G)&,(BG\leftarrow U/G \rightarrow U/G)_* E)\\
&\simeq \lim_{BG} \Hom_{\Sp}(A(*),(BG\leftarrow * \rightarrow *)_* E)\\
&\simeq \Hom_{\Sp}(A(*),\res_G^e E)^{hG} \\
&\simeq \Hom_{\Fun(BG,\Sp)}(\triv^G(A(*)),E),
\end{align*}
as desired.  

We have shown that the $\Psi_{\Bor}$-right Kan extension exists at every point, and so by \cite[Lemma 4.3.2.13]{HA}, the $\Psi_{\Bor}$-right Kan extension exists.  Moreover, our calculation shows that this Kan extension takes an $\E_\infty$-algebra to a Borel global algebra with the same underlying $\E_{\infty}$-ring.  It follows that both the unit and the counit of the resulting adjunction are equivalences.  Thus, the left adjoint induces an equivalence of $\infty$-categories $\CAlg^{\mathrm{Glo}}_{\mathrm{Bor}} \to \CAlg$ as desired.
\end{proof}

This concludes the proof of Theorem \ref{thm:mainalgaction}.  We make two remarks about the proof:

\begin{rmk}\label{rmk:Qtilde}
In fact, the proof of Proposition \ref{prop:Qinside} shows that the monoidal functor from $\mathcal{Q}$ arises from one defined on the larger 2-category $\widetilde{\mathcal{Q}}$.  We do not know if this additional generality has interesting consequences and will not use it in this paper.
\end{rmk}

\begin{rmk}
The action of Theorem \ref{thm:mainalgaction} could not have been produced directly using Proposition \ref{prop:mainaction} because the coCartesian fibration $\Psi_{\Bor}$ does not arise as the restriction of a fibration over all of $\Glo^+$.  This results in an oplax, rather than strict, action.  In terms of equivariant homotopy theory, this corresponds to the fact that Borel equivariant homotopy theory does not admit a monoidal fixed point functor analogous to geometric fixed points in genuine equivariant homotopy theory.
\end{rmk}

\section{Partial algebraic $K$-theory}\label{sect:partK}

Let $\mathcal{C}$ be an exact category in the sense of  Quillen \cite{QuillenK}.  Then, the zeroth $K$-theory of $\C$, denoted $K_0(\mathcal{C})$, is the free abelian group on the objects of $\mathcal{C}$ subject to the relation $[A]+[C] = [B]$ for every short exact sequence $0\to A\to B\to C\to 0$ in $\mathcal{C}$.  Quillen categorified this construction, defining the higher algebraic $K$-theory space $K(\C)$ by means of a certain category $Q(\mathcal{C})$.  Waldhausen \cite{Wald} generalized the construction of algebraic $K$-theory to what are now known as ``Waldhausen categories" by means of his $S_{\bullet}$-construction and proved that the definition coincides with Quillen's in the special case of an exact category.  These constructions were generalized to the higher categorical setting by Barwick \cite{BarQ, BarAKT, BarK}.  

The goal of this section is to give analogs of these constructions in a non-group-complete setting.  In \S \ref{sub:Sdot}, we introduce a construction called \emph{partial $K$-theory}, which is a non-group-complete analog of algebraic $K$-theory.  It associates to a Waldhausen $\infty$-category $\C $ a (not necessarily grouplike) $\E_{\infty}$-space $K^{\mathrm{part}}(\C)$ with the following two properties:
\begin{enumerate}
\item There is a canonical equivalence of $\E_{\infty}$-spaces $K^{\mathrm{part}}(\C)^{\mathrm{gp}} \simeq K(\C)$ (Corollary \ref{cor:kpartgp}).
\item The monoid $\pi_0 (K^{\mathrm{part}}(\C))$ is the free (discrete) monoid on $\mathcal{C}$ subject to the relation  $[A]+[C] = [B]$ for every short exact sequence $0\to A\to B\to C\to 0$ (Proposition \ref{prop:Kpart0}). 
\end{enumerate} 
Then, in \S \ref{sub:Q}, we give an alternate construction of partial $K$-theory for exact $\infty$-categories via Quillen's $Q$-construction and show that it coincides with the previous definition (Theorem \ref{thm:QequalsS}).

\subsection{Partial $K$-theory via the $S_{\bullet}$-construction}\label{sub:Sdot}

Let $\C$ be a Waldhausen $\infty$-category in the sense of \cite{BarK}.  One can extract from $\C$ a simplicial $\infty$-category $S_{\bullet}(\mathcal{C})$ such that $S_n(\C)$ is equivalent to the $\infty$-category of sequences of cofibrations $$* \hookrightarrow X_1 \hookrightarrow X_2 \hookrightarrow \cdots \hookrightarrow X_n$$ between objects $X_i\in \C$  (\cite{Wald, BarK}).

\begin{defn}[\cite{Wald, BarK}]\label{defn:Ktheory}
The algebraic $K$-theory of $\C$ is the $\E_1$-space $$K(\C) := \Omega |S_{\bullet}(\C)^{\simeq}|,$$ where $S_{\bullet}(\C)^{\simeq}$ denotes the simplicial space obtained by taking the maximal subgroupoid of $S_{\bullet}(\C)$ level-wise.  
\end{defn}

The algebraic $K$-theory of $\C$ can be thought of as the universal way to make $S_{\bullet}(\C)$ into a \emph{grouplike}  $\E_1$-monoid in spaces.  As explained in the introduction, we need a variant of this construction which can produce non-group-complete monoids.

\begin{defn}\label{defn:segalspace}
A \emph{Segal space} is a functor $X(-):\Delta^{\op} \to \Spaces$ such that for each $n\geq 1$, the collection of maps $\rho_i: [1] \to [n]$ in $\Delta$ defined by $\rho_i(0) = i$, $\rho_i(1) = i+1$ for $0\leq i \leq n-1$ induces an equivalence $$\prod_{i=0}^{n-1} X(\rho_i) : X([n]) \simeq X([1])\times_{X([0])} \cdots \times_{X([0])} X([1]).$$   We will denote the $\infty$-category of Segal spaces by $\mathrm{Seg}(\Spaces)$.  
\end{defn}

\begin{defnprop}[\cite{HA}, Proposition 4.1.2.10]\label{defnprop:monoid}
Let $\mathrm{Mon}(\Spaces)$ denote the $\infty$-category of $\E_1$-monoids in spaces.  Then there is a fully faithful functor $$\mathbb{B}: \mathrm{Mon}(\Spaces) \to \Fun(\Delta^{\mathrm{op}},\Spaces)$$ which sends a monoid $M$ to its bar construction
\begin{equation*}
\begin{tikzcd}
\mathbb{B}M = \big( *  & M\arrow[l,shift left] \arrow[l,shift right] & M\times M \arrow[l,shift left=2] \arrow[l] \arrow[l,shift right=2]  & \cdots \arrow[l,shift left=3] \arrow[l,shift right=3]  \arrow[l,shift left] \arrow[l,shift right] \big). 
\end{tikzcd}
\end{equation*}
The essential image of $\mathbb{B}$ is the full subcategory of Segal spaces $X$ with the additional property that $X([0])\simeq *$.  We will sometimes implicitly identify $\mathrm{Mon}(\Spaces)$ with this subcategory of simplicial spaces.  
\end{defnprop}

\begin{defn}\label{defn:L}
Since the full subcategory $\mathrm{Mon}(\Spaces) \subset \Fun(\Delta^{\op},\Spaces)$ is closed under limits and filtered colimits, the functor $\mathbb{B}$ admits a left adjoint \cite[Corollary 5.5.2.9]{HTT}, which we denote by $$\mathbb{L} :\Fun(\Delta^{\mathrm{op}},\Spaces) \to \mathrm{Mon}(\Spaces).$$
\end{defn}

The main object of study in this section is:

\begin{defn}\label{defn:Kpart}
Let $\C$ be a Waldhausen $\infty$-category.  Then the \emph{partial algebraic K-theory} of $\C$ is the $\E_1$-monoidal space $$K^{\mathrm{part}}(\C) := \mathbb{L}(S_{\bullet}(\C)^{\simeq}).$$
\end{defn}

It can be helpful to rephrase this definition in terms of the notion of complete Segal spaces, which we now very briefly recall.  

\begin{rec}\label{rec:css}
The $\infty$-category of small $\infty$-categories $\Cat_{\infty}$ can be identified with a full subcategory $$\mathrm{CplSeg}(\Spaces)\subset \mathrm{Seg}(\Spaces)$$ of Segal spaces known as \emph{complete} Segal spaces \cite[Corollary 4.3.16]{LurGood} and due to \cite{Rezk}.  Via this identification, the inclusion $\Cat_{\infty} \subset \Fun(\Delta^{\op},\Spaces)$ admits a left adjoint, which we denote by $$\mathrm{CSS}: \Fun(\Delta^{\op},\Spaces) \to \Cat_{\infty}.$$ 
The functor $\mathrm{CSS}$ can be described as the unique colimit preserving functor sending the representable simplicial space corresponding to $[n]\in \Delta^{\op}$ to the standard $n$-simplex as an $\infty$-category.  

Just as the subcategory $\mathrm{CplSeg}(\Spaces)\subset \Fun(\Delta^{\op},\Spaces)$ of complete Segal spaces can be identified with $\Cat_{\infty}$, the $\infty$-category $\Seg(\Spaces)\subset \Fun(\Delta^{\op},\Spaces)$ can be identified with the $\infty$-category $\Cat_{\infty}^{\mathrm{fl}}$ of \emph{flagged $\infty$-categories} \cite[Theorem 0.26]{AFflag}.   A flagged $\infty$-category is a triple $(\C, X, f)$ consisting of an $\infty$-category $\C$, a space $X$, and an essentially surjective functor $f:X\to \C$.  A Segal space $Y_{\bullet}$ determines a flagged $\infty$-category via the canonical functor $Y_0 \to \mathrm{CSS}(Y_{\bullet})$.  Under this identification, the full subcategory $\CplSeg(\Spaces)\subset \Seg(\Spaces)$ corresponds to the full subcategory of flagged $\infty$-categories $(\C, X, f)$ with the property that $f$ induces an equivalence of spaces $X \simeq \C^{\simeq}$.  The left adjoint to this inclusion corresponds to the forgetful functor $\Cat_{\infty}^{\mathrm{fl}} \to \Cat_{\infty}$ given by $(\C,X,f) \mapsto \C$.  
\end{rec}

\begin{rmk}\label{rmk:connectedinftycat}
Combining Definition/Proposition \ref{defnprop:monoid} with Recollection \ref{rec:css}, we obtain an equivalence of $\infty$-categories between $\Mon(\Spaces)$ and the $\infty$-category $(\Cat_{\infty})_0$ of $\infty$-categories equipped with an essential surjection from a point.  This equivalence can be thought of as sending an $\E_1$-monoid $M$ to the $\infty$-category $BM$ with one object whose space of endomorphisms is $M$.  

Since the above constructions are all compatible with finite products, this equivalence also lifts to an equivalence, for any $n\geq 1$, between $\E_n$-monoids in spaces and $\E_{n-1}$-monoidal $\infty$-categories with the property that the inclusion of the unit is an essential surjection.
\end{rmk}

Using this language, Definition \ref{defn:Kpart} can be rephrased as follows:

\begin{prop}\label{prop:cssinterp}
Let $\C$ be a Waldhausen $\infty$-category.  Then there is an equivalence of $\infty$-categories $$\mathrm{CSS}(S_{\bullet}(\C)^{\simeq}) \simeq BK^{\mathrm{part}}(\C).$$

\end{prop}

In other words, just as $BK(\C)$ is the underlying space of $S_{\bullet}(\C)^{\simeq}$, the $\infty$-category $BK^{\mathrm{part}}(\C)$ is the ``underlying $\infty$-category" of $S_{\bullet}(\C)^{\simeq}$.  

\begin{proof}
\begin{notation*}
In the course of this proof, if $f$ is a fully faithful functor, we will denote by $f^L$ (resp. $f^R$) its left (resp. right) adjoint, provided it exists. 
\end{notation*}

Let $\Fun_*(\Delta^{\op},\Spaces) \subset \Fun(\Delta^{\op},\Spaces)$ be the full subcategory of simplicial spaces $X_{\bullet}$ such that $X_0 \simeq *$.  Then, the inclusion  $\Mon(\Spaces) \subset \Fun(\Delta^{\op},\Spaces)$ factors through an inclusion $k: \Mon(\Spaces) \hookrightarrow \Fun_*(\Delta^{\op},\Spaces)$.  Because $S_{\bullet}(\C)^{\simeq}$ is in the full subcategory $\Fun_*(\Delta^{\op},\Spaces)$, there is an equivalence 
\begin{equation}\label{eqn:css0}
\mathbb{L} S_{\bullet}(\C)^{\simeq} \simeq k^L S_{\bullet}(\C)^{\simeq}.
\end{equation}

Since $*\in \Fun(\Delta^{\op},\Spaces)$ is left Kan extended from its value at $[0]$, it is an initial object in $\Fun_*(\Delta^{\op},\Spaces)$.  It follows that there is a fully faithful functor embedding $$i_0:\Fun_*(\Delta^{\op},\Spaces) \to \Fun(\Delta^{\op},\Spaces)_{*/}.$$  This extends to a commutative diagram of fully faithful functors 
\begin{equation*}
\begin{tikzcd}
\mathrm{CplSeg}(\Spaces)_{*/} \arrow[r,"j_1", hook] & \mathrm{Seg}(\Spaces)_{*/} \arrow[r,"j_0",hook] & \Fun(\Delta^{\op},\Spaces)_{*/}\\
& \Mon(\Spaces)\arrow[u,"i_1",hook]\arrow[r,"k",hook] & \Fun_*(\Delta^{\op},\Spaces)\arrow[u,"i_0",hook].
\end{tikzcd}
\end{equation*}

We note that $i_0$ admits a right adjoint $i_0^R$ which extracts the simplices which only involve the given zero simplex.  It is given on a simplicial space $T_{\bullet}$ by the formula $$(i_0^R T_{\bullet})_{n} \simeq T_n \times_{(T_0)^{\times n+1}} *,$$ where $T_n\to T_0^{\times n+1}$ are the vertex maps.  It is immediate from this formula that $i_0^R$ takes the full subcategory $\mathrm{Seg}(\Spaces)_{*/}\subset \Fun(\Delta^{\op},\Spaces)_{*/}$ to the full subcategory $\Mon(\Spaces)\subset \Fun_*(\Delta^{\op}, \Spaces)$, so $i_1$ admits a right adjoint $i_1^R$ such that $i_0^R j_0 \simeq k i_1^R$.  It follows by taking left adjoints that $j_0^L i_0 \simeq i_1 k^L$.  In particular, $j_0^L i_0 (S_{\bullet}(\C))$ is a Segal space whose zeroth space is contractible (since it is in the image of $i_1$); it follows from Remark \ref{rmk:connectedinftycat} that $j_1^L j_0^L i_0(S_{\bullet}(\C)^{\simeq}) = (j_0 j_1)^L i_0 (S_{\bullet}(\C)^{\simeq})$, which is $\mathrm{CSS}(S_{\bullet}(\C)^{\simeq})$ with the canonical basepoint, is an $\infty$-category of the form $BM$ for a monoid $M$ in spaces.  But $M$ is exactly $k^L(S_{\bullet}(\C)^{\simeq})$, since the functor $i_1^R j_1$ takes $BM$ to $M$ and 
\begin{equation*}\label{eqn:css1}
k^L(S_{\bullet}(\C)^{\simeq}) \simeq i_1^R i_1 k^L (S_{\bullet}(\C)^{\simeq})\simeq i_1^R j_0^L i_0 (S_{\bullet}(\C)^{\simeq}) \simeq i_1^R j_1 j_1^L j_0^L i_0 (S_{\bullet}(\C)^{\simeq}).
\end{equation*}
\end{proof}

One consequence of Proposition \ref{prop:cssinterp} is that although partial $K$-theory has a universal property as an $\E_1$-space, it naturally admits the structure of an $\E_{\infty}$-space.

\begin{lem}\label{lem:cssmonoidal}
The functor $\mathrm{CSS}: \Fun(\Delta^{\op},\Spaces) \to \Cat_{\infty}$ commutes with finite products.
\end{lem}
\begin{proof}
We would like to show that for simplicial spaces $X$ and $Y$, the natural map
$$\mathrm{CSS}(X\times Y) \to \mathrm{CSS}(X) \times \mathrm{CSS}(Y)$$ is an equivalence.  We observe that products preserve colimits separately in each variable in both $\Fun(\Delta^{\op},\Spaces)$ and $\Cat_{\infty},$ and the functor $\mathrm{CSS}$ preserves colimits.  Consequently, it suffices to check this statement on representable objects in $\Fun(\Delta^{\op},\Spaces)$.  But each representable simplicial space $\Delta^n$ is in the image of the fully faithful right adjoint $\Cat_{\infty} \subset \Fun(\Delta^{\op},\Spaces),$ which clearly preserves products, so the conclusion follows.  
\end{proof}

Since the coproduct on $\mathcal{C}$ endows $S_{\bullet}(\C)^{\simeq}$ with the structure of an $\E_{\infty}$-monoid in simplicial spaces, we may use Lemma \ref{lem:cssmonoidal} to equip $\mathrm{CSS}(S_{\bullet}(\C)^{\simeq})$ with the structure of a symmetric monoidal $\infty$-category.  We then have the following two corollaries of Proposition \ref{prop:cssinterp}.  

\begin{cor}\label{cor:partKeinfty}
Let $\C$ be a Waldhausen $\infty$-category.  Then the coproduct on $\C$ endows $K^{\mathrm{part}}(\C)$ with the structure of an $\E_{\infty}$-monoidal space.  
\end{cor}

\begin{cor} \label{cor:kpartgp}
There is a natural equivalence  $K^{\mathrm{part}}(\C)^{\mathrm{gp}} \simeq K(\C)$ of $\E_{\infty}$-spaces.
\end{cor}

\begin{proof}
Note that the functor $\Spaces \to \Fun(\Delta^{\op},\Spaces)$ sending a space to the constant simplicial space factors as a composite $$\Spaces \to \mathrm{CplSeg}(\Spaces) \subset \Fun(\Delta^{\op},\Spaces)$$ through complete Segal spaces.  Taking left adjoints and looping, we obtain an equivalence of $\E_1$-spaces $$(K^{\mathrm{part}}(\C))^{\mathrm{gp}} \simeq \Hom_{|\mathrm{CSS}(S_{\bullet}(\C)^{\simeq})|}(*,*)  \simeq \Omega |\mathrm{CSS}(S_{\bullet}(\C)^{\simeq})| \simeq \Omega |S_{\bullet}(\C)^{\simeq}| \simeq K(\C).$$  We have seen that the relevant left adjoints commute with finite products, so this is in fact an equivalence of $\E_{\infty}$-spaces.  
\end{proof}

As in the case of ordinary $K$-theory, one can explicitly describe $K^{\mathrm{part}}_0(\C)$:

\begin{prop} \label{prop:Kpart0}
Let $\C$ be a Waldhausen $\infty$-category.  Then the monoid $$K^{\mathrm{part}}_0(\C) := \pi_0 K^{\mathrm{part}}(\C)$$ is freely generated by the objects of $\C$ modulo the relation $[A] + [C] = [B]$ for every short exact sequence $0\to A \to B\to C\to 0$.  
\end{prop}

\begin{proof}
Let $\Mon(\Set)$ denote the category of monoids in sets and let  $$\mathbb{B}_0:\Mon(\Set) \to \Fun(\Delta^{\op},\Set)$$  denote the functor which sends a monoid to its bar construction.  Then, $\mathbb{B}_0$ factors through the functor $$i_*: \Fun(\Delta_{\leq 2}^{\op}, \Set) \to \Fun(\Delta^{\op},\Set)$$ given by right Kan extension along the inclusion $i: \Delta_{\leq 2}^{\op} \to \Delta^{\op}$ of the full subcategory spanned by the objects $[0]$, $[1]$, and $[2]$.  We obtain a commutative diagram of right adjoints
\begin{equation*}
\begin{tikzcd}
\Mon(\Spaces) \arrow[r,"\mathbb{B}"] &\Fun(\Delta^{\op},\Spaces) \\
\Mon(\Set) \arrow[u] \arrow[rd,"\overline{\mathbb{B}}_0"']\arrow[r,"\mathbb{B}_0"] & \Fun(\Delta^{\op},\Set)\arrow[u]\\
& \Fun(\Delta^{\op}_{\leq 2},\Set)\arrow[u,"i_*"].
\end{tikzcd}
\end{equation*}

All of the functors in the diagram have left adjoints; the left adjoint of the upper vertical arrows are by taking $\pi_0$, the left adjoint of $i_*$ is the restriction $i^*,$ and we let $\overline{\mathbb{L}}_0$ denote the left adjoint to $\overline{\mathbb{B}}_0$.  Using the commutative diagram of left adjoints, we deduce that for a simplicial space $X$, there is an isomorphism 
$$\pi_0 \mathbb{L} X \cong \overline{\mathbb{L}}_0 i^* \pi_0 X.$$  We would like to compute this monoid in the case that $X = S_{\bullet}(\C)^{\simeq}$ and show that it has the proposed description of $K_0^{\mathrm{part}}(\C)$. 

Let $M_0$ be a monoid in sets, and let us unwind the data of a map $$f:i^* \pi_0 (S_{\bullet}(\C)^{\simeq}) \to \overline{\mathbb{B}}_0(M_0)$$ in $\Fun(\Delta_{\leq 2}^{\op},\Set).$  The map $f$ is determined by three maps 
\begin{align*}
f_{[0]}: \pi_0(S_0(\C)^{\simeq}) &\to * \\
f_{[1]}: \pi_0(S_1(\C)^{\simeq}) &\to M_0 \\
f_{[2]}: \pi_0(S_2(\C)^{\simeq}) &\to M_0\times M_0,
\end{align*}
corresponding to the objects of $\Delta^{\op}_{\leq 2}$.  The map $f_{[0]}$ is no data because $S_0(\C) = *$, and the map $f_{[1]}$ can be thought of as assigning an object of $M_0$ to each equivalence class of object in $\C$.  In order for these maps $f_{[i]}$ to determine a  natural transformation $f$ in $\Fun(\Delta_{\leq 2}^{\op},\Set)$, they must satisfy the following conditions:
\begin{enumerate}
\item Using the compatibilities coming from maps between $[0]$ and $[1]$ in $\Delta^{\op}$, one deduces that $f_{[1]}$ must send the equivalence class of $0\in \C$ to $0\in M_0$.  
\item Using the compatibilities coming from maps between $[1]$ and $[2]$ in $\Delta^{\op}$, one first deduces that $f_{[2]}$ sends the equivalence class of an exact sequence $0\to A \to B \to C \to 0$ to $(f_{[1]}(A),f_{[1]}(C)) \in M_0\times M_0$.  Then, compatibility with the face map corresponding to the multiplication $M_0\times M_0 \to M_0$ imposes the relation $f_{[1]}(A) + f_{[1]}(C) = f_{[1]}(B)$.  
\end{enumerate}
Thus, the data of $f$ is exactly the data of an object in $M_0$ for each nonzero equivalence class of object in $\C$ satisfying the usual additivity condition in exact sequences.  By the Yoneda lemma, we conclude that $\overline{\mathbb{L}}_0 i^* \pi_0 (S_{\bullet}(\C)^{\simeq})$ has the proposed description.  
\end{proof}

\begin{rmk}
In a stable setting, partial $K$-theory does not produce anything new.  For instance, if $\C$ is a stable $\infty$-category, one has for every $X\in \C$ a cofiber sequence $X \to 0 \to \Sigma X$.  It follows from Proposition \ref{prop:Kpart0} that in $K^{\mathrm{part}}_0(\C)$, $[X]$ has an inverse given by $[\Sigma X]$.  Consequently, $K^{\mathrm{part}}_0(\C)$ is group complete and the natural map $K^{\mathrm{part}}(\C) \to K(\C)$ is an equivalence by Corollary \ref{cor:kpartgp}.  
\end{rmk}

\subsection{Partial K-theory via the $Q$-construction}\label{sub:Q}
Throughout this section, let $\C$ be an exact $\infty$-category in the sense of \cite[Definition 1.3]{BarQ}.  Then, one can form an $\infty$-category $Q\C$ known as the \emph{$Q$-construction} on $\C$  \cite[Definition 3.8]{BarQ}.  
\begin{exm}
When $\C = \mathrm{Vect}_{\F_p}$ is the category of finite dimensional $\F_p$-vector spaces, $Q\C$ is equivalent to (the nerve of) the ordinary category whose objects are finite dimensional $\F_p$-vector spaces $V$ and where $\Hom_{Q\mathrm{Vect}_{\F_p}}(U,V)$ is the set of isomorphism classes of spans $U \twoheadleftarrow W \hookrightarrow V$ where the backward arrow is surjective and the forward arrow is injective \cite[Proposition 3.11]{BarQ}.  Composition is given by the usual composition of spans.  
\end{exm}

Quillen \cite{QuillenK} defined the $K$-theory space of $\C$ as $\Omega |Q\C|.$  On the other hand, $\C$ can be regarded as a Waldhausen $\infty$-category and one can consider its $K$-theory in the sense of Definition \ref{defn:Ktheory}.  The following theorem asserts that these constructions agree:

\begin{thm}[\cite{Wald} \S1.9, \cite{BarQ} Proposition 3.7]\label{thm:waldbar}
There is an equivalence of spaces $$|Q\C | \simeq |S_{\bullet}(\C)^{\simeq}|.$$
\end{thm}

We now describe how to make a pre-group-completed variant of this construction.  Thinking of the space $|Q\C|$ as an $\infty$-category, the natural map $Q\C \mapsto |Q\C|$ can be thought of as formally inverting all the morphisms.  On the other hand, the $K$-theory of $\C$ arises as the endomorphisms of the unit object in $|Q\C|$.   Accordingly, to create a version of $K$-theory which is not group complete, one could instead contemplate inverting only some of the morphisms of $Q\C$.  Let $\mathcal{L} \subset \mathrm{Mor}(Q\C)$ denote the backward (\emph{left}-pointing) arrows, i.e., those of the form $X \twoheadleftarrow Y \rightarrow Y$.  The localization $Q\C [\mathcal{L}^{-1}]$ is a symmetric monoidal $\infty$-category which comes with a canonical point represented by $0 \in \C$.  The endomorphisms of $0$ in $Q\C [\mathcal{L}^{-1}]$ is then a monoid in spaces which group completes to $K(\C)$.  The main theorem of this section is that this monoid coincides with the partial $K$-theory of $\C$:

\begin{thm}\label{thm:QequalsS}
Let $\C$ be an exact $\infty$-category.  Then, regarding $\C$ as a Waldhausen $\infty$-category as in \cite[Corollary 4.8.1]{BarHeart}, there is an equivalence of $\E_1$-spaces $$K^{\mathrm{part}}(\C) \simeq \Hom_{Q\C[\mathcal{L}^{-1}]}(0,0).$$
\end{thm} 

This is a categorified analog of the theorems of Waldhausen and Barwick relating $K$-theory via the $S_{\bullet}$-construction to $K$-theory via the $Q$-construction.  We will prove Theorem \ref{thm:QequalsS} in a more precise form as Corollary \ref{cor:QequalsS}.  We begin by reviewing the notion of edgewise subdivision.  


\begin{defn}
Let $\epsilon : \Delta \to \Delta$ be the functor which takes a linearly ordered set $I$ to the join $I^{\op}\star I$.  Given an $\infty$-category $\D$ and a simplicial object $T\in \Fun(\Delta^{\op},\D)$, one can form a new simplicial object $\epsilon^* T$ by precomposition with $\epsilon$, which we will refer to as the \emph{edgewise subdivision} of $T$.  This construction comes equipped with two natural maps induced by the inclusions $I \subset I^{\op} \star I$ and $I^{op}\subset I^{\op}\star I$, which we denote by $\eta_T: \epsilon^* T \to T$ and $\eta^{\op}_T: \epsilon^* T\to T^{\op},$ respectively.  
\end{defn}

\begin{exm}\label{exm:twarr}
When $\D = \mathrm{Set}$ and $T$ is a quasicategory, $\epsilon^* T$ is also a quasicategory and presents the twisted arrow category of $T$ (\cite[Proposition 5.2.1.3]{HA}, beware the opposite convention for morphism direction).
\end{exm}

We will be particularly interested in the following example, which says roughly that Quillen's $Q$-construction arises from Waldhausen's $S_{\bullet}$-construction by edgewise subdivision:
\begin{exm}\label{exm:subdivideS}
When $\D = \Spaces$ and $T = S_{\bullet}(\C)^{\simeq}$ for an exact $\infty$-category $\C$, there is an equivalence of $\infty$-categories $$Q\C \simeq \mathrm{CSS}(\epsilon^* S_{\bullet}(\C)^{\simeq}).$$
This follows from combining \cite[Proposition 3.4]{BarQ} and \cite[Proposition 3.7]{BarQ}.
\end{exm} 

Recall that $K^{\mathrm{part}}(\C)$ is defined as the endomorphisms of the unit object in $\mathrm{CSS}(S_{\bullet}(\C)^{\simeq})$; accordingly, Theorem \ref{thm:QequalsS} asserts a relationship between the simplicial spaces $S_{\bullet}(\C)^{\simeq}$ and $\epsilon^* S_{\bullet}(\C)^{\simeq}$.  In the setting of ordinary $K$-theory (Theorem \ref{thm:waldbar}), one needs to compare these simplicial spaces at the level of geometric realization; this boils down to the classical fact that for any simplicial space $T$, the map $\eta_T$ becomes an equivalence after passing to underlying spaces \cite[A.1]{Seg}.  The proof of Theorem \ref{thm:QequalsS} refines this to a statement about underlying $\infty$-categories.  Namely, instead of passing all the way to underlying spaces, one can study the functor $\eta_T$ after applying $\mathrm{CSS}$.   The main technical result of this section is that while the resulting functor $\mathrm{CSS}(\eta_T): \mathrm{CSS}(\epsilon^*T) \to \mathrm{CSS}(T)$ is not generally an equivalence, it can be described as a localization at a particular collection of morphisms.


\begin{rmk}\label{rmk:JT}
Let $K$ be a simplicial set which is a quasicategory (i.e., fibrant in the Joyal model structure), and let $K_{\Spaces} \in \Fun(\Delta^{\op},\Spaces)$ denote $K$ regarded as a discrete simplicial space.  Then one can consider the $\infty$-category $\mathrm{CSS}(K_{\Spaces})$ obtained by applying the localization $\mathrm{CSS}: \Fun(\Delta^{\op},\Spaces) \to \Cat_{\infty}$.  It follows from \cite[Theorem 4.11]{JT} that there is a natural equivalence of $\infty$-categories $K \simeq \mathrm{CSS}(K_{\Spaces}).$  
\end{rmk}

We first study the case of when $T$ is a standard simplex.  

\begin{exm}\label{exm:subsimplex}
By Example \ref{exm:twarr} and Remark \ref{rmk:JT}, the $\infty$-category $\mathrm{CSS}(\epsilon^* \Delta^n_{\Spaces})$ is the twisted arrow category of $\Delta^n$:

\begin{equation}\label{dia:twdeltan}
\begin{tikzcd}[cramped, sep=small]
 & & & & nn \arrow[d]\\
 & & & \iddots & \vdots\arrow[d] \\
 & & 22\arrow[r]\arrow[d]& \cdots \arrow[r]& 2n\arrow[d]\\
 & 11\arrow[r]\arrow[d]& 12\arrow[r]\arrow[d] & \cdots \arrow[r]& 1n \arrow[d]\\
00 \arrow[r]&01\arrow[r] &02 \arrow[r]& \cdots\arrow[r]& 0n .\\
\end{tikzcd}
\end{equation}
With reference to diagram (\ref{dia:twdeltan}), the functor $$\mathrm{CSS}(\eta_{\Delta^n_{\Spaces}} ): \mathrm{CSS}(\epsilon^*\Delta^n_{\Spaces}) \to \mathrm{CSS}(\Delta^n_{\Spaces}) \simeq \Delta^n$$ projects down to the horizontal axis.  
\end{exm}

\begin{defn}\label{defn:simplexedges}
Let $\mathcal{L}(\Delta^n_{\Spaces})$ denote the subset of the morphisms of $\mathrm{CSS}(\epsilon^* \Delta^n_{\Spaces})$ whose images under $\mathrm{CSS}(\eta_{\Delta^n_{\Spaces}})$ are homotopic to identity morphisms.  These correspond to vertical maps in diagram (\ref{dia:twdeltan}).  
\end{defn}
\begin{lem}\label{lem:simplexloc}
The functor $$\mathrm{CSS}(\eta_{\Delta^n_{\Spaces}} ): \mathrm{CSS}(\epsilon^*\Delta^n_{\Spaces}) \to \mathrm{CSS}(\Delta^n_{\Spaces}) \simeq \Delta^n$$ exhibits $\Delta^n$ as the localization of $\mathrm{CSS}(\epsilon^*\Delta^n_{\Spaces})$ at the collection of morphisms $\mathcal{L}(\Delta^n_{\Spaces})$.
\end{lem}
\begin{proof}
We will refer to diagram (\ref{dia:twdeltan}) of Example \ref{exm:subsimplex}.  By Remark \ref{rmk:JT}, the functor of the lemma is identified with the natural functor of $\infty$-categories $\eta_{\Delta^n}: \epsilon^*\Delta^n \to \Delta^n$.  Regarding $\epsilon^* \Delta^n$ as a marked simplicial set by marking the morphisms in $\mathcal{L}(\Delta^n_{\Spaces})$, it suffices to show that $\eta_{\Delta^n}: \epsilon^*\Delta^n \to (\Delta^n)^{\flat}$ is a weak equivalence in the marked model structure, where $(\Delta^n)^{\flat}$ denotes the simplicial set $\Delta^n$ with only the degenerate edges marked (cf. \cite[Proposition 3.1.3.7]{HTT}, \cite[1.1.3]{Hinich}).  Consider the sequence of maps $$(\Delta^n)^{\flat} \xrightarrow{i} \epsilon^* \Delta^n \xrightarrow{\eta_{\Delta^n}} (\Delta^n)^{\flat}$$ where $i$ includes the bottom edge $00\to 01 \to 02 \to \cdots \to 0n$ of (\ref{dia:twdeltan}). The composite is the identity, so the lemma would be a consequence of knowing that the inclusion $i$ is a marked anodyne extension (and thus a Cartesian equivalence, by \cite[Remark 3.1.3.4]{HTT}).  This follows by using the criterion (2'') of \cite[Proposition 3.1.1.5]{HTT} and stability of marked anodyne morphisms under pushouts \cite[Definition 3.1.1.1]{HTT}.  
\end{proof}

\begin{defn}
Let $T$ be a simplicial space.  With reference to Definition \ref{defn:simplexedges}, define the subset $\mathcal{L}(T)$ of the $1$-simplices of the $\infty$-category $\mathrm{CSS}(\epsilon^* T)$ by $$\mathcal{L}(T) := \{ \gamma \text{ }|\text{ } \gamma = \mathrm{CSS}(\epsilon^* f) \gamma' \text{ for some morphism }f:\Delta^n_{\Spaces} \to T \text{ and edge } \gamma' \in \mathcal{L}(\Delta^n_{\Spaces})\}.$$  
\end{defn}

We can now state and prove the main technical result.  

\begin{prop}\label{prop:subloc}
Let $T$ be a simplicial space.  Then the functor $$\mathrm{CSS}(\epsilon^* \eta_T): \mathrm{CSS}(\epsilon^* T) \to \mathrm{CSS}(T)$$ exhibits $\mathrm{CSS}(T)$ as the localization of $\mathrm{CSS}(\epsilon^* T)$ at the morphisms in $\mathcal{L}(T)$.  
\end{prop}
\begin{proof}
We would like to show that for any $\infty$-category $\mathcal{D}$, the induced map $$\Fun(\mathrm{CSS}(T),\D)\to \Fun(\mathrm{CSS}(\epsilon^* T),\D)$$ is the inclusion of the full subcategory of functors which send the morphisms in $\mathcal{L}(T)$ to equivalences in $\D$.  

Write the simplicial space $T$ as a colimit of representables: $$T = \colim_{\Delta^n_{\Spaces} \to T} \Delta^n_{\Spaces}.$$ Then, since $\epsilon^*$ and $\mathrm{CSS}$ preserve colimits, we have the commutative square
\begin{equation*}
\begin{tikzcd}
\Fun(\mathrm{CSS}(T),\D)\arrow[r] \arrow[d,"\sim" labl]& \Fun(\mathrm{CSS}(\epsilon^* T),\D)\arrow[d,"\sim" labl] \\
\displaystyle{\lim_{\Delta^n_{\Spaces}\to T}} \Fun(\mathrm{CSS}(\Delta^n_{\Spaces}),\D) \arrow[r] & \displaystyle{\lim_{\Delta^n_{\Spaces}\to T}} \Fun( \mathrm{CSS}(\epsilon^* \Delta^n_{\Spaces}), \D). 
\end{tikzcd}
\end{equation*}
By Lemma \ref{lem:simplexloc}, the bottom horizontal arrow is the inclusion of precisely the subcategory of functors which invert $\mathcal{L}(\Delta^n_{\Spaces})$ for each $\Delta^n_{\Spaces}\to T$, and the proposition follows. 
\end{proof}

We apply Proposition \ref{prop:subloc} in the situation of Example \ref{exm:subdivideS}.  Under the equivalence $$Q\C  \simeq \mathrm{CSS}(\epsilon^* S_{\bullet}(\C)^{\simeq}),$$ the subset $\mathcal{L}(S_{\bullet}(\C)^{\simeq})$ of morphisms corresponds to the collection $\mathcal{L}$ of backward arrows of $Q\C$ (i.e., those of the form $X \twoheadleftarrow Y \to Y$).  We therefore have the following corollary of Proposition \ref{prop:subloc}, which completes the proof of Theorem \ref{thm:QequalsS}:

\begin{cor}\label{cor:QequalsS}
There is a functor $Q\C \to \mathrm{CSS}(S_{\bullet}(\C)^{\simeq})$ which extends to an equivalence of $\infty$-categories $$Q\C[\mathcal{L}^{-1}] \simeq \mathrm{CSS}(S_{\bullet}(\C)^{\simeq}),$$ where $\mathcal{L}$ denotes the collection of backward morphisms.  Taking endomorphisms of the zero object on both sides and applying Proposition \ref{prop:cssinterp}, we obtain an equivalence of $\E_1$-spaces $$K^{\mathrm{part}}(\C) \simeq \Hom_{Q\C[\mathcal{L}^{-1}]}(0,0).$$
\end{cor}

\section{The partial $K$-theory of $\F_p$}\label{sect:KpartFp}

We will see in \S \ref{sect:frobstable} that the monoidal $\infty$-category $BK^{\mathrm{part}}(\F_p)$ acts on the $\infty$-category $\CAlg^F_p$ of $F_p$-stable $\E_{\infty}$-rings.  Our goal in this section is to give a computation of $K^{\mathrm{part}}(\F_p)$ up to $p$-completion.  To motivate the result, recall the following theorem of Quillen:

\begin{thm}[Quillen \cite{Quillen}]\label{thm:quillenKFp}
The natural map $$K(\mathbb{F}_p) \to \pi_0 K(\mathbb{F}_p) \cong \Z$$ induces an isomorphism in $\F_p$-homology.  
\end{thm}

In particular, the $\F_p$-homology of $K(\mathbb{F}_p)$ is trivial in positive degrees.  The main result of this section is the analog of Quillen's theorem for partial $K$-theory:
 
\begin{thm}\label{thm:KpartFp}
The natural map $$K^{\mathrm{part}}(\F_p) \to \pi_0 K^{\mathrm{part}}(\F_p) \cong \Z_{\geq 0}$$ induces an isomorphism in $\F_p$-homology.  
\end{thm}

In \S \ref{sub:compute}, we provide a formula for partial $K$-theory in terms of a colimit of spaces in the $S_{\bullet}$-construction.  Then in \S \ref{sub:Fp}, we specialize to the case of $K^{\mathrm{part}}(\F_p)$ and evaluate this formula up to $\F_p$-homology equivalence.


\subsection{Computing partial $K$-theory}\label{sub:compute}

In this section, we give a formula for the functor $\mathbb{L}$ of Definition \ref{defn:L} in terms of a certain colimit.  To do so, we show that $\E_1$-spaces can be presented via an $\infty$-categorical Lawvere theory (Proposition \ref{prop:lawvereass}).  We then reinterpret $\mathbb{B}$ as a certain restriction map and $\mathbb{L}$ as the corresponding left Kan extension.

\begin{defn}\label{defn:TA}
Let $\mathcal{T}_A$ denote the opposite category of the full subcategory of (discrete, associative) monoids spanned by those which are free and finitely generated.  For a finite set $S$, we will denote the free monoid on $S$ by $\mathrm{Free}(S)$.  We will refer to $\mathcal{T}_A$ as the \emph{theory of associative monoids.} 
\end{defn}

\begin{notation}
For a linearly ordered set $S$, let $S^{\pm}$ denote the linearly ordered set $\{ -\infty \} \cup S \cup \{ \infty\}$.  
It will be convenient to think of $\Delta^{\mathrm{op}}$ as the category of (possibly empty) linearly ordered sets $S$, where morphisms from $S$ to $T$ are order preserving maps $S^{\pm} \to T^{\pm}$ preserving $\pm \infty$.  
  This description arises by associating to a finite linearly ordered set in $\Delta$ its corresponding linearly ordered set of ``gaps,'' or pairs of adjacent elements.    We will denote the linearly ordered set of gaps in $[n]\in \Delta$ by $(n)\in \Delta^{\op}$, so that $(n)$ has $n$ elements, which we call $1,2,\cdots,n$, and $(n)$ is always implicitly regarded as an object in $\Delta^{\op}$.  
\end{notation}

\begin{defn}
Define a functor $$F: \Delta^{\mathrm{op}} \to \mathcal{T}_A$$  by $(n) \mapsto \mathrm{Free}((n))$ on objects, and by sending the morphism $f:(n) \to (m)$ in $\Delta^{\mathrm{op}}$ to the morphism $$F(f):\mathrm{Free}((m)) \to \mathrm{Free}((n))$$ which sends the generator corresponding to $i\in (m)$ to the product of the generators corresponding to $f^{-1}(i)$, using the order from $(n)$.  
\end{defn}

We now show the Lawvere theory presentation of monoids suggested by Definition \ref{defn:TA} coincides with the notion of monoid from Definition/Proposition \ref{defnprop:monoid}.

\begin{prop}\label{prop:lawvereass}
The restriction functor $F^*: \Fun(\mathcal{T}_A,\Spaces) \to \Fun(\Delta^{\mathrm{op}},\Spaces)$ sends the full subcategory $\Fun^{\times}(\mathcal{T}_A,\Spaces)\subset  \Fun(\mathcal{T}_A,\Spaces)$ of product preserving functors to the full subcategory $\mathrm{Mon}(\Spaces) \subset \Fun(\Delta^{\mathrm{op}},\Spaces)$ of monoids in spaces.  Moreover, the restricted functor $$F^* : \Fun^{\times}(\mathcal{T}_A,\Spaces) \to \mathrm{Mon}(\Spaces) $$ is an equivalence of $\infty$-categories.
\end{prop}
\begin{proof}
The first statement is clear so we focus on the second.  Let $M:\Delta^{\mathrm{op}} \to \Spaces$ be a monoid.  We shall compute the right Kan extension along the functor $F$ and see that it determines an inverse equivalence.

The right Kan extension at $\mathrm{Free}((n))\in \mathcal{T}_A$ is indexed by the category $\Delta^{\op}\times_{\mathcal{T}_A} (\mathcal{T}_A)_{\mathrm{Free}((n))/}$ whose objects are pairs $((m),f)$ where $(m)\in \Delta^{\op}$ and $d: \mathrm{Free}((m)) \to \mathrm{Free}((n))$ is a map of monoids.  We will show that the value of the right Kan extension at $\mathrm{Free}((n))$ is $M^n$.  

We first describe in detail the special case $n=1$.  In this case, the map $d:\mathrm{Free}((m)) \to \mathrm{Free}((1))\simeq \Z_{\geq 0}$ assigns a nonnegative integer, which we can think of as a ``degree,'' to each of the generators of $\mathrm{Free}((m))$.  The morphisms in $\Delta^{\op}\times_{\mathcal{T}_A} (\mathcal{T}_A)_{\mathrm{Free}((n))/}$ from $((m_1),d_1)$ to $((m_2),d_2)$ are maps $f:(m_1) \to (m_2)$ in $\Delta^{\op}$ such that $F(f)$ preserves the degree.   

\begin{obs*}
Consider the full subcategory $\mathcal{J}_1 \subset \Delta^{\op}\times_{\mathcal{T}_A} (\mathcal{T}_A)_{\mathrm{Free}((1))/}$ spanned by those $((m),d)$ with the property that $d$ sends every generator of $\mathrm{Free}((m))$ to $1$.  Then the inclusion of $\mathcal{J}_1$ admits a right adjoint.
Explicitly, let $x_1, x_2, \cdots , x_m \in \mathrm{Free}((m))$ denote the generators and $d:\mathrm{Free}((m)) \to \Z_{\geq 0}$ be any map of monoids; then the right adjoint sends $((m),d)$ to the unique object in $\mathcal{J}_1$ whose underlying monoid is $\mathrm{Free}((d(x_1)+d(x_2)+\cdots +d(x_m)))$.  
\end{obs*}

It follows that the subcategory $\mathcal{J}_1$ is coinitial.  One can (and we will) view $\mathcal{J}_1$ in an alternate way, as the wide subcategory of $\Delta^{\op}$ whose morphisms from $(m_1)$ to $(m_2)$ are those determined by maps of linearly ordered sets $(m_1)^{\pm} \to (m_2)^{\pm}$ which are isomorphisms when restricted to the preimage of $(m_2)\subset (m_2)^{\pm}$.  With this notation, the value of the desired right Kan extension is computed as $$\lim_{(m) \in \mathcal{J}_1} M^{m}.$$  This diagram of spaces is right Kan extended from the full subcategory of $\mathcal{J}_1$ spanned by the object $(1)\in \mathcal{J}_1$, and so the limit is given simply by $M$, as desired.

For a general $n$, one can think of the function $d:\mathrm{Free}((m)) \to \mathrm{Free}((n))$ as assigning a ``generalized degree'' which takes values in $\mathrm{Free}((n))$ instead of just $\Z_{\geq 0} = \mathrm{Free}((1)).$  One then extracts the coinitial subcategory $\mathcal{J}_m$ determined by those $d$ which assign to each generator an element of ``generalized degree one'' in the sense that they send generators of $\mathrm{Free}((m))$ to generators of $\mathrm{Free}((n))$.  The resulting diagram is right Kan extended from the full subcategory of $\mathcal{J}_m$ spanned by objects whose underlying monoid is free on one generator; there are $n$ of these (one for each generator of $\mathrm{Free}((n))$) and no maps between them, so the value of the limit is $M^{n}$, as desired.  It follows from the above calculation that the resulting functor is product preserving and that the unit and counit maps induce equivalences.    
\end{proof}


Under the identification of Proposition \ref{prop:lawvereass}, the composite $$\mathrm{Mon}(S) \simeq \Fun^{\times}(\mathcal{T}_A, \Spaces) \to \Fun(\Delta^{\mathrm{op}},\Spaces)$$ is exactly the functor $\mathbb{B}$ of Definition/Proposition \ref{defnprop:monoid}.  But the target of $\mathbb{B}$ also admits a description as an $\infty$-category of \emph{product preserving} functors.  

\begin{notation}
Let $(\Delta^{\mathrm{op}})^{\times}$ denote the free product completion of $\Delta^{\op}$.  It has the universal property that for any $\infty$-category $\C$ which has finite products, restriction along the inclusion $\Delta^{\op} \to (\Delta^{\op})^{\times}$ induces an equivalence of $\infty$-categories $\Fun^{\times}((\Delta^{\mathrm{op}})^{\times},\C) \simeq \Fun(\Delta^{\mathrm{op}},\C). $  

Explicitly, the objects of $(\Delta^{\mathrm{op}})^{\times}$ are finite sets $S$ together with a collection $\{ (n_s)\}_{s\in S}$ of finite linearly ordered sets indexed by $S$.  A morphism from $\{ (n_s) \}_{s \in S}$ to $\{ (m_t) \}_{t\in T}$ is the data of a function $\phi :T \to S$ together with a morphism $(n_{\phi(t)}) \to (m_{t})$ in $\Delta^{\op}$ for each $t\in T$.  
\end{notation}


Since the category $\mathcal{T}_A$ admits finite products, the functor $F: \Delta^{\mathrm{op}} \to \mathcal{T}_A$ extends uniquely to a product preserving functor $F^{\times}: (\Delta^{\mathrm{op}})^{\times} \to \mathcal{T}_A.$  Under this identification, $\mathbb{B}$ is given by restriction along $F^{\times}$: $$\mathbb{B} = (F^{\times})^*: \Fun^{\times}((\Delta^{\mathrm{op}})^{\times},\Spaces) \to \Fun^{\times}(\mathcal{T}_A,\Spaces).$$  Its left adjoint $\mathbb{L}$ will be given by left Kan extension, and will be computed by a colimit over a certain indexing category that we now discuss:

\begin{defn}
We define a category $\mathcal{I}$ as follows.  The objects of $\mathcal{I}$ are tuples $(I,J, f)$ where $I$ and $J$ are (possibly empty) finite linearly ordered sets and $f:I \twoheadrightarrow J$ is an order preserving surjection.  We think of $f$ as partitioning $I$ into convex subsets indexed by $J$.  As such, if we denote the elements of $J$ by $j_1 < j_2 <\cdots < j_{m}$, then we will refer to a typical element $(I,J,f)\in \mathcal{I}$ by $(f^{-1}j_1)(f^{-1}j_2)\cdots(f^{-1}j_{m})$, or more informally by $(|f^{-1}(j_1)|)(|f^{-1}(j_2)|)\cdots(|f^{-1}(j_{m})|).$

The morphisms in $\mathcal{I}$ from $(I,J,f)$ to $(I',J',f')$ are given by commuting squares of order preserving maps:
\begin{equation*}
\begin{tikzcd}
I \arrow[r]\arrow[d, twoheadrightarrow,"f"] & I'\arrow[d,twoheadrightarrow,"f'"]\\
J & J', \arrow[l]\\
\end{tikzcd}
\end{equation*}
which we think of as a map of linearly ordered sets $I\to I'$ that refines the partitions.  
\end{defn}

As promised, $\mathcal{I}$ indexes a colimit which computes $\mathbb{L}$.  

\begin{cnstr}\label{cnstr:bolded}
Note that $\mathcal{I}$ admits a functor $\mathcal{I} \to (\Delta^{\op})^{\times}$ sending $(n_1)(n_2)\cdots (n_k) \in \mathcal{I} $ to $\{ (n_1),(n_2),\cdots ,(n_k)\} \in (\Delta^{\op})^{\times}.$  We associate to each simplicial space $X_{\bullet}$ the functor $\mathbf{X}^{\times}: \mathcal{I} \to \Spaces$ given by the composite $$\mathcal{I} \to (\Delta^{\op})^{\times} \xrightarrow{X^{\times}_{\bullet}} \Spaces,$$ where $X^{\times}_{\bullet}: (\Delta^{\op})^{\times}\to \Spaces$ denotes the functor induced by $X_{\bullet}$.  
\end{cnstr}

\begin{prop}\label{prop:Badjoint}
Let $X_{\bullet}$ be a simplicial space. Then the underlying space of the monoid $\mathbb{L}X_{\bullet}$ is given by the formula $$\mathbb{L} X_{\bullet} \simeq \colim_{\mathcal{I}} \mathbf{X}^{\times} \simeq \colim_{(n_1)\cdots(n_i) \in \mathcal{I}} X_{n_1}\times \cdots \times X_{n_i}.$$  
\end{prop}

\begin{proof}
Since $\mathbb{B}$ is restriction along the functor $F^{\times}:(\Delta^{\op})^{\times} \to \mathcal{T}_A$, it suffices to show that left Kan extension along $F^{\times}$ takes product preserving functors $(\Delta^{\op})^{\times} \to \Spaces$ to product preserving functors $\mathcal{T}_A \to \Spaces$, and that the Kan extension is given by the desired formula.  

We first compute the value of the left Kan extension at $\Z_{\geq 0} \in \mathcal{T}_A$, which is indexed by the category $\mathcal{K} := (\Delta^{\mathrm{op}})^{\times}\times_{\mathcal{T}_A} (\mathcal{T}_A)_{/\Z_{\geq 0}}.$  Unwinding the definitions, an object of $\mathcal{K}$ is a pair $$\big(\{ (n_1),(n_2),\cdots ,(n_k)\}, x\big)$$ where $\{ (n_1),(n_2),\cdots ,(n_k)\} \in (\Delta^{\mathrm{op}})^{\times}$ and $x$ is a chosen element in $\mathrm{Free}(\coprod_i(n_i))$.  The morphisms in $\mathcal{K}$ are maps in $(\Delta^{\mathrm{op}})^{\times}$ preserving this element. 

\begin{notation*}
Inside $\mathrm{Free}(\coprod_i(n_i))$, we will denote by $x_{1}^{(i)}, x_{2}^{(i)}, \cdots , x_{n_i}^{(i)}$ the generators corresponding to $(n_i)$.
\end{notation*}

We define a functor $G:\mathcal{I} \to \mathcal{K}$ as follows.  On objects, it sends $(n_1)(n_2)\cdots (n_k) \in \mathcal{I}$ to $\{ (n_1),(n_2),\cdots ,(n_k) \} \in (\Delta^{\mathrm{op}})^{\times}$ together with the element given by the product 
$$(x_1^{(1)}x_2^{(1)}\cdots x_{n_1}^{(1)})(x_1^{(2)}x_2^{(2)}\cdots x_{n_2}^{(2)}) \cdots (x_1^{(k)} x_2^{(k)}\cdots x_{n_k}^{(k)}).$$  A morphism $$(n_1)\cdots (n_k) \to (m_1)\cdots (m_j)$$ in $\mathcal{I}$ determines a canonical morphism $$\{ (n_1),\cdots ,(n_k) \} \to \{(m_1),\cdots ,(m_j)\}$$ in $(\Delta^{\op})^{\times}$ and it  preserves the chosen elements of the corresponding free monoids because the morphisms in $\mathcal{I}$ are assumed to be order preserving.

\begin{lem}
The functor $G$ admits a left adjoint $H:\mathcal{K} \to \mathcal{I}$ and therefore $G$ is cofinal.  
\end{lem}
\begin{proof}[Proof of lemma]
We begin by defining $H$.  Consider an object $$\big(\{ (n_1),(n_2),\cdots ,(n_k)\}, x\big)\in \mathcal{K}.$$    The element $x$ is given by some word in the generators $x_{j}^{(i)}$.  We call a convex substring of $x$  (i.e., set of adjacent letters) \emph{snug} if it is of the form $x_{j}^{(i)} x_{j+1}^{(i)} \cdots x_{j+m-1}^{(i)}$ for some positive integers $(i,j,m)$. 
A snug substring of $x$ is called \emph{maximal} if it is not a proper substring of a larger snug substring.  Consider the unique partition of $x$ into maximal snug substrings
$$(x_{j_1}^{(i_1)} x_{j_1+1}^{(i_1)} \cdots x_{j_1+m_1-1}^{(i_1)}) (x_{j_2}^{(i_2)} x_{j_2+1}^{(i_2)} \cdots x_{j_2+m_2-1}^{(i_2)})\cdots (x_{j_l}^{(i_l)} x_{j_l+1}^{(i_l)} \cdots x_{j_l+m_l-1}^{(i_l)}).$$  We regard this partition as an object $x^{\mathcal{I}} \in \mathcal{I}$ (which is isomorphic to $(m_1)(m_2)\cdots (m_l)\in\mathcal{I}$), and define $H$ on objects by  $$H\big(\{ (n_1),(n_2),\cdots ,(n_k)\}, x\big) = x^{\mathcal{I}}.$$
\begin{exm*}
If $x = x_1^{(2)}x_2^{(1)}x_3^{(1)}x_4^{(1)}x_2^{(3)}x_1^{(3)}x_2^{(3)}$, then the partition into maximal snug substrings is $x = (x_1^{(2)})(x_2^{(1)}x_3^{(1)}x_4^{(1)})(x_2^{(3)})(x_1^{(3)}x_2^{(3)})$, and $H$ would send this to the object $(1)(3)(1)(2) \in \mathcal{I}$.  
\end{exm*}

We now describe the effect of $H$ on morphisms.  Consider a morphism  $$\big(\{ (n_1),(n_2),\cdots ,(n_s) \} ,x\big) \to \big(\{ (m_1),(m_2),\cdots ,(m_t)\}, y\big)$$ in $\mathcal{K}$, and suppose $x$ and $y$ are written as a union of maximal snug substrings as follows:
$$x= (x_{j_1}^{(i_1)} x_{j_1+1}^{(i_1)} \cdots x_{j_1+h_1-1}^{(i_1)}) (x_{j_2}^{(i_2)} x_{j_2+1}^{(i_2)} \cdots x_{j_2+h_2-1}^{(i_2)})\cdots (x_{j_l}^{(i_l)} x_{j_l+1}^{(i_l)} \cdots x_{j_l+h_l-1}^{(i_l)})$$ 
$$y=(y_{g_1}^{(e_1)} y_{g_1+1}^{(e_1)}\cdots y_{g_1+f_1-1}^{(e_1)})(y_{g_2}^{(e_2)}y_{g_2+1}^{(e_2)}\cdots y_{g_2+f_2-1}^{(e_2)}) \cdots (y_{g_r}^{(e_r)} y_{g_r+1}^{(e_r)}\cdots y_{g_r+f_r-1}^{(e_r)})$$(where the generators $y_{\bullet}^{(\bullet)}$ are defined analogously to the generators $x_{\bullet}^{(\bullet)}$).  
Such a morphism is determined by a map of sets $\gamma: \{ 1, \cdots , t\} \to \{ 1 ,\cdots, s\}$ together with maps $(n_{\gamma(\alpha)}) \to (m_{\alpha})$ in $\Delta^{\mathrm{op}}$ for each $1\leq \alpha \leq t$ such that the resulting map $\theta: \mathrm{Free}(\coprod_{\alpha} (m_{\alpha})) \to \mathrm{Free}(\coprod_{\beta} (n_{\beta}))$ sends $y$ to $x$.  We define the corresponding map $x^{\mathcal{I}} \to y^{\mathcal{I}}$ by sending the element of $x^{\mathcal{I}}$ corresponding to $x^{(i)}_j$ to the element  $y^{(e)}_g$ of $y^{\mathcal{I}}$  that hits it under $\theta$.  This refines the partitions because under $\theta$, each maximal snug substring of $y$ is sent to a (possibly empty, but not necessarily maximal) snug substring of $x$.   


It suffices now to exhibit appropriate unit and counit transformations.  It is immediate that $H\circ G$ is naturally the identity, so we will define the unit transformation.  Consider the object $(\{ (n_1),(n_2),\cdots ,(n_k) \}, x)\in \mathcal{K}$, where $$x= (x_{j_1}^{(i_1)} x_{j_1+1}^{(i_1)} \cdots x_{j_1+m_1-1}^{(i_1)}) (x_{j_2}^{(i_2)} x_{j_2+1}^{(i_2)} \cdots x_{j_2+m_2-1}^{(i_2)})\cdots (x_{j_l}^{(i_l)} x_{j_l+1}^{(i_l)} \cdots x_{j_l+m_l-1}^{(i_l)})$$ is a partition into maximal snug substrings. The functor $G\circ H$ sends this to the object $\big(\{ (m_1),(m_2),\cdots ,(m_l)\},y\big)\in \mathcal{K}$ where  $$y=(y_1^{(1)} y_2^{(1)}\cdots y_{m_1}^{(1)})(y_1^{(2)}y_2^{(2)}\cdots y_{m_2}^{(2)}) \cdots (y_1^{(l)} y_2^{(l)}\cdots y_{m_l}^{(l)})$$
($y_{\bullet}^{(\bullet)}$ as above).  We then define the unit natural transformation on this object, which is the data of a map $$\big(\{ (n_1),(n_2),\cdots ,(n_s) \} ,x\big) \to \big(\{ (m_1),(m_2),\cdots ,(m_t)\}, y\big)$$ in $\mathcal{K}$, by sending $r\in \{1, \cdots , t\}$ to $i_r \in \{1, \cdots , s\}$ and mapping $(n_{i_r})$ to $(m_r)$ in the unique way such that the generators $y_1^{(r)},\cdots , y_{m_r}^{(r)}$ are sent to $x_{j_r}^{(i_r)}, \cdots , x_{j_r + m_r -1}^{(i_r)},$ respectively.  It is easy to check that the triangle identities are satisfied and thus we have produced the adjoint.  
\end{proof}

This lemma implies that the value of the left Kan extension on $\Z_{\geq 0}$ is given by the desired formula.  It suffices to show that the resulting Kan extended functor is product preserving.  However, the argument above generalizes in a straightforward way to show that for any finite set $S$, there is a cofinal functor $\mathcal{I}^S \to (\Delta^{\mathrm{op}})^{\times}\times_{\mathcal{T}_A} (\mathcal{T}_A)_{/\mathrm{Free}(S)} .$  This implies the Kan extended functor is product preserving because
\begin{align*}
\colim_{\mathcal{I}\times \mathcal{I}} (X_{n_1}\times X_{n_2} &\times \cdots \times X_{n_j}) \times (X_{m_1}\times X_{m_2} \times \cdots \times X_{m_k}) \\
&\simeq (\colim_{\mathcal{I}}X_{n_1}\times X_{n_2} \times \cdots \times X_{n_j})\times (\colim_{\mathcal{I}}X_{m_1}\times X_{m_2} \times \cdots \times X_{m_k})
\end{align*}
where $((n_1)(n_2)\cdots (n_j),(m_1)(m_2)\cdots (m_k)) \in \mathcal{I}\times \mathcal{I}$, where we have used that colimits are universal in spaces.  
\end{proof}

\subsection{The partial $K$-theory of $\mathbb{F}_p$}\label{sub:Fp}

We now turn to the proof of Theorem \ref{thm:KpartFp}.  

\begin{notation*}
In the course of this proof, we set $S_{\bullet} = S_{\bullet}(\mathrm{Vect}^{\mathrm{fd}}_{\F_p})$.  All vector spaces will implicitly be over $\F_p$.
\end{notation*}

The proof will proceed by applying the formula of Proposition~\ref{prop:Badjoint} in the case $X_{\bullet} = S_{\bullet}$.  We aim to compute the colimit 
\begin{equation}\label{eqn:KpartFpcolim}
\colim_{(n_1)\cdots(n_j)\in \mathcal{I}}S_{n_1}\times S_{n_2} \times \cdots \times S_{n_j}
\end{equation}
of the functor $\mathbf{S}^{\times}:\mathcal{I}\to \Spaces$ (cf. Construction \ref{cnstr:bolded}).  Let $\mathcal{Z} \to \Spaces$ denote the universal left fibration, and let $\mathcal{Z}_0 \to \Spaces$ denote the left fibration classifying the functor $\pi_0: \Spaces \to \Spaces$.  The natural transformation $\id_{\Spaces} \to \pi_0$ induces a map $p: \mathcal{Z} \to \mathcal{Z}_0$ of fibrations.  Define $\infty$-categories $\mathcal{X}$ and $\mathcal{F}$ and functors $G$ and $H$ so that the squares in the following diagram are Cartesian:

\begin{equation*}
\begin{tikzcd}
\mathcal{X} \arrow[r,"G"] \arrow[d] & \mathcal{F} \arrow[r,"H"]\arrow[d] & \mathcal{I} \arrow[d,"\mathbf{S}^{\times}"] \\
\mathcal{Z} \arrow[r,"p"] & \mathcal{Z}_0 \arrow[r] & \Spaces .
\end{tikzcd}
\end{equation*}

The composite $H\circ G$ is a left fibration because it arises as the pullback of the universal left fibration.  It follows from \cite[Corollary 3.3.4.6]{HTT} that the colimit of $\mathbf{S}^{\times}$ is homotopy equivalent to $|\mathcal{X}|$.  On the other hand, the map $p$ is also a left fibration; explicitly, an object of $\mathcal{Z}_0$ is a space $X$ together with an element $a\in \pi_0(X)$, and $p$ classifies the functor that sends this object to the component of $X$ containing $a$.  It follows that $G$ is itself a left fibration.  Let $\mathbf{S}: \mathcal{F} \to \Spaces$ be the functor that classifies $G$ (the reason for this notation will soon be clear).  Applying \cite[Corollary 3.3.4.6]{HTT} again, we see that the colimit of $\mathbf{S}$ is also homotopy equivalent to $|\mathcal{X}|$, and so we are reduced to computing the colimit of $\mathbf{S}$.  We turn to analyzing this colimit, starting with obtaining a more explicit description of $\mathcal{F}$.  

\begin{defn}\label{defn:filtdim}
Define a \emph{filtered dimension} to be a nonempty sequence $\mathbf{d}= \langle d_1,d_2,\cdots ,d_k \rangle$ of nonnegative integers.   We will soon think of $\mathbf{d}$ as keeping track of the dimensions of the successive quotients in the isomorphism class of filtered vector spaces determined by $$* \subset \F_p^{d_1} \subset \F_p^{d_1 +d_2} \subset \cdots \subset \F_p^{d_1 +\cdots + d_k}.$$ As such, we define its \emph{length} to be $l(\mathbf{d})=k$ and its \emph{dimension} to be $|\mathbf{d}| = d_1 + d_2 +\cdots +d_k$.  

\end{defn}

By definition, an object of $\mathcal{F}$ is an object $(n_1)(n_2)\cdots (n_k) \in \mathcal{I}$ together with a choice of element of $\pi_0(S_{n_1} \times S_{n_2} \times \cdots S_{n_k})$. 
Unwinding the definitions, we see that this is the data of a (possibly empty) sequence $D = (\mathbf{d}^{(1)},\mathbf{d}^{(2)},\cdots ,\mathbf{d}^{(j)})$ of filtered dimensions; we call such a sequence a \emph{filtered dimension sequence} and define $l(D) := \Sigma_i l(\mathbf{d}^{(i)})$ and $|D| := \Sigma_i |\mathbf{d}^{(i)}|$.  Explicitly, this correspondence between filtered dimension sequences and objects of $\mathcal{F}$ sends the filtered dimension sequence $D$ to $(l(\mathbf{d}^{(1)}))(l(\mathbf{d}^{(2)}))\cdots (l(\mathbf{d}^{(j)})) \in \mathcal{I}$ together with the unique point in $\pi_0(S_{l(\mathbf{d}^{(1)})} \times \cdots \times S_{l(\mathbf{d}^{(j)})})$ whose image in $\pi_0(S_{l(\mathbf{d}^{(i)})})$ is the isomorphism class of filtered vector space determined by $\mathbf{d}^{(i)}$ for all $i$.  We can then express the functor $\mathbf{S} : \mathcal{F} \to \Spaces$ classifying the fibration $G$ by the formula $$\mathbf{S}(D) =S_{l(\mathbf{d}^{(1)})} \times \cdots \times S_{l(\mathbf{d}^{(j)})}.$$

Next, we simplify the calculation of $\displaystyle{\colim_{\mathcal{F}} \mathbf{S}}$ by extracting a cofinal subcategory of $\mathcal{F}$:

\begin{defn}
Let $\mathcal{F}^{\red} \subset \mathcal{F}$ be the full subcategory of filtered dimension sequences which are \emph{reduced} in the sense that each filtered dimension $\mathbf{d}= \langle d_1,d_2,\cdots ,d_k \rangle$ in the sequence satisfies $d_1, d_2, \cdots , d_k \geq 1$.  
\end{defn}

\begin{lem}
The inclusion $\mathcal{F}^{\red}\subset \mathcal{F}$ admits a left adjoint, and is therefore cofinal.
\end{lem}
\begin{proof}
The left adjoint is given by the functor $\red:\mathcal{F} \to \mathcal{F}^{\red}$ which takes a filtered dimension sequence and removes all zeros (and any resulting empty filtered dimensions).  
\end{proof}

We are reduced to computing the colimit of the functor $\mathbf{S}: \mathcal{F}^{\red} \to \Spaces$ which, we remind the reader, has the following two properties:
\begin{enumerate}
\item A filtered dimension sequence consisting of a single filtered dimension $\mathbf{d}= \langle d_1,d_2,\cdots ,d_k \rangle$ is sent to the groupoid $\mathbf{S}((\mathbf{d}))$ of filtered vector spaces with filtered dimension $\mathbf{d}$.  
\item The functor $\mathbf{S}$ takes concatenation of filtered dimension sequences to products of spaces; in other words, for $D =   (\mathbf{d}^{(1)},\mathbf{d}^{(2)},\cdots ,\mathbf{d}^{(j)})\in \mathcal{F}^{\red}$, $$\mathbf{S}(D) = \prod_i \mathbf{S}((\mathbf{d}^{(i)})).$$
\end{enumerate}

\begin{obs}\label{obs:fredposet}
The category $\mathcal{F}^{\red}$ is a poset.  The maps in $\mathcal{F}^{\red}$ are generated under concatenation of filtered dimension sequences by the following two operations:
\begin{enumerate}
\item For a filtered dimension $\mathbf{d}$ of length $k$ and an integer $1\leq i < k$, there is a \emph{collapse} map $$ \langle d_1, d_2,\cdots ,d_k\rangle \to  \langle d_1, \cdots d_{i-1}, d_i+d_{i+1}, d_{i+2},\cdots d_k\rangle .$$
\item For a filtered dimension $\mathbf{d}$ of length $k$ and an integer $0<i<k$, there is a \emph{splitting} map $$ \langle d_1, d_2,\cdots ,d_k\rangle \to  (\langle d_1, d_2, \cdots , d_i \rangle , \langle d_{i+1} , d_{i+2},\cdots , d_k\rangle).$$
\end{enumerate}
\end{obs}

Since these maps preserve the total dimension $|D|$, we may write $\mathcal{F}^{\red}$ as a disjoint union of posets $$\mathcal{F}^{\red} = \coprod_m \mathcal{F}^{\red}_m$$ where $\mathcal{F}^{\red}_m\subset \mathcal{F}^{\red}$ is the subset spanned by filtered dimension sequences $D$ such that $|D|=m$.  This induces an equivalence:
$$\colim_{D\in \mathcal{F}^{\red}} \mathbf{S}(D) = \coprod_m \colim_{D\in \mathcal{F}^{\red}_m} \mathbf{S}(D)$$ such that the $m$th component corresponds to $m\in \Z_{\geq 0}$ under the isomorphism $\pi_0(K^{\mathrm{part}}(\F_p)) \simeq \Z_{\geq 0}$.   
It remains to show that each $\colim_{\mathcal{F}^{\red}_m}\mathbf{S}(D)$ has trivial $\F_p$-homology above degree zero.  We proceed by induction on $m$.  The case $m=0$ is trivial.  When $m=1$, the category $\mathcal{F}^{\red}_m$ has a single element $(\langle 1 \rangle)$ which is sent to $BGL_1(\F_p)$, which has vanishing $\F_p$-homology.  Now, fix $m\geq 2$ and assume the statement for all $m'<m$.  We analyze the diagram $\mathbf{S}:\mathcal{F}_m^{\mathrm{red}} \to \Spaces$ in detail, starting with a particular subset of it:

\begin{prop}\label{prop:steinberg}
Let $\C_m^+ \subset \mathcal{F}_m^{\mathrm{red}}$ be the subset of filtered dimension sequences of the form $(\mathbf{d})$ for a single filtered dimension $\mathbf{d}$, and let $\C_m \subset \C_m^+$ be the complement of the object $(\langle m\rangle).$  Then the map $$\colim_{D\in \C_m}\mathbf{S}(D) \to \colim_{D\in \C_m^+}\mathbf{S}(D)$$ induces an isomorphism in $\F_p$-homology.  Equivalently, since $(\langle m\rangle)\in \C_m^+$ is a final object, the natural map $$\colim_{D\in \C_m}\mathbf{S}(D) \to \mathbf{S}(\langle m\rangle)$$ induces an isomorphism on $\F_p$-homology.  
\end{prop}
\begin{proof}
Let $\mathcal{T}_m$ denote the poset of nontrivial proper subspaces of the vector space $\F_p^m$ and let $\mathcal{D}$ denote the category of nondegenerate simplices of $\mathcal{T}_m$.   There is a canonical inclusion $\mathcal{T}_m \hookrightarrow \mathcal{D}$ which \ is equivariant for the natural action of $G=GL_m(\F_p)$.  It therefore induces a homotopy equivalence $|\mathcal{T}_m \sslash G| \simeq |\mathcal{D} \sslash G|$.

On the other hand, note that an object of $\mathcal{D}$ is a flag of strict inclusions of nontrivial proper subspaces of $\F_p^m$, so there is a natural map $\mathcal{D} \to \C_m$ which is $G$-equivariant (for the trivial action on $\mathcal{C}_m$).  Moreover, the induced map $\mathcal{D}\sslash G \to \C_m$ is precisely the Grothendieck construction applied to the functor $\mathbf{S}|_{\C_m}: \C_m \to \Spaces$ whose colimit we aim to compute.  Therefore, we have equivalences $$\colim_{D\in\C_m} \mathbf{S}(D) \simeq |\mathcal{D} \sslash G| \simeq |\mathcal{T}_m \sslash G| \simeq |\mathcal{T}_m|_{hG}.$$

By the Solomon-Tits theorem \cite{Sol}, the $G$-space $|\mathcal{T}_m|$ is a wedge of $(m-2)$-dimensional spheres.  Moreover, its top cohomology $H^{m-2}(|\mathcal{T}_m|,\F_p)$ is the Steinberg representation of $G$ over $\F_p$, which is irreducible and projective and therefore has vanishing homology in all degrees \cite{Humphreys}.  It follows from the homotopy orbit spectral sequence that the natural map $|\mathcal{T}_m|_{hG} \to (*)_{hG}$ induces an equivalence in $\F_p$-homology, and so the natural map $$\colim_{D\in \C_m}\mathbf{S}(D) \to \mathbf{S}(\langle m\rangle)$$ induces an isomorphism in $\F_p$-homology as desired.  
\end{proof}

To analyze the rest of the diagram, we define a filtration on $\mathcal{F}_m^{\mathrm{red}}$.  

\begin{enumerate}
\item Let $\A_i \subset \mathcal{F}_m^{\mathrm{red}}$ be the subset which are sequences of filtered dimensions $\mathbf{d} = \langle d_1,d_2,\cdots , d_k\rangle$ such that $d_1,d_2,\cdots ,d_k \leq i$.  
\item Let $\A'_i \subset \A_i$ be the subset which are sequences of filtered dimensions $\mathbf{d} = \langle d_1,d_2,\cdots , d_k\rangle$ such that either $k=1$ and $d_1=i$, or $k\geq 1$ and $d_1,d_2,\cdots ,d_k \leq i-1.$  
\end{enumerate}

Note that there are also natural inclusions $\delta_i: \A_i \hookrightarrow \A'_{i+1}$ for $i\geq 1$ and that for $i\geq m$, $\A'_i = \A_i = \mathcal{F}_m^{\mathrm{red}}$.  Hence, the sequence $$\A'_1 \subset \A_1 \subset \A'_2 \subset \A_2 \subset \cdots \mathcal{F}_m^{\mathrm{red}}$$ provides an exhaustive and finite filtration of $\mathcal{F}_m^{\mathrm{red}}$.  In the following two lemmas, we show that the $\F_p$-homology of the colimit of the functor $\mathbf{S}$ is unchanged as we move up this filtration.

\begin{lem}\label{lem:A'cofinal}
For all $i\geq 1$, the inclusion $\A'_i \subset \A_i$ is cofinal.  Thus, the natural map $$\colim_{D\in \A'_i} \mathbf{S}(D) \to \colim_{D\in \A_i} \mathbf{S}(D)$$ is an equivalence.  
\end{lem}
\begin{proof}
This follows from the fact that the inclusion $\A'_i \subset \A_i$ admits a left adjoint, which takes all instances of $i$ and splits them off (i.e., the string ``$, i,$" is replaced by ``$\rangle \langle i \rangle \langle$"). 
\end{proof}
\begin{lem}\label{lem:A'kan}
For $i\geq 1$, let $(\delta_i)_! : \Fun(\A_i,\Spaces) \to \Fun(\A'_{i+1},\Spaces)$ denote left Kan extension along the inclusion $\delta_i:\A_i \hookrightarrow \A'_{i+1}$.  Then the natural map $(\delta_i)_! (\mathbf{S}|_{\A_i})(D) \to \mathbf{S}|_{\A'_{i+1}}(D)$ induces an isomorphism on $\F_p$-homology for all $D\in \A'_{i+1}$.  Thus, the natural map $$\colim_{D\in \A_i} \mathbf{S}(D) \to \colim_{D\in \A'_{i+1}} \mathbf{S}(D)$$ induces an equivalence in $\F_p$-homology.  

\end{lem}
\begin{proof}
Consider some $D \in \A'_{i+1} - \A_i$ which we may write in the form $$D = (\mathbf{d}_1,\mathbf{d}_2,\cdots ,\mathbf{d}_j).$$  We need to show that the map $$\colim_{D'\in (\A_{i})_{/D}} \mathbf{S}(D') \to \mathbf{S}(D)$$ induces an equivalence in $\F_p$-homology.  Consider the subset of $\mathcal{M} \subset (\A_{i})_{/D}$ consisting of filtered dimension sequences $D' = (\mathbf{d}'_1, \mathbf{d}'_2, \cdots, \mathbf{d}'_j)$ satisfying the following conditions for $1\leq r\leq j$ (recall the notation from Definition \ref{defn:filtdim}):
\begin{enumerate}
\item $|\mathbf{d}'_r| = |\mathbf{d}_r|.$
\item If $l(\mathbf{d}_r)>1$, then $\mathbf{d}_r = \mathbf{d}'_r$.  
\end{enumerate}

In particular, $D$ must be obtained from $D'$ via some sequence of collapse maps.  It is easy to see that the inclusion of posets $\mathcal{M}\subset (\A_{i})_{/D}$ is cofinal.  But the poset $\mathcal{M}$ splits as a product $$\mathcal{M} = C_1\times C_2 \times \cdots \times C_j$$ of posets $C_r \subset \mathcal{F}^{\red}_{|\mathbf{d}_r|}$, where $C_r = \{ (\mathbf{d}_r)\} $ if $l(\mathbf{d}_r)>1$ and $C_r = \C_{|\mathbf{d}_r|}$ (as defined in Proposition \ref{prop:steinberg}) if $l(\mathbf{d}_r) =1.$  Thus, we have:
\begin{align*}
\colim_{D'\in (\A_{i})_{/D}} \mathbf{S}(D') &\simeq \colim_{D'\in \mathcal{M}} \mathbf{S}(D') \\
&\simeq \colim_{D'_1\in C_1}\mathbf{S}(D'_1) \times \cdots \times \colim_{D'_j\in C_j}\mathbf{S}(D'_j)\\
&\simeq_p \mathbf{S}((\mathbf{d}_1))\times \cdots \times \mathbf{S}((\mathbf{d}_j))\\
&\simeq \mathbf{S}(D)\\
\end{align*}
where $\simeq_p$ denotes $\F_p$-homology equivalence, and we have used Proposition \ref{prop:steinberg} for this equivalence.  
\end{proof}

We can now finish the proof of Theorem \ref{thm:KpartFp}.  Combining Lemmas \ref{lem:A'cofinal} and \ref{lem:A'kan}, we conclude that the natural map 
$$\colim_{D\in \A'_1}\mathbf{S}(D) \to \colim_{D\in \A'_m} \mathbf{S}(D) = \colim_{D\in \mathcal{F}^{\red}_m}\mathbf{S}(D)$$ induces an equivalence in $\F_p$-homology.  But $\A'_1$ is a one element set containing only the filtered dimension sequence $(\langle 1\rangle ,\cdots ,\langle 1\rangle)$ (with $m$ $\langle 1\rangle$'s).  Since  $\mathbf{S}((\langle 1\rangle ,\cdots , \langle 1\rangle)) = \mathbf{S}((\langle 1\rangle))^{m} \simeq BGL_1(\F_p)^{m}$ has no $\F_p$-homology in positive degrees, as desired.

\section{The $p$-complete Frobenius and the action of $B\Z_{\geq 0}$}\label{sect:frobstable}

In this section, we combine the work of the previous sections to prove Theorem A, which appears as Theorem \ref{thm:mainaction}.   Throughout, we fix a prime $p$ and restrict attention to Frobenius maps corresponding to elementary abelian $p$-groups.

\begin{notation*}
Let $\QVect \subset \mathcal{Q}$ be the full subcategory spanned by elementary abelian $p$-groups (i.e., finite dimensional $\F_p$-vector spaces).
\end{notation*}

In \S \ref{sub:Fstable}, we introduce the subcategory $\CAlg^F_p\subset \CAlg$ of \emph{$F_p$-stable} $\E_\infty$-rings, which are roughly defined to be the $p$-complete $\E_{\infty}$-rings for which all the generalized Frobenius maps (cf. \S \ref{sub:genfrob}) can be regarded as endomorphisms.  In \S \ref{sub:pf}, we turn to the proof of Theorem A.  The oplax action of $\mathcal{Q}$ on $\CAlg$ from Theorem \ref{thm:mainalgaction} restricts to an action of $\QVect$ on $\CAlg^F_p$.  Essentially by definition, this extends to an action of the $\infty$-category $\QVect[\mathcal{L}^{-1}]$, which is identified with $BK^{\mathrm{part}}(\F_p)$ by Corollary \ref{cor:QequalsS}.  To finish, we use the $\F_p$-homology equivalence $K^{\mathrm{part}}(\F_p) \to \Z_{\geq 0}$ of Theorem \ref{thm:KpartFp} to pass from an action of $BK^{\mathrm{part}}(\F_p)$ to an action of $B\Z_{\geq 0}$.  
We conclude \S \ref{sub:pf} by recording Theorem \ref{thm:gloalgpartK}, which describes the action of Frobenius on global algebras via partial $K$-theory.  While we do not apply this theorem later in the paper, we feel that it may be of independent interest.  Finally, in \S \ref{sub:perfect}, we include a brief discussion of \emph{$p$-perfect} $\E_{\infty}$-rings, which are the $F_p$-stable $\E_{\infty}$-rings for which the Frobenius is an equivalence.  In particular, we note that the $\infty$-category of $p$-perfect $\E_{\infty}$-rings admits an action of $S^1$ by Frobenius (Corollary \ref{cor:perfectaction}).

\subsection{$F_p$-stable $\E_{\infty}$-rings}\label{sub:Fstable}

\begin{defn}\label{defn:frobstable}
We say that a spectrum $X$ is \emph{$F_p$-stable} if $X$ is $p$-complete and for every finite dimensional $\F_p$-vector space $V$, the canonical map (Construction \ref{cnstr:gencanmaps})
 $$\mathrm{can}^V: X\to X^{\tau V}$$ is an equivalence.  Let $\CAlg_{p}^{F}\subset \CAlg$ denote the full subcategory of $\E_{\infty}$-rings whose underlying spectrum is $F_p$-stable.  
\end{defn}

\begin{rmk}
If $X$ is $F_p$-stable and $V\twoheadrightarrow W$ is a surjection of $\F_p$-vector spaces, then by examining the commutative diagram 
\begin{equation*}
\begin{tikzcd}
 & X^{\tau W} \arrow[rd, "\can_W^V"] & \\ 
 X \arrow[ru, "\can^W"] \arrow[rr, "\can^V"] & &X^{\tau V},
\end{tikzcd}
\end{equation*}
we see that the generalized canonical map $\can_W^V$ is also an equivalence.
\end{rmk}

\begin{warn}
A discrete $\F_p$-algebra is generally \emph{not} $F_p$-stable when regarded as a spectrum.  For example, the Eilenberg-MacLane spectrum $\F_p$ is not $F_p$-stable because $\F_p^{tC_p}$ has nontrivial homotopy groups in every degree.  
\end{warn}

We saw in Example \ref{exm:introlin} that the canonical map $\pS \to (\pS)^{tC_p}$ is an equivalence by work of Lin and Gunawardena.  This admits the following extension, which is essentially the Segal conjecture for elementary abelian $p$-groups, due to work of Adams, Gunawardena, and Miller:

\begin{thm}[Adams-Gunawardena-Miller \cite{AGM}]\label{thm:agmsegal}
Let $V$ be a finite dimensional $\F_p$-vector space.  Then the canonical map $$\can^V : \pS \to (\pS)^{\tau V}$$ is an equivalence.  
\end{thm}

In fact, the statement holds with $V$ replaced by any $p$-group $G$, as a consequence of the Segal conjecture, which is a theorem of Carlsson \cite{Carlsson}.  To see this, note that by \cite[Proposition B]{MayMcClure}, the Segal conjecture for a $p$-group $G$ implies that $(\pS)_G$, the $p$-complete genuine $G$-equivariant sphere, is Borel.  Hence, the natural Borelification map $(\pS)_G \to \beta (\pS)_G$ is an equivalence, and applying $\Phi^G$ to both sides yields the result (cf. \Cref{cnstr:gencanmaps}).  

\begin{cor}\label{cor:Fstab}
The $p$-complete sphere is $F_p$-stable.  Since the proper Tate construction commutes with finite colimits, it follows that any spectrum which is finite over the $p$-complete sphere is also $F_p$-stable.  
\end{cor}

Another family of $F_p$-stable spectra will be important for us in Section \ref{sect:padic}.    

\begin{exm}[\cite{Ell2}, Example 5.2.7]\label{exm:sphericalWitt}
Let $R$ be a (discrete) perfect $\F_p$-algebra.  Then there is an essentially unique flat $\E_{\infty}$-algebra $W^+(R)$ over $\pS$ with the following properties:
\begin{enumerate}
\item $W^+(R)$ is $p$-complete.
\item The map induced on $\pi_0$ by the unit map $\pS \to W^+(R)$ is the natural map $\Z_p \to  W(R)$.
\item For any connective $p$-complete $\E_{\infty}$-algebra $A$ over $\pS$, the canonical map $$\Map_{\CAlg_{\pS}}(W^+(R),A) \to \Hom_{\F_p}(R, \pi_0(A)/p)$$ is an equivalence.  In particular, these spaces are discrete.    
\end{enumerate}
We will refer to $W^+(R)$ as the \emph{spherical Witt vectors} of $R$.
\end{exm}

\begin{prop}\label{prop:sfpbar}
Let $R$ be a (discrete) perfect $\F_p$-algebra.  Then the spectrum $W^+(R)$ is $F_p$-stable. 
\end{prop}

Note that $F_p$-stability is a condition only on the underlying spectrum of $W^+(R)$, which is the $p$-completion of a direct sum of copies of $\pS$.  When $R$ is finite dimensional as an $\F_p$-vector space, the proposition is immediate because $(-)^{\tau V}$ commutes with finite colimits. The content of the proposition is that the $p$-completion of an \emph{arbitrary} direct sum of copies of $\pS$ remains $F_p$-stable.  We will need the following lemma.  

\begin{lem}\label{lem:tatesum}
Let $I$ be a set and let $V$ be a finite dimensional $\F_p$-vector space.  Then the following statements hold:
\begin{enumerate}
\item The natural map $$\bigoplus_{\alpha \in I} (\pS)^{tV} \to \big(\bigoplus_{\alpha \in I} \pS\big)^{tV}$$ is $p$-completion.  
\item The natural map $$\big(\bigoplus_{\alpha \in I} \pS\big)^{tV} \to \big((\bigoplus_{\alpha \in I} \pS\big)^{\wedge}_p)^{tV}$$ is an equivalence.  
\end{enumerate}
\end{lem}
\begin{proof}

By \cite[Lemma I.2.9]{NS}, the target of (1) and both spectra in (2) are $p$-complete.  Therefore, it suffices to show that the two maps are equivalences after smashing with the Moore spectrum $\bS/p$.  This is clear for (2), so we address (1).  The statement is clear with $(-)^{tV}$ replaced by $(-)_{hV}$, so we may replace $(-)^{tV}$ by $(-)^{hV}$.  Since $(-)^{hV}$ commutes with inverse limits, we have that 
\begin{equation}\label{eqn:tatesum0}
\lim_n (\tau_{\leq n} (\bS/p))^{hV} \simeq (\bS/p)^{hV}.\end{equation}
  However, we can make a stronger statement about the convergence of this inverse limit.  Note that by the Segal conjecture, $\bS^{hV}$ is a finite wedge sum of spectra of the form $\Sigma^{\infty}_+ BG$ for finite groups $G$; thus, all the homotopy groups of $(\bS/p)^{hV}$ are finite.  Moreover, for each $n$, $\tau_{\leq n}(\bS/p)$ has a finite number of nonzero homotopy groups, each of which is finite; thus, by the homotopy fixed point spectral sequence, each term in the limit (\ref{eqn:tatesum0}) also has finite homotopy groups.  It follows from \cite[Lemma III.1.8]{NS} that the pro-system $\pi_i (\tau_{\leq n}\bS/p)^{hV}$ is pro-constant with value $\pi_i (\bS/p)^{hV}$.  

It follows that the pro-system $\lim_n \pi_i \bigoplus (\tau_{\leq n} (\bS/p))^{hV}$ is pro-constant with limit $\pi_i \bigoplus (\bS/p)^{hV}$.  On the other hand, by the Milnor exact sequence, this pro-constancy (for $\pi_i$ and $\pi_{i+1}$) implies that $\lim_n \pi_i \bigoplus (\tau_{\leq n}(\bS/p))^{hV} \simeq \pi_i \lim_n \bigoplus (\tau_{\leq n} (\bS/p))^{hV}.$  We relate this latter object to $\pi_i (\bigoplus \bS/p)^{hV}$ via a sequence of equivalences:  because homotopy fixed points commutes with filtered colimits of bounded above spectra, we have $$\pi_i \lim_n \bigoplus \big(\tau_{\leq n} (\bS/p)\big)^{hV} \simeq \pi_i \lim_n \big(\bigoplus \tau_{\leq n} (\bS/p)\big)^{hV}.$$ Since homotopy groups commute with infinite direct sums, we have $$\pi_i \lim_n \big(\bigoplus \tau_{\leq n} (\bS/p)\big)^{hV} \simeq \pi_i \lim_n \big(\tau_{\leq n}  \bigoplus \bS/p\big)^{hV}.$$ To finish, we note once again that $(-)^{hV}$ commutes with inverse limits.  
\end{proof}


\begin{cor}\label{cor:tatesum}
Let $I$ be a set and let $V$ be a finite dimensional $\F_p$-vector space.  Then the following statements hold:
\begin{enumerate}
\item The natural map $$\bigoplus_{\alpha \in I} (\pS)^{\tau V} \to \big(\bigoplus_{\alpha \in I} \pS\big)^{\tau V}$$ is $p$-completion.  
\item The natural map $$\big(\bigoplus_{\alpha \in I} \pS\big)^{\tau V} \to \big((\bigoplus_{\alpha \in I} \pS\big)^{\wedge}_p)^{\tau V}$$ is an equivalence.  
\end{enumerate}
\end{cor}
\begin{proof}
We recall from Remark \ref{rmk:propertatecofib} that the proper Tate construction for $V$ is the cofiber of a map $C \to (-)^{hV}$ where $C$ is a finite colimit of functors of the form $((-)^{hW})_{hV/W}$ for subgroups $W\subset V$.  Letting $F$ denote the cofiber of the natural map $(-)_{hV} \to C$, we consider the diagram of functors
\begin{equation*}
\begin{tikzcd}
(-)_{hV} \arrow[r] \arrow[d] & (-)^{hV} \arrow[r] \arrow[d] & (-)^{tV}\arrow[d] \\
C \arrow[r] \arrow[d] &(-)^{hV} \arrow[r]\arrow[d] &(-)^{\tau V} \arrow[d]\\
F \arrow[r] & * \arrow[r]& \Sigma F \\
\end{tikzcd}
\end{equation*}
where the rows and columns are cofiber sequences.  By the above remarks, $F$ is a finite colimit of functors of the form $((-)^{tW})_{hV/W}$ for proper subgroups $W\subset V$.  It follows from Lemma \ref{lem:tatesum} that the functors $(-)^{tV}$ and $\Sigma F$ both have the property that they send the $p$-completion map $$\bigoplus \pS \to (\bigoplus \pS\big)^{\wedge}_p$$ to an equivalence.  Thus, $(-)^{\tau V}$ also has this property and part (2) is proved.  

To prove part (1), we first claim that $\big(\bigoplus \pS\big)^{\tau V}$ is $p$-complete.  We have seen above that $\big(\bigoplus \pS\big)^{tV}$ is $p$-complete, so it suffices to see that $F(\bigoplus \pS)$ is as well. For this, we show that  $$\big(\big(\bigoplus \pS\big)^{tW}\big)_{hV/W} \simeq \big(\bigoplus \pS\big)^{tW} \otimes B(V/W)$$ is $p$-complete for any subgroup $W\subset V$.  This amounts to checking that the inverse limit of the system
$$\cdots \xrightarrow{p} \big(\bigoplus \pS\big)^{tW}\otimes B(V/W) \xrightarrow{p} \big(\bigoplus \pS\big)^{tW}\otimes B(V/W) \xrightarrow{p} \big(\bigoplus \pS\big)^{tW}\otimes B(V/W)$$ is zero. Since $\big(\bigoplus \pS\big)^{tW}$ is bounded below, each homotopy group of this inverse limit only depends on a finite skeleton of $B(V/W)$.  Thus, the statement follows because the tensor product of the $p$-complete spectrum $\big(\bigoplus \pS\big)^{tW}$ with any finite complex is certainly $p$-complete.  

To finish, we claim that the natural map of (1) is an equivalence after tensoring with the Moore spectrum $\bS/p$.  Similarly to part (2), it suffices to prove the analogous statement for each functor $((-)^{tW})_{hV/W}$ in place of $(-)^{\tau V}$.  But this follows by Lemma \ref{lem:tatesum} and the fact that homotopy orbits commutes with infinite direct sums.  
\end{proof}

\begin{proof}[Proof of Proposition \ref{prop:sfpbar}]
Choose a basis $\{x_{\alpha} \}_{\alpha \in I}$ for $R$ as an $\F_p$-vector space.  This determines a map $\bigoplus_{\alpha \in I} \pS \to W^+(R)$ which exhibits $W^+(R)$ as the $p$-completion of $\bigoplus_{\alpha \in I} \pS$.  The proposition then follows by combining Theorem \ref{thm:agmsegal} with Corollary \ref{cor:tatesum}.  
\end{proof}

\subsection{The action of $B\Z_{\geq 0}$ on $F_p$-stable $\E_{\infty}$-rings}\label{sub:pf}

We now prove our main theorem.  

\begin{thm}[Theorem A]\label{thm:mainaction}
There is an action of $B\Z_{\geq 0}$ on the $\infty$-category $\CAlg^F_p$ of $F_p$-stable $\E_{\infty}$-rings for which $n \in \Z_{\geq 0}$ acts by the natural transformation $\varphi^n: \id \to \id$.
\end{thm}

\begin{proof}
The oplax monoidal functor of Theorem \ref{thm:mainalgaction} restricts to an oplax monoidal functor $$\QVect \to \Fun(\CAlg,\CAlg).$$  By definition, the functor $(-)^{\tau V}$ fixes any $F_p$-stable $\E_{\infty}$-ring, so we obtain an oplax monoidal functor $$\Theta_p: \QVect \to \Fun(\CAlg^F_p,\CAlg^F_p).$$  In fact, for $\F_p$-vector spaces $U,V$ and an $\E_{\infty}$-ring $A$, the oplax structure map $A^{\tau U\oplus V} \to (A^{\tau U})^{\tau V}$ fits into a square 
\begin{equation*}
\begin{tikzcd}
A \arrow[r,"\can^V"]\arrow[d,"\can^{U\oplus V}"] & A^{\tau V} \arrow[d,"(\can^U)^{\tau V}"] \\
A^{\tau U\oplus V}\arrow[r] & (A^{\tau U})^{\tau V}.
\end{tikzcd}
\end{equation*}
When $A$ is $F_p$-stable, the three labeled maps are equivalences, and thus the bottom map is an equivalence as well.  It follows that the functor $\Theta_p$ is (strong) monoidal, rather than just oplax monoidal. Moreover, by the definition of $F_p$-stable, $\Theta_p$ has the property that the backward morphisms $\mathcal{L} \subset \mathrm{Mor}(\QVect)$ are sent to equivalences. Thus, by \cite[Proposition 4.1.7.4]{HA}, $\Theta_p$ extends to a functor of monoidal $\infty$-categories $$\QVect \lbrack \mathcal{L}^{-1}\rbrack \to \Fun(\CAlg^F_p, \CAlg^F_p).$$

By Corollary \ref{cor:QequalsS} and Proposition \ref{prop:cssinterp}, we have an equivalence of $\infty$-categories $BK^{\mathrm{part}}(\F_p) \simeq \QVect [\mathcal{L}^{-1}].$  The natural map $K^{\mathrm{part}}(\F_p) \to \pi_0(K^{\mathrm{part}}(\F_p)) \simeq \Z_{\geq 0}$ induces the diagram of monoidal $\infty$-categories given by the solid arrows:
\begin{equation*}
\begin{tikzcd}
BK^{\mathrm{part}}(\F_p)  \arrow[r]\arrow[d] &  \Fun(\CAlg^F_p, \CAlg^F_p). \\
B\Z_{\geq 0} \arrow[ru, dashed] &
\end{tikzcd}
\end{equation*}

It suffices to exhibit a monoidal functor filling in the dotted arrow.  By Remark \ref{rmk:connectedinftycat}, this is equivalent to providing a lift in the diagram of $\E_2$-spaces 
\begin{equation*}
\begin{tikzcd}
K^{\mathrm{part}}(\F_p) \arrow[r]\arrow[d] & \End (\mathrm{id}_{\CAlg^F_p}) \\
\Z_{\geq 0} \arrow[ru, dashed]  & 
\end{tikzcd}
\end{equation*}
where $\End (\mathrm{id}_{\CAlg^F_p}) $ denotes the $\E_2$-monoidal space of endomorphisms of the identity functor on the $\infty$-category $\CAlg^F_p$.  
By Theorem \ref{thm:KpartFp}, the vertical map $K^{\mathrm{part}}(\F_p) \to \Z_{\geq 0}$ induces an isomorphism in $\F_p$-homology.  We would like to use this to show that the restriction map $$\Hom_{\Alg_{\E_2}(\Spaces)}(\Z_{\geq 0}, \End (\mathrm{id}_{\CAlg^F_p})) \to \Hom_{\Alg_{\E_2}(\Spaces)}(K^{\mathrm{part}}(\F_p), \End (\mathrm{id}_{\CAlg^F_p}))$$ is an equivalence.  Note that for $\E_2$-spaces $X$ and $Y$, the space of $\E_2$-maps $X\to Y$ is computed as a limit of mapping spaces from products of copies of $X$ to $Y$. Since products of $\F_p$-homology isomorphisms are $\F_p$-homology isomorphisms, it suffices to show that $\End(\mathrm{id}_{\CAlg^F_p})$ is a $p$-complete space.

But $\End(\mathrm{id}_{\CAlg^F_p})$ can be written as a limit of spaces of the form $\Hom_{\CAlg^F_p}(A,B)$ for $p$-complete $\E_{\infty}$-rings $A,B$ (cf. \cite[Proposition 5.1]{GepHaugNik} or \cite[Proposition 2.3]{Glasman}). 
These spaces, in turn, are each the limit of mapping spaces between $p$-complete spectra, which are $p$-complete.  Since the full subcategory of $p$-complete spaces is closed under limits, we conclude that $\End(\mathrm{id}_{\CAlg^F_p})$ is $p$-complete and the proof is complete.
\end{proof}

One can also prove a variant of this theorem for general global algebras:

\begin{thm}\label{thm:gloalgpartK}
\leavevmode
\begin{enumerate}
\item The $\infty$-category $\CAlg^{\Glo}$ of global algebras admits an action of the monoidal $\infty$-category $BK^{\mathrm{part}}(\Z)$ for which an abelian group $G$ acts by the Frobenius $\varphi_G: \id \to \id$.  
\item The full subcategory $\CAlg^{\Glo}_p\subset \CAlg^{\Glo}$ of global algebras with $p$-complete underlying $\E_{\infty}$-ring admits an action of $B\Z_{\geq 0}$ for which $1\in \Z_{\geq 0}$ acts by the Frobenius (for the group $C_p$).  
\end{enumerate}
\end{thm}

\begin{proof}
To see the first part, restrict the action of \ref{thm:Qactgloalg} to the full subcategory of $\mathcal{Q}$ spanned by the abelian groups (cf. Remark \ref{rmk:Qabelian}).  This inverts the left morphisms and thus, by Corollary \ref{cor:QequalsS}, gives an action of $BK^{\mathrm{part}}(\Z)$.  The proof of the second part is exactly analogous to the proof of Theorem \ref{thm:mainaction}.  
\end{proof}

\subsection{Perfect $\E_{\infty}$-rings and the action of $S^1$}\label{sub:perfect}
\begin{defn}
An $\E_{\infty}$-ring $A$ is \emph{$p$-perfect} if $A$ is $F_p$-stable and the Frobenius map $\varphi  : A \to A^{tC_p}$ is an equivalence.  We will denote the $\infty$-category of $p$-perfect $\E_{\infty}$-rings by $\CAlg_p^{\mathrm{perf}}$.  
\end{defn}

\begin{prop}\label{prop:perfcrit}
Let $A$ be a $p$-perfect $\E_{\infty}$-ring.  Then for all finite dimensional $\F_p$-vector spaces $V$, the Frobenius map for $V$ (Construction \ref{cnstr:genfrobmaps}) $$\varphi^V:A \to A^{\tau V}$$ is an equivalence.  
\end{prop}
\begin{proof}
It suffices to show that $\varphi^{C_p^{\times n}}:A \to A^{\tau C_p^{\times n}}$ is an equivalence for all $n$.  Note that inside $\QVect$, the span $(*\leftarrow C_p^{\times n} \rightarrow C_p^{\times n})$ is the $n$-fold sum of $(*\leftarrow C_p \rightarrow C_p)$, and similarly $(*\leftarrow * \rightarrow C_p^{\times n}) = (* \leftarrow * \rightarrow C_p)^{\times n}$.  It follows from Theorem \ref{thm:mainalgaction} that there is a commutative diagram
\begin{equation*}
\begin{tikzcd}
 & ((A^{tC_p})^{\cdots})^{tC_p}& \\
A\arrow[ru, "\varphi^{\circ n}"] \arrow[r, "\varphi^{C_p^{\times n}}",swap] & A^{\tau C_p^{\times n}} \arrow[u]& A. \arrow[l, "\mathrm{can}^{C_p^{\times n}}"]\arrow[lu,"\mathrm{can}^{\circ n}",swap]\\
\end{tikzcd}
\end{equation*}
But all the outer arrows except for $\varphi^{C_p^{\times n}}$ are assumed to be equivalences, so that one is as well. 
\end{proof}

\begin{exm}
The $p$-complete sphere $\pS$ is $p$-perfect.  This is because the Frobenius is a map of $\E_{\infty}$-rings, and since $\pS$ is the initial $p$-complete $\E_{\infty}$-ring, the Frobenius must be the unit map for $(\pS)^{tC_p}$, which we have seen is an equivalence (e.g. \Cref{exm:introlin}).  
\end{exm}

\begin{exm}\label{exm:cochainsperfect}
Let $A\in \CAlg_p^{\mathrm{perf}}$ and $X$ be a finite or $p$-complete finite space.  Then the mapping spectrum $A^X$ is also a $p$-perfect $\E_{\infty}$-ring. First, since any space is canonically an $\E_{\infty}$-coalgebra in spaces, the construction $X \mapsto A^X$ takes values in $\E_{\infty}$-rings, and takes colimits of spaces to limits of $\E_{\infty}$-rings (which are computed in spectra).  Since the proper Tate construction commutes with finite limits and any ($p$-complete) finite space is built out of a finite number of ($p$-complete) spheres, the general case follows from the case when $X$ is a point, which is true by assumption.  In particular, for any ($p$-complete) finite space $X$, the cochain algebra $(\pS)^X$ is a $p$-perfect $\E_{\infty}$-ring.  
\end{exm}

\begin{exm}\label{exm:sphericalwittperfect}
Let $R$ be a (discrete) perfect $\F_p$-algebra.  Then the $\E_{\infty}$-ring $W^+(R)$ of spherical Witt vectors is $p$-perfect.  Since it is $F_p$-stable by Proposition \ref{prop:sfpbar}, it suffices to check that the Frobenius map is an equivalence.  Since $W^+(R)$ is $F_p$-stable, we may identify the Frobenius map with an endomorphism $\tilde{\varphi}: W^+(R) \to W^+(R)$ of $\E_{\infty}$-rings.  Such an endomorphism is determined on $\pi_0$ by the defining properties of spherical Witt vectors.  By comparing with the $\E_{\infty}$-Frobenius map for $\pi_0(W^+(R))$, we see that $\tilde{\varphi}$ is in fact the map induced by the Frobenius on $R$, and is therefore an equivalence because $R$ was assumed to be perfect.  
\end{exm}

\begin{exm}
The inclusion $\CAlg^F_p \subset \CAlg^{\mathrm{perf}}_p$ is proper.  For instance, consider  $\pS \oplus M$, the trivial square zero extension of $\pS$ by a nonzero spectrum $M$ which is finite over $\pS$.  While this ring is $F_p$-stable by \Cref{cor:Fstab}, we claim that it is not $p$-perfect.  By definition, the multiplication map $(\pS \oplus M)^{\otimes p}\to \pS \oplus M$ is $C_p$-equivariantly null on the $M^{\otimes p}$ summand in the source.  After applying $(-)^{tC_p}$ and using that $(-)^{tC_p}$ kills free orbits of $C_p$, the multiplication map becomes a map
\[
(\pS^{\otimes p})^{tC_p} \oplus (M^{\otimes p})^{tC_p} \simeq ((\pS \oplus M)^{\otimes p})^{tC_p} \to (\pS \oplus M)^{tC_p},
\]
which must then factor through $\pS^{tC_p}$ by the first observation.  Consequently, the $\E_{\infty}$-Frobenius on $\pS\oplus M$ factors through $\pS^{tC_p} \simeq \pS$, and therefore cannot be an equivalence.  
\end{exm}

The Frobenius action further simplifies when restricted to $p$-perfect algebras.  By Theorem \ref{thm:mainaction}, we obtain a monoidal functor $$B\Z_{\geq 0} \to \Fun(\CAlg^{\mathrm{perf}}_p,\CAlg^{\mathrm{perf}}_p). $$  By the definition of $p$-perfect, this restricted functor has the property that every morphism in $B\Z_{\geq 0}$ is sent to an equivalence.  It therefore factors through the group completion map $B\Z_{\geq 0}\to B\Z \simeq S^1$, and so we have:

\begin{cor}[Theorem $\text{A}^{\circ}$]\label{cor:perfectaction}
The $\infty$-category $\CAlgperfp$ of $p$-perfect $\E_{\infty}$-rings admits an action of $S^1$ whose monodromy induces the Frobenius automorphism on each object.
\end{cor}

\begin{rmk}
Corollary \ref{cor:perfectaction} can also be deduced directly from Theorem \ref{thm:mainalgaction} using Quillen's computation of $K(\F_p)$ (Theorem \ref{thm:quillenKFp}), and thus does not logically depend on the results of Sections \ref{sect:partK} and \ref{sect:KpartFp}.  It does not, however, imply Theorem \ref{thm:mainaction}, which applies to the larger class of $F_p$-stable $\E_{\infty}$-rings.  
\end{rmk}

\section{Integral models for the unstable homotopy category} \label{sect:padic}

In this final section, we apply our results on the Frobenius action to obtain models for spaces in terms of $\E_{\infty}$-rings.  We reiterate that the idea for this application is due to Thomas Nikolaus, and many of the key ideas are due to Mandell's work \cite{MandellZ}.  In \S \ref{subsect:padic}, we deduce the fully faithfulness statement of Theorem B from Theorem A.  In \S \ref{sub:essim}, we complete the proof of Theorem B by determining the essential image.  Then, in \S \ref{subsect:integral}, we deduce Theorem C from Theorem B.  Finally, in \S \ref{subsect:extensions}, we discuss some open questions and possible extensions of this work.  

\subsection{$p$-adic homotopy theory over the sphere}\label{subsect:padic}
Recall from the introduction that $\Spacesfnp \subset \Spaces$ denotes the full subcategory of spaces $X$ which simply connected and $p$-complete finite, and $\CAlg^{\varphi = 1}_p := (\CAlgperfp)^{hS^1}$ denotes the $\infty$-category of $p$-Frobenius fixed $\E_{\infty}$-rings.  Here we prove half of Theorem B:

\begin{thm}\label{thm:B1}
The functor $(\Spacesfnp)^{\op} \to \CAlg_{\pS}$ given by $X \mapsto (\pS)^X$ lifts to a fully faithful functor 
$$(\pS)^{(-)}_{\varphi=1}: (\Spacesfnp)^{\op} \to \CAlgffp.$$
\end{thm}

The proof begins by noting that for any $p$-perfect $\E_{\infty}$-ring $A$, one can extract a $p$-Frobenius fixed algebra of ``Frobenius fixed points" (Lemma \ref{lem:frobadj}).   
Composing this procedure with the functor $X \mapsto W^+(\Fpbar)^X$ (cf. Example \ref{exm:sphericalWitt}), we construct the functor $(\pS)^{(-)}_{\varphi=1}$ in the statement of Theorem B (Construction \ref{cnstr:thmBfun}).  By construction, this latter functor is fully faithful if and only if the functor $W^+(\Fpbar)^{(-)}: (\Spacesfnp)^{\op} \to  \CAlg_{W^+(\Fpbar)}$ is fully faithful.  We verify this by realizing an unpublished observation of Mandell which relates cochains with $W^+(\Fpbar)$ coefficients to the previously understood case of $\Fpbar$ cochains (Corollary \ref{cor:sfpbarlift}).  

\begin{lem}\label{lem:frobadj}
There is an adjunction $$i^*: (\CAlg^{\mathrm{perf}}_p)^{hS^1} \xrightleftharpoons{\quad} \CAlg^{\mathrm{perf}}_p : i_*$$ such that
\begin{itemize}
\item The left adjoint $i^*$ sends a $p$-Frobenius fixed algebra to its underlying $p$-perfect $\E_{\infty}$-ring.  
\item Any $p$-perfect algebra $A$ acquires an action of $\Z$ by Frobenius and the counit map $i^* i_* A \to A$ is homotopic to the natural map $A^{h\Z} \to A$.  
\end{itemize}
\end{lem}

\begin{rmk}\label{rmk:computeZfp}
The $\E_{\infty}$-ring $A^{h\Z}$ is computed by the procedure $$A^{h\Z} \simeq \lim_{S^1} A \simeq  \fib(\id_A -\varphi_A).$$  Thus, the above lemma says that given a $p$-perfect $\E_{\infty}$-ring $A$, one can take its ``Frobenius fixed points" by taking a limit over its $S^1$ orbit, and this procedure determines a $p$-Frobenius fixed algebra.  
\end{rmk}

\begin{proof}
Let us consider $BS^1$ as a pointed space, and denote the inclusion of the basepoint by $\{ * \} \xrightarrow{i} BS^1$.  The $S^1$ action of Corollary \ref{cor:perfectaction} determines a pullback square
\begin{equation*}
\begin{tikzcd}
\CAlg^{\mathrm{perf}}_p \arrow[d,"q"]\arrow[r] & \mathcal{D} \arrow[d,"f"] \\
\{ * \} \arrow[r,"i"] & BS^1 \\
\end{tikzcd}
\end{equation*}
where $f$ is a coCartesian (and Cartesian) fibration.  Let $\sect^{\circ}(f) \subset \sect (f)$ be the full subcategory of sections which send any morphism in $BS^1$ to a coCartesian morphism in $\D$.  We have an equivalence $\sect^{\circ}(f) \simeq (\CAlg^{\mathrm{perf}}_p)^{hS^1},$ and the functor $i^*$ of the lemma statement is the restriction of a section of $f$ along $i$.  

The adjoint will be given by $f$-right Kan extension along $i$ \cite[Definition 2.8]{QuigShah}.  By \cite[Corollary 2.11]{QuigShah}, this exists as long as for any $A\in \CAlg^{\mathrm{perf}}_p$, the induced diagram $$S^1\simeq *\times_{BS^1} BS^1_{*/} \to \D$$ has an $f$-limit.  Since $f$ is a Cartesian fibration, this relative limit reduces to taking a limit over a diagram $S^1 \to \CAlg^{\mathrm{perf}}_p$, which exists because $\CAlg^{\mathrm{perf}}_p$ has finite limits.  Moreover, the limit is computed as the fiber of $1-\varphi$ and the resulting section of $f$ is in the full subcategory $\sect^{\circ}(f)$.  
\end{proof}

\begin{notation*}
By analogy to the equivalence $\pS \simeq W^+(\F_p)$, we will hereafter use the more compact notation $\pSbar := W^+(\overline{\F}_p)$ to denote the spherical Witt vectors of $\Fpbar$ (cf. Example \ref{exm:sphericalWitt}). 
\end{notation*}

In Example \ref{exm:sphericalwittperfect}, we saw that $\pSbar$ is a $p$-perfect $\E_{\infty}$-ring and $\varphi_{\pSbar}$ is the unique ring endomorphism of $\pSbar$ which induces the Witt vector Frobenius on $\pi_0$.  

\begin{lem}\label{lem:counitwitt}
There is an equivalence of $\E_{\infty}$-rings  $i^* i_* \pSbar \simeq \pS$ under which the counit map $\epsilon: i^*i_*\pSbar \to \pSbar$ is homotopic to the unit map $\pS \to \pSbar$ for the ring $\pSbar$.  
\end{lem}
\begin{proof}
Since $\pS$ is the unit in $\CAlg^{\mathrm{perf}}_p$ and the counit map $\epsilon: i^* i_* \pSbar \to \pSbar$ is a ring map, the unit map for the ring $\pSbar$ factors canonically through $\epsilon$ as the composite $\pS \to i^*i_*\pSbar \xrightarrow{\epsilon} \pSbar$.  The unit map $\pS\to \pSbar$ is given on homotopy groups by the inclusion
 \begin{equation}\label{eqn:witttensor}
 \pi_*(\pS) \to  \pi_*(\pSbar ) \cong W(\Fpbar) \otimes_{\Z_p} \pi_*(\pS)
 \end{equation}
induced by the natural inclusion $\Z_p \to W(\Fpbar)$.  Thus, by the initial remarks, it suffices to show that the counit map $\epsilon: i^* i_*\pSbar \to \pSbar$ admits the same description on homotopy groups.  This follows directly by the fiber sequence
\[
i^* i_* \pSbar \xrightarrow{\epsilon} \pSbar \xrightarrow{1 - \varphi_{\pSbar}} \pSbar
\]
of \Cref{rmk:computeZfp} and the fact that the map $1-\varphi_{\pSbar }$ on homotopy groups is the map induced by $$1-\varphi_{W(\overline{\F}_p)} : W(\overline{\F}_p) \to W(\overline{\F}_p)$$ on the first tensor factor of (\ref{eqn:witttensor}), where $\varphi_{W(\overline{\F}_p)}$ is the usual Witt vector Frobenius.  
\end{proof}

\begin{cnstr}\label{cnstr:thmBfun}
We now construct the functor $$(\pS)^{(-)}_{\varphi=1}: (\Spacesfnp)^{\op}\to (\CAlgperfp)^{hS^1}$$  appearing in the statement of Theorem B.  Consider the functor $(\pSbar)^{(-)}:\Spaces^{\op} \to \CAlg.$  By Example \ref{exm:cochainsperfect}, this functor takes $p$-complete finite spaces to $p$-perfect $\E_{\infty}$-rings, so it restricts to a functor $$(\pSbar )^{(-)}: (\Spacesfnp)^{\op} \to \CAlg_p^{\mathrm{perf}}.$$  Define the desired functor by the formula
$$(\pS)^{(-)}_{\varphi =1} := i_*(\pSbar)^{(-)}: (\Spacesfnp)^{\op}\to (\CAlgperfp)^{hS^1}.$$  
By Lemma \ref{lem:counitwitt}, the functor $i^* (i_*(\pSbar)^{(-)})$ agrees with $(\pS)^{(-)}$ as required by the theorem statement. 
\end{cnstr}

The proof of \Cref{thm:B1} requires one additional fact:

\begin{prop}[Mandell, Lurie]  \label{prop:formallyetale}
Let $X$ be a space.  Then, the $\E_{\infty}$-algebra $\F_p^X$ is formally \'etale over $\F_p$, i.e. the cotangent complex $L_{\F_p^X/ \F_p}$ is contractible.
\end{prop}
\begin{proof}
This follows from the results of \cite{Lurpadic}: namely, the $\E_{\infty}$-algebra $\F_p^{X}$ is equivalent to the algebra of $\F_p$ cochains on the $p$-profinite completion of $X$ \cite[Notation 3.3.14]{Lurpadic}.  This is a (filtered) colimit of algebras of the form $\F_p^Y$ for $p$-finite $Y$, which are formally \'etale over $\F_p$ by \cite[Theorem 2.4.9]{Lurpadic}, and so the result follows because formally \'etale $\F_p$-algebras are closed under colimits.  
\end{proof}

\begin{cor}\label{cor:sfpbarlift}
Let $X\in \Spacesfnp$ and let $Y\in \Spaces$ be any space.  Then the natural map $$\CAlg_{\pSbar}((\pSbar)^X, (\pSbar)^Y ) \to \CAlg_{\overline{\F}_p}(\overline{\F}_p^X , \overline{\F}_p^Y)$$ is an equivalence.   
\end{cor}
\begin{proof}
Both sides send colimits in $Y$ to limits in mapping spaces, so it suffices to consider the case $Y=*$.  Since $X$ is assumed to be $p$-complete finite, the natural map $$(\pSbar)^X \otimes_{\pSbar} \Fpbar \to \Fpbar^X$$ is an equivalence, and so there is an equivalence $$\CAlg_{\overline{\F}_p}(\overline{\F}_p^X ,\overline{\F}_p) \simeq \CAlg_{\pSbar}((\pSbar)^X, \overline{\F}_p).$$ Thus, it suffices to see that the natural map $$\CAlg_{\pSbar}((\pSbar)^X, \pSbar) \to \CAlg_{\pSbar}((\pSbar)^X, \overline{\F}_p)$$ is an equivalence.   This follows from the following two claims:

\begin{claim1*}
For each $n$, the natural map $$\CAlg_{\pSbar}((\pSbar)^X, W_n(\overline{\F}_p)) \to \CAlg_{\pSbar}((\pSbar)^X, \overline{\F}_p)$$ is an equivalence.
\end{claim1*}

\begin{claim2*}
The natural map $$\CAlg_{\pSbar}((\pSbar)^X, \pSbar) \to \CAlg_{\pSbar}((\pSbar)^X, W(\Fpbar))$$ is an equivalence.
\end{claim2*}

\begin{proof}[Proof of Claim 1]
The proof is by induction.  The base case $n=1$ is a tautology.  Assume the statement has been proven for $n$ and consider the map  $$\gamma_*:  \CAlg_{\pSbar}((\pSbar)^X, W_{n+1}(\overline{\F}_p)) \to  \CAlg_{\pSbar}((\pSbar)^X, W_{n}(\overline{\F}_p))$$ induced by the projection $\gamma: W_{n+1}(\overline{\F}_p)\to W_n (\overline{\F}_p).$  The map $\gamma:W_{n+1}(\overline{\F}_p)\to W_n (\overline{\F}_p)$ is a square zero extension of $W_n(\Fpbar)$ by $\Fpbar$ (in the sense of \cite[\S 7.4.1]{HA}); thus, for any map $(\pSbar)^X \to W_n(\Fpbar)$ (which we now fix),  by  \cite[Remark 7.4.1.8]{HA}, the fibers of $\gamma_*$ are equivalent to the spaces of paths between certain pairs of points in the mapping space $ \Mod_{(\pSbar)^X} (L_{(\pSbar)^X / \pSbar}, \Sigma^{-1} \overline{\F}_p)$.  

We will finish by showing that, in fact, the mapping spectrum $ \Mod_{(\pSbar)^X} (L_{(\pSbar)^X / \pSbar}, \overline{\F}_p)$ is contractible.  Recall that since $X$ is $p$-complete finite, we have an equivalence $$(\pSbar)^X \otimes_{\pSbar} \overline{\F}_p \simeq \overline{\F}_p^X$$ of $\Fpbar$-algebras.  
By the base change properties of the cotangent complex, this yields an equivalence $$\Mod_{(\pSbar)^X} (L_{(\pSbar)^X / \pSbar}, \overline{\F}_p) \simeq \Mod_{(\overline{\F}_p)^X} (L_{(\pSbar)^X / \pSbar}\otimes_{(\pSbar)^X} \overline{\F}_p^X , \overline{\F}_p) \simeq \Mod_{\overline{\F}_p^X} (L_{\overline{\F}_p^X/ \overline{\F}_p}, \overline{\F}_p).$$  
This space is contractible because $$L_{\overline{\F}_p^X/ \overline{\F}_p} \simeq \overline{\F}_p^X \otimes_{\F_p^X} L_{\F_p^X/\F_p} \simeq 0$$ by Proposition \ref{prop:formallyetale}.  

\end{proof}

The proof of Claim 2 is similar, as $\pSbar \to W(\overline{\F}_p)$ factors as a sequence of square zero extensions by shifted copies of $\overline{\F}_p$.  It follows that the composite $$\CAlg_{\pSbar}((\pSbar)^X, \pSbar) \to \CAlg_{\pSbar}((\pSbar)^X, W(\Fpbar)) \to \CAlg_{\pSbar}((\pSbar)^X, \overline{\F}_p)$$ is an equivalence, which concludes the proof of Corollary \ref{cor:sfpbarlift}.  
\end{proof}

\Cref{thm:B1} now follows:
\begin{proof}[Proof of \Cref{thm:B1}]
Let $X,Y\in \Spacesfnp$.  We would like to show that the natural map $$\Spacesfnp(Y,X) \to \CAlgffp((\pS)^X_{\ff} , (\pS)^Y_{\ff})$$ is an equivalence.  Unwinding the definitions, we have
\begin{align*}
\CAlgffp((\pS)^X_{\ff} , (\pS)^Y_{\ff}) &\simeq \CAlgffp(i_*(\pSbar)^X, i_*(\pSbar)^Y)\\
&\simeq \CAlg_{\pS}(i^*i_*(\pSbar)^X, (\pSbar)^Y)\\
&\simeq \CAlg_{\pS}((\pS)^X, (\pSbar)^Y)\\
&\simeq \CAlg_{\pSbar}((\pSbar)^X, (\pSbar)^Y)\\
&\simeq\CAlg_{\overline{\F}_p}(\overline{\F}_p^X, \overline{\F}_p^Y)  &(\text{Corollary \ref{cor:sfpbarlift}})\\
&\simeq \Spacesfnp(Y, X),  &(\text{Theorem \ref{thm:Mandell}})\\
\end{align*}
as desired.  
\end{proof}

\subsection{The essential image}\label{sub:essim}

In this section, we will determine the essential image of the $p$-Frobenius fixed cochain functor, thus finishing the proof of Theorem B.   We first need some helpful terminology.

\begin{defn}
Let $R$ be a (discrete) commutative ring and let $A$ be an $\E_{\infty}$-algebra over $R$.  We will say $A$ is \emph{simply connected over $R$} if the unit map $R\to A$ induces an isomorphism in $\pi_k$ for $k\geq -1$.
\end{defn}

Our goal is to show:

\begin{thm}\label{thm:essentialimage}
Let $A \in \CAlgffp$ be a $p$-Frobenius fixed $\E_{\infty}$-ring whose underlying $\E_{\infty}$-ring is finite over $\pS$, and assume that $A\otimes \F_p$ is simply connected over $\F_p$.  Then there is a simply connected $p$-complete finite space $X$ such that $A \simeq (\pS)^X_{\varphi=1}$.  
\end{thm}

The main ingredient is a theorem of Mandell, which characterizes the essential image of the $\Fpbar$-cochain functor (on $p$-complete spaces of finite type).  The characterization is in terms of a certain power operation denoted $\Sq^0$ for $p=2$ and $P^0$ for $p$ odd (which we will refer to as $P^0$ at all primes for the sake of uniformity) which acts on the homotopy groups of any $\E_{\infty}$-algebra over $\F_p$; we refer the reader to \cite[\S 11]{Mandell} or \cite[\S 2.2]{Lurpadic} for background on $P^0$.

\begin{thm}[Mandell \cite{Mandell}]\label{thm:mandellimage}
The following conditions are equivalent for an $\E_{\infty}$-algebra $R$ over $\Fpbar$:
\begin{enumerate}
\item There exists a simply connected $p$-complete space $X$ of finite type and an equivalence of $\E_{\infty}$-algebras $\Fpbar^X \simeq R$.
\item The $\E_\infty$-algebra $R$ is simply connected over $\Fpbar$ and each $\pi_i(R)$ is a finite dimensional $\Fpbar$-vector space which is generated over $\Fpbar$ by classes fixed by the operation $P^0$.
\end{enumerate}
\end{thm}

\begin{prop}\label{prop:Fpbariscochains}
Let $A \in \CAlgffp$ be a $p$-Frobenius fixed $\E_{\infty}$-ring whose underlying spectrum is finite over $\pS$.  Then the power operation $P^0$ acts as the identity on $\pi_*(A\otimes \F_p)$.
\end{prop}
\begin{proof}
Because $(-)^{tC_p}$ is a lax symmetric monoidal functor, there is a natural map of $\E_{\infty}$-rings
\[
\nu: A^{tC_p}\otimes \F_p^{tC_p} \to(A \otimes \F_p)^{tC_p}.
\]
Since $\can: \id \to (-)^{tC_p}$ is a lax symmetric monoidal natural transformation, $\nu$ fits into a commutative diagram of $\E_{\infty}$-rings
\begin{equation}\label{eqn:canlax}
\begin{tikzcd}
A \otimes \F_p \arrow[r,"\can_{A \otimes \F_p}"]\arrow[rd, swap, "\can_{A} \otimes \can_{\F_p}"]& (A \otimes \F_p)^{tC_p}\\
& A^{tC_p}\otimes \F_p^{tC_p}\arrow[u,"\nu"].
\end{tikzcd}
\end{equation}
Moreover, let $\can_{\F_p}^* (A \otimes \F_p)$ denote $(A \otimes \F_p) \otimes_{\F_p} \F_p^{tC_p}$, where $\F_p^{tC_p}$ is regarded as an $\F_p$-algebra via $\can$.  Then, since $\nu$ is a $\F_p^{tC_p}$-algebra map, (\ref{eqn:canlax}) induces a commutative triangle

\begin{equation}\label{eqn:canlax2}
\begin{tikzcd}
\can^*(A \otimes \F_p) \arrow[r,"\can_{A \otimes \F_p}"]\arrow[rd, swap, "\can_{A} \otimes \can_{\F_p}"]& (A \otimes \F_p)^{tC_p}\\
&  A^{tC_p} \otimes \F_p^{tC_p}\arrow[u,"\nu"].
\end{tikzcd}
\end{equation}

Observe that all the morphisms in (\ref{eqn:canlax2}) are equivalences in the special case when $A= \pS$: this is clear for $\can_{A \otimes \F_p}$ and is the case for $\can_{A} \otimes \can_{\F_p}$ by the Segal conjecture (\Cref{exm:frobsphereequiv}); hence, it follows for $\nu$ by commutativity of the diagram.  As a consequence, it follows that all the maps in (\ref{eqn:canlax2}) are equivalences for any $A$ which is finite over $\pS$, as in the statement of the proposition.  

Since $\varphi: \id \to (-)^{tC_p}$ is a lax symmetric monoidal natural transformation (of endofunctors on $\E_{\infty}$-rings), one also has a commutative triangle analogous to (\ref{eqn:canlax}) with $\varphi$ replacing $\can$.  Pasting this triangle together with (\ref{eqn:canlax2}), we obtain a commutative diagram
\begin{equation}\label{eqn:canfroblax}
\begin{tikzcd}
A \otimes \F_p \arrow[r,"\varphi_{A \otimes \F_p}"]\arrow[rd,"\varphi_{A}\otimes \varphi_{\F_p}"'] & (A \otimes \F_p)^{tC_p} & \can^*(A \otimes \F_p) \arrow[l,"\sim","\can_{A \otimes \F_p}"'] \arrow[ld, "\can_{A}\otimes \can_{\F_p}", "\sim"']\\
& A^{tC_p} \otimes \F_p^{tC_p} \arrow[u,"\nu", "\sim"'] & .\\
\end{tikzcd}
\end{equation}
of $\E_{\infty}$-rings.  

To complete the proof, we compare the upper and lower composite maps $A \otimes \F_p \to \can^*(A \otimes \F_p) := (A \otimes \F_p)\otimes_{\F_p} \F_p^{tC_p}$ in this diagram.  In fact, we will compare, on homotopy groups, their composites with the $\F_p$-module map $\mu: \can^*(A \otimes \F_p) \to A \otimes \F_p$, induced by the canonical map (of $\F_p$-modules) $\F_p^{tC_p} \to \F_p$ which projects onto the unit copy of $\F_p$.  

\begin{enumerate}
\item The first map is the upper composite 
\[
\mu \circ (\can_{A \otimes \F_p})^{-1} \circ \varphi_{A \otimes \F_p}: A \otimes \F_p \to A \otimes \F_p.
\]
We note that by the formula of \Cref{exm:totalsteenrod}, this map is given on homotopy groups by the formula $x \mapsto P^0(x)$ (the projection $\mu$ exactly projects away from the components of $\F_p^{tC_p}$ involving $t$ or $e$).  
\item The second map is the lower composite 
\[
\mu\circ  (\can_{A}\otimes \can_{\F_p})^{-1}\circ  (\varphi_{A}\otimes \varphi_{\F_p}): A \otimes \F_p \to A \otimes \F_p.
\]
Since $A$ is assumed to be $p$-Frobenius fixed, the maps $\can_A : A \to A^{tC_p}$ and $\varphi_A: A \to A^{tC_p}$ are homotopic, and are both equivalences.  One deduces that the map $(\can_{A}\otimes \can_{\F_p})^{-1} (\varphi_{A}\otimes \varphi_{\F_p})$ is homotopic to the map $\id \otimes \varphi_{\F_p}  : A \otimes \F_p \to A\otimes \F_p^{tC_p} $.  By \Cref{exm:F2Frob}, the degree $0$ component of $\varphi_{\F_p}$ is the identity; it follows that the (spectrum) map $\mu \circ (\can_{A}\otimes \can_{\F_p})^{-1} \circ (\varphi_{A}\otimes \varphi_{F_p})$ is homotopic to the identity.  
\end{enumerate}

Since the diagram (\ref{eqn:canfroblax}) commutes, these two composites are homotopic, and thus we conclude that $P^0$ acts as the identity on $\pi_*(A \otimes \F_p)$.  
\end{proof}

\begin{cor}\label{cor:Fpbariscochains}
Let $A \in \CAlgffp$ be a $p$-Frobenius fixed $\E_{\infty}$-ring whose underlying spectrum is finite over $\pS$ and such that $A\otimes \F_p$ is simply connected over $\F_p$.  Then there exists a simply connected $p$-complete finite space $X$ and an equivalence $A\otimes \Fpbar \simeq \Fpbar^X$ of $\E_{\infty}$-algebras over $\Fpbar$.  
\end{cor}
\begin{proof}
We have $\pi_i(A\otimes \Fpbar) \simeq \pi_i(A\otimes \F_p)\otimes_{\F_p}\Fpbar$ for all $i$.  Since the natural map $A\otimes \F_p \to \A \otimes \Fpbar$ is a map of $\E_{\infty}$-algebras over $\F_p$, it commutes with $P^0$; hence, by \Cref{prop:Fpbariscochains}, we conclude that each $\pi_i(A\otimes \Fpbar)$ is generated by fixed points of $P^0$ as an $\Fpbar$ vector space.  Moreover, since $A$ is assumed to be finite over $\pS$, each $\pi_i(A)$ is a finite dimensional $\Fpbar$-vector space.  Hence, by Mandell's theorem (\Cref{thm:mandellimage}), we conclude that there is a simply connected $p$-complete space $X$ of finite type and an equivalence $\Fpbar^X \simeq A\otimes \Fpbar$ of $\Fpbar$-algebras.  Furthermore, again because $A$ is finite over $\pS$, the $p$-complete space $X$ is necessarily $p$-complete finite (and not just of finite type).  
\end{proof}

Finally, we are ready to prove the main theorem of the section.

\begin{proof}[Proof of \Cref{thm:essentialimage}]

By \Cref{cor:Fpbariscochains}, there exists a simply connected $p$-complete finite space $X$ and an equivalence $\eta: A\otimes \Fpbar \simeq \Fpbar^X$.  We will consider a composite:
\begin{align}\label{eqn:essimcomp}
\CAlgffp(A, \pSff) \xrightarrow{f_1} \CAlg(A, \pSbar) \xrightarrow{f_2} \CAlg(A, \Fpbar) &\simeq \CAlg_{\Fpbar}(A\otimes \Fpbar, \Fpbar)\\ 
&\simeq \CAlg_{\Fpbar}(\Fpbar^X, \Fpbar)\nonumber
\end{align}
where the final equivalence is the inverse of restriction along $\eta$.  

The map $f_1$ is given by forgetting the $p$-Frobenius-fixed structure and composing with the natural map $\pS\to\pSbar$.  By \Cref{lem:counitwitt}, this admits another description in terms of the $(i^*, i_*)$-adjunction of \Cref{lem:frobadj}; namely, since $A$ is assumed to be $p$-Frobenius fixed (i.e., we have written $A$ for $i^*A$ in the target of $f_1$), one has an equivalence 
\[
\CAlgffp( A, i_* \pSbar) \simeq \CAlgperfp(i^*A, \pSbar)
\]
and $i_* \pSbar$ is identified with $\pSff$ by \Cref{lem:counitwitt} (and the fact that $\pSff$ is the initial object).  This equivalence can be identified with $f_1$, which is therefore an equivalence.  

The map $f_2$ is obtained by composing with the natural map $\pSbar \to \Fpbar$.  Note that $f_2$ can also be identified with the natural map
\[
\CAlg_{\pSbar}(A\otimes \pSbar, \pSbar) \to \CAlg_{\pSbar}(A\otimes \pSbar, \Fpbar)
\]
induced by $\pSbar\to \Fpbar$.  We may now apply the argument of \Cref{cor:sfpbarlift} with $(\pSbar)^X$ replaced by $A\otimes \pSbar$: the only fact used about $(\pSbar)^X$ in that argument was that $(\pSbar)^X \otimes_{\pSbar} \Fpbar \simeq \Fpbar^X$ (which has vanishing cotangent complex over $\Fpbar$), and this fact is also satisfied by $A\otimes \pSbar$ by the initial remarks of this proof.  The argument implies $f_2$ is also an equivalence.  

We now return to the composite (\ref{eqn:essimcomp}), which we have now shown is an equivalence.  By Mandell's theorem (\Cref{thm:Mandell}), one has an equivalence $X \simeq \CAlg_{\Fpbar}(\Fpbar^X, \Fpbar)$.  Going backwards along the composite (\ref{eqn:essimcomp}), we obtain an equivalence (and in particular a map) $X \xrightarrow{\sim} \CAlg(A, \pSbar)$, which determines a map of $\E_{\infty}$-rings $A \to \pSbar^X$.  This induces a map of $p$-Frobenius-fixed $\E_{\infty}$-rings $A \to \pSff^X$; by construction, the tensor product of this map (on underlying $\E_{\infty}$-rings) with $\Fpbar$ is the equivalence of $\E_{\infty}$-rings $\eta: A\otimes \Fpbar \simeq \Fpbar^X$.  Thus, since $A$ and $\pS^X$ are finite over $\pS$, we conclude that the map $A\to \pSff^X$ is an equivalence.  

\end{proof}


\subsection{Integral models for spaces} \label{subsect:integral}

Here, we assemble Sullivan's rational homotopy theory with Theorem B to prove (see Definition \ref{defn:assembly} for notation):

\begin{thmC}
The functor
$$\bS^{(-)}_{\varphi = 1} : (\Spacesf)^{\op} \to \CAlg^{\varphi = 1}.$$ 
is fully faithful on the full subcategory $\Spacesfn \subset \Spacesf$ of simply connected finite complexes.  
\end{thmC}

\begin{proof}
As in the proof of Theorem B, it suffices to show that for any $X\in \Spacesfn$, the natural map $$X \to \CAlgff(\bS^X_{\ff} , \bS_{\ff})$$ is an equivalence.  By definition of the $\infty$-category $\CAlgff$, this amounts to showing that the square
\begin{equation}\label{dia:proofC}
\begin{tikzcd}
X \arrow[d]\arrow[r] & \prod_p \CAlgffp((\pS)^X_{\ff},(\pS)_{\ff}) \arrow[d]\\
\CAlg(\bS^X, \bS)\arrow[r] & \prod_p \CAlg((\pS)^X, \pS)\\
\end{tikzcd}
\end{equation}
is Cartesian.  

\begin{notation*}
Let $\bS^{\wedge} := \prod_p \pS$, $\Z^{\wedge} := \prod_p \Z_p$, $X^{\wedge} := \prod_p X_p^{\wedge}$, and $\Q^{\wedge} := \Z^{\wedge} \otimes \Q \simeq \bS^{\wedge}\otimes \Q$.  
\end{notation*}

Using the equivalence $ \prod_p \CAlg((\pS)^X, \pS) \simeq \CAlg(\bS^X, \bS^{\wedge})$, we may realize (\ref{dia:proofC}) as the top square in the following larger diagram:

\begin{equation}\label{dia:thmC}
\begin{tikzcd}
X \arrow[d]\arrow[r] & \prod_p \CAlgffp((\pS)^X_{\ff},(\pS)_{\ff}) \arrow[d]\\
\CAlg(\bS^X, \bS)\arrow[r] \arrow[d] & \CAlg(\bS^X, \bS^{\wedge})\arrow[d] \\
\CAlg(\bS^X, \Q)  \arrow[r] &  \CAlg(\bS^X,  \Q^{\wedge}),
\end{tikzcd}
\end{equation}
where the bottom square is the Cartesian square coming from considering maps of $\E_{\infty}$-rings from $\bS^X$ into the Cartesian square
\begin{equation*}
\begin{tikzcd}
\bS \arrow[r] \arrow[d]& \bS^{\wedge} \arrow[d]\\
\Q \arrow[r]  & \Q^{\wedge}.
\end{tikzcd}
\end{equation*}

It therefore suffices to show that the large outer rectangle in (\ref{dia:thmC}) is Cartesian.  We do this by showing that it is equivalent to the square
\begin{equation} \label{dia:arithC}
\begin{tikzcd}
X \arrow[r] \arrow[d]& X^{\wedge} \arrow[d]\\
X_{\Q} \arrow[r]  & (X^{\wedge})_{\Q},
\end{tikzcd}
\end{equation}
which is Cartesian by work of Sullivan \cite[Proposition 3.20]{Sullivan}.  Since the space $X$ is assumed to be a finite complex, we have equivalences 
\[
\CAlg(\bS^X, \Q) \simeq \CAlg_{\Q}((\bS^X)_{\Q}, \Q) \simeq \CAlg_{\Q}(\Q^X, \Q).
\]
Combining this with Sullivan's theorem on rational homotopy theory (Theorem \ref{thm:RHT}), we conclude that the left vertical composite in (\ref{dia:thmC}) is equivalent to the rationalization map $X \to X_{\Q}$.  On the other hand, by Theorem B,  the top map in (\ref{dia:thmC}) can be identified with the natural map $X \to \prod_p X^{\wedge}_p$.   Hence, by the universal property of rationalization, it suffices to show that the right vertical maps in (\ref{dia:thmC}) and (\ref{dia:arithC}) agree.

This follows from work of Mandell \cite{MandellZ}:  the right vertical composite in (\ref{dia:thmC}) can be factored as a composite
\[
\prod_p \CAlgffp((\pS)^X_{\ff},(\pS)_{\ff})  \to \CAlg(\bS^X, \bS^{\wedge}) \to \CAlg(\bS^X, \Z^{\wedge}) \to \CAlg(\bS^X, \Q^{\wedge}).
\]
We have seen that the first map can be identified with the natural map $X^{\wedge} \to \mathcal{L}X^{\wedge}$ given by taking the constant free loop (cf. \Cref{lem:frobadj}), and the second map is canonically an equivalence (by the product over primes $p$ of the $\F_p$ analog of \Cref{cor:sfpbarlift}).  Then, a theorem of Mandell \cite[Theorem 1.11]{MandellZ} asserts that the final map can be identified with the composite
\[
\mathcal{L}X^{\wedge} \to X^{\wedge} \to (X^{\wedge})_{\Q}.
\]
Combining these, we conclude that the right vertical composite in (\ref{dia:thmC}) is homotopic to the rationalization map $X^{\wedge} \to (X^{\wedge})_{\Q}$.  
\end{proof}

As remarked in the introduction, we have not determined the essential image of the Frobenius fixed cochain functor $\bSff^{(-)}$, though such a statement may well be within reach.  We formulate our guess as a question:

\begin{qst}\label{qst:Zessim}
Let $A\in \CAlgff$ be a Frobenius fixed $\E_{\infty}$-ring which is finite over $\bS$ and such that $A\otimes \Z$ is simply connected over $\Z$.  Then does there exist a simply connected finite complex $X$ and an equivalence of Frobenius fixed $\E_{\infty}$-rings $A\simeq \bSff^X$?
\end{qst}

\subsection{Further extensions and questions}\label{subsect:extensions}

We include here a series of remarks highlighting additional questions which are not addressed in this paper.  

1. Mandell's theorem (Theorem \ref{thm:Mandell}) applies to finite type $p$-complete spaces which are not necessarily finite.  The author does not know if there is a version of Theorem B which works in this generality.  For example, Mandell's theorem applies to the space $BC_p$, providing an identification $$\Hom_{\CAlg_{\F_p}}(\F_p^{BC_p}, \F_p) \simeq \mathcal{L}BC_p.$$  On the other hand, the following observation shows that the $\E_{\infty}$-algebra $(\pS)^{BC_p}$ fails to be $F_p$-stable, so our results do not apply.  

\begin{obs}
Letting $\simeq_p$ denote equivalence after $p$-completion, one articulation of the Segal conjecture for an abelian $p$-group $G$ is that $$(\pS)^{BG} \simeq_p \bigoplus_{H\subset G} \Sigma^{\infty}_+ B(G/H),$$ where the sum ranges over subgroups $H\subset G$ (cf. \cite[Chapter XX]{MayAlaska} for an exposition).  Thus, we have $(\pS)^{BC_p} \simeq_p \pS \oplus \Sigma^{\infty}_+ BC_p$, whereas $((\pS)^{BC_p})^{tC_p}$ is the cofiber of a map from 
$$((\pS)^{BC_p})_{hC_p} \simeq_p (\pS \oplus \Sigma^{\infty}_+ BC_p)_{hC_p} \simeq_p \Sigma^{\infty}_+ BC_p \oplus \Sigma^{\infty}_+ B(C_p\times C_p)$$
to
$$((\pS)^{BC_p})^{hC_p} = (\pS)^{BC_p\times C_p} \simeq_p \pS \oplus (\Sigma^{\infty}_+ BC_p)^{\oplus p+1} \oplus B(C_p\times C_p).$$  By inspection, one concludes the spectrum $(\pS)^{BC_p}$ is not $F_p$-stable.  
\end{obs}

Nevertheless, one can wonder whether one can recover the space $BC_p$ from the algebra $(\pS)^{BC_p}$:

\begin{qst}
What is the space of $\E_{\infty}$-ring maps $\Hom_{\CAlg}((\pS)^{BC_p},\pS)$?  
\end{qst}

2. Recall the 2-category $\widetilde{\mathcal{Q}}$ from Remark \ref{rmk:Qtilde}, and let $\widetilde{\QVect}$ denote the full subcategory on groupoids of the form $BV$ where $V$ is an $\F_p$-vector space.  Our methods show that in fact the $\E_{\infty}$-space $|\widetilde{\QVect}|$ acts on $\CAlg^{\mathrm{perf}}_p$ (and the analogous statement in the non-group-complete setting).

\begin{qst}
What is the homotopy type of $|\widetilde{\QVect}|$? For instance, is it $p$-adically discrete?  What about the corresponding partial $K$-theory?
\end{qst}

3. In \S \ref{sect:algact}, we defined a notion of global algebra which has genuine equivariant multiplication maps corresponding to all finite covers of groupoids.  It is natural to consider the analogous object with multiplicative transfers for \emph{all} maps of groupoids:

\begin{defn}
An \emph{extended global algebra} is a section of the fibration $\Psi^+: \GlopSp \to \Glop$ which is coCartesian over the left morphisms.  Let $\CAlg^{\Glop}$ denote the $\infty$-category of extended global algebras.  
\end{defn}

One can show that the functor $\bS^{(-)}: (\Spacesfn)^{\op} \to \CAlg$ refines to a functor $\bS_{\Glop}^{(-)}: (\Spacesfn)^{\op} \to \CAlg^{\Glop}.$  The additional functorialities corresponding to geometric fixed points can be seen as giving trivializations of the Frobenius maps.  We therefore conjecture that a space can be recovered from its cochains as an extended global algebra:

\begin{cnj}
The functor $\bS_{\Glop}^{(-)}:(\Spacesfn)^{\op} \to \CAlg^{\Glop}$ is fully faithful.
\end{cnj}


\appendix

\section{Generalities on producing actions}\label{sect:actgen} 

In this appendix, we will provide the technical details necessary for producing the integral Frobenius action of Theorem $\text{A}^{\natural}$ (Theorem \ref{thm:mainalgaction}).  Let us first describe, informally, a simpler case of the setup.  Let $q: \mE \to S$ be a coCartesian fibration of $\infty$-categories.  There is a monoidal $\infty$-category $\Fun(S,S)_{/\id}$ whose objects can be thought of as pairs $(\varphi, \eta)$ where $\varphi\in \Fun(S,S)$ and $\eta: \varphi \to \id$ is a natural transformation.  The monoidal structure is described on objects by the formula 
\begin{equation*}\label{eqn:endidmon}
(\varphi_1,\eta_1)\circ (\varphi_2,\eta_2) = (\varphi_1\circ \varphi_2, \varphi_1\circ \varphi_2 \xrightarrow{\eta_1(\varphi_2(-))} \varphi_2 \xrightarrow{\eta_2} \id).
\end{equation*} 
In this situation, we will show (\Cref{lem:endbact}) that the sections of $q$ admit a natural action of the monoidal $\infty$-category $\Fun(S,S)_{/\id}.$  Concretely, the action of $(\varphi,\eta)$ on a section $\sigma: S\to \mE$ produces a new section which takes $s$ to $ (\eta_s)_* \sigma(\varphi(s)).$

We work with an extension of this situation: for the remainder of this appendix, fix a pullback square of $\infty$-categories
\begin{equation*}
\begin{tikzcd}
\mE \arrow[r]\arrow[d,"q"] & \mE^+ \arrow[d,"q^+"]\\
S\arrow[r,"i"] & S^+\\
\end{tikzcd}
\end{equation*}
where the vertical arrows are coCartesian fibrations and $i$ is the inclusion of a (not necessarily full) subcategory.  We will see that the extension of $q$ to $q^+:\mE^+ \to S^+$ describes additional symmetries on the sections of $q$ beyond those coming from $\Fun(S,S)_{/\id}$ as above.  In particular, we produce an action of the $\infty$-category $\Fun(S,S)\times_{\Fun(S,S^+)} \Fun(S,S^+)_{/i}$.  Its objects are the data of a functor $\varphi:S\to S$ together with a natural transformation $\eta: i\varphi \to i$.  The monoidal structure can be described informally by the formula
$$(\varphi_1,\eta_1) \circ (\varphi_2,\eta_2) = (\varphi_1\circ \varphi_2, i\varphi_1\circ \varphi_2 \xrightarrow{\eta_1(\varphi_2(-))} i\varphi_2 \xrightarrow{\eta_2} i).$$  Before we construct the monoidal structure formally, we need some preliminary lemmas.  

\begin{lem}\label{lem:oplax}
Let $\C,\D$ be $\infty$-categories and $$F: \C \xrightleftharpoons{\quad} \D :G$$ be functors such that $F$ is left adjoint to $G$ and $G$ is fully faithful (so $F$ is a localization in the sense of \cite[Section 5.2.7]{HTT}).  Then there is a natural oplax monoidal functor $$\Fun(\C,\C)\to \Fun(\D,\D)$$ which sends a functor $\varphi:\C \to \C$ to $F\varphi G:\D \to \D$.  
\end{lem}
\begin{proof}
This is (the opposite of) \cite[Lemma 3.10]{BrantCampNuit}.
\end{proof}
\begin{rmk}
In fact, the hypotheses of Lemma \ref{lem:oplax} are stronger than should be necessary -- it should be sufficient that $F$ and $G$ be functors together with a natural transformation $\id \to GF.$  We would be happy to see a proof of this stronger statement.  
\end{rmk}

\begin{lem}\label{lem:endbact}
Let $B$ be an $\infty$-category and $q:\mE \to B$ be a Cartesian fibration.  Then:
\begin{enumerate}
\item The $\infty$-category $\End(B)_{/\id}$ has a monoidal structure which is given on objects by the formula
\[
(\varphi_1 \xrightarrow{\eta_1} \id)\circ (\varphi_2 \xrightarrow{\eta_2} \id) = (\varphi_1\circ \varphi_2 \xrightarrow{\eta_2\circ \eta_1(\varphi_2(-))} \id).
\]
\item The total space $\mE$ acquires a natural action of the monoidal $\infty$-category $\End(B)_{/\id}$ where the action of $(\eta:\phi \to \id) \in \End(B)_{/\id}$ takes $(b\in B, x\in \mE_b) \in \mE$ to $(\phi(b), \eta_x^* x\in \mE_{\phi(b)}).$
\end{enumerate}
\end{lem}
\begin{proof}
The monoidal structure on $\End(B)_{/\id}$ can be modeled explicitly by a simplicial monoid (which we will also denote by $\End(B)_{/\id}$) characterized as follows: for any finite nonempty linearly ordered set $J$, let $J^+$ denote $J\cup \{+ \}$ where $\{+ \}$ is a new maximal element; then, the set of maps $\Delta^J \to \End(B)_{/\id}$ is given by the set of maps $\Delta^{J^+} \times B \to B$ with the property that the restriction to $\{ + \} \subset \Delta^{J^+}$ is the identity.  The monoid structure takes two maps $f,g:\Delta^{J^+}\times B \to B$ to the composite $$g\circ f : \Delta^{J^+}\times B \to \Delta^{J^+} \times \Delta^{J^+} \times B \xrightarrow{\id \times g} \Delta^{J^+} \times B \xrightarrow{f} B$$ where the first map is induced by the diagonal on $\Delta^{J^+}.$  

The simplicial monoid $\End(B)_{/\id}$ left acts on the underlying simplicial set of the arrow category $\mathrm{Arr}(B)= \Fun(\Delta^1,B)$ of $B$.  On objects, the action of $(\eta: \phi \to \id)\in \End(B)_{/\id}$ on an arrow $(b\to c)\in \mathrm{Arr}(B)$ produces an arrow representing the composite $\phi(b) \xrightarrow{\eta_b} b \to c.$   

In general, we describe this action directly on $J$-simplices; let $f:\Delta^{J^+}\times B \to B$ determine a $J$-simplex of $\End(B)_{/\id}$ and $\gamma: \Delta^J \times \Delta^1 \to B$ be a $J$-simplex of $\mathrm{Arr}(B)$.  There is a map $\psi_J: \Delta^J \times \Delta^1 \to \Delta^{J^+}$ which sends $(j,0)$ to $j$ and $(j,1)$ to $+$ for all $j\in J$.  The action on $J$-simplices produces from $f$ and $\gamma$ a new $J$-simplex of $\mathrm{Arr}(B)$ defined by the composite $$f\gamma: \Delta^J \times \Delta^1 \xrightarrow{\psi_J \times \id} \Delta^{J^+} \times (\Delta^J\times \Delta^1) \xrightarrow{\id \times \gamma} \Delta^{J^+} \times B \xrightarrow{f} B.$$  It is immediate that this defines a left action, and so we have an action at the level of $\infty$-categories of the monoidal $\infty$-category $\End(B)_{/\id}$ on $\mathrm{Arr}(B).$   

Moreover, the action fixes the target of the arrow; thus, if we consider $\mathrm{Arr}(B)$ as lying over $B$ via the target map, $\End(B)_{/\id}$ acts on $\mathrm{Arr}(B)$ as an object over $B$.  Thus, for any functor $q:\mE \to B$, the fiber product $\mathrm{Arr}(B)\times_{B} \mE$ acquires a left action of $\End(B)_{/\id}.$   

One has a fully faithful functor $\mE \to \mathrm{Arr}(B)\times_B \mE$ given by sending $x\in \mE$ to $x$ together with the identity arrow $q(x) \to q(x)$.  When $q$ is a Cartesian fibration, this functor admits a right adjoint which takes $x\in \mE$ together with an arrow $f:y\to q(x)$ to the pullback $f^*x.$  By (the opposite of) Lemma \ref{lem:oplax}, we obtain a lax action of $\End(B)_{/\id}$ on $\mE$.  Explicitly, the action of $(\eta:\phi \to \id) \in \End(B)_{/\id}$ takes $(b\in B, x\in \mE_b) \in \mE$ to $(\phi(b), \eta_x^* x\in \mE_{\phi(b)}).$  We immediately see that the lax structure map is an equivalence, and so we have produced the desired action.
\end{proof}


Our strategy for producing and understanding the monoidal structure on $\Fun(S,S)\times_{\Fun(S,S^+)}\Fun(S,S^+)_{/i}$ is to construct an $\infty$-category $\mathcal{M}$ which is tensored over $\Cat_{\infty}$ and such that $\Fun(S,S)\times_{\Fun(S,S^+)}\Fun(S,S^+)_{/i} $ arises as the endomorphism $\infty$-category of an object of $\mathcal{M}$ in the sense of \cite[Section 4.7.1]{HA}.  

\begin{cnstr}\label{cnstr:slicemonoidal}
Let $\mathcal{M}\to \Cat_{\infty}$ denote the Cartesian fibration classified by the functor $\Cat_{\infty} \to \Cat_{\infty}^{\op}$ sending an $\infty$-category $\C$ to the functor category $\Fun(\C,S^+)^{\op}.$  One thinks of the objects of $\mathcal{M}$ as pairs $(\C, \varphi)$, where $\C$ is an $\infty$-category together with a functor $\varphi: \C\to S^+$.   By Lemma \ref{lem:endbact}, the $\infty$-category $\mathcal{M}$ acquires an action of $\End(\Cat_{\infty})_{/\id}$.  Since $\Cat_{\infty}$ is a Cartesian monoidal $\infty$-category, there is a monoidal functor $(\Cat_{\infty})_{/*} \to \End(\Cat_{\infty})_{/\id}$ where $*$ denotes the terminal $\infty$-category.  Since $*$ is terminal, the natural monoidal functor $(\Cat_{\infty})_{/*}\to \Cat_{\infty}$ is an equivalence.  We conclude that the $\infty$-category $\mathcal{M}$ is naturally left tensored over $\Cat_{\infty}$ in the sense of \cite[Definition 4.2.1.19]{HA}; explicitly, for $\D\in \Cat_{\infty}$, the tensor is described by the formula  $$D\otimes (\C, \varphi) = (\D\times \C, \D\times \C \to \C \xrightarrow{\varphi} S^+).$$  It therefore makes sense to discuss endomorphism $\infty$-categories of objects of $M$.  

\begin{lem}\label{lem:monendcat}
The object $(S,i)\in\mathcal{M}$ admits an endomorphism $\infty$-category which is equivalent to $\Fun(S,S)\times_{\Fun(S,S^+)}\Fun(S,S^+)_{/i}.$  
\end{lem}
\begin{proof}
Unwinding the definitions, there is a natural map in $\mathcal{M}$: $$\mathrm{ev}: \Fun(S,S)\times_{\Fun(S,S^+)}\Fun(S,S^+)_{/i} \otimes (S,i) \to (S,i).$$  We would like to check that this exhibits $\Fun(S,S)\times_{\Fun(S,S^+)}\Fun(S,S^+)_{/i}$ as an endomorphism object for $(S, i)\in \mathcal{M}$ in the sense that for any $\infty$-category $K$, the map $\mathrm{ev}$ induces a homotopy equivalence of spaces $$\Hom_{\Cat_{\infty}}(K , \Fun(S,S)\times_{\Fun(S,S^+)}\Fun(S,S^+)_{/i}) \xrightarrow{\sim} \Hom_{\mathcal{M}}( K\otimes (S,i), (S,i)).$$  %
To verify this, we note that both sides admit compatible maps to $\Hom_{\Cat_{\infty}}(K\times S,S).$  It suffices to choose a particular functor $\psi: K\times S\to S$ and check the claim fiberwise over $\psi$.  On the left hand side, this fiber is the fiber of the natural map $\Fun(K, \Fun(S,S^+)_{/i}) \to \Fun(K,\Fun(S,S^+))$ over the map $i\circ \psi : K\times S \to S^+$.  This can be identified as the space of natural transformations filling in the diagram
\begin{equation*}
\begin{tikzcd}
K\times S \arrow[r,"\psi"]\arrow[d] & S\arrow[d,"i"]\\
S\arrow[r,"i"] &S^+
\end{tikzcd}
\end{equation*}
where the left vertical arrow is projection onto the second coordinate.  This is evidently also the fiber on the right hand side and $\mathrm{ev}$ induces the desired equivalence.  
\end{proof}
\end{cnstr}

By \cite[Section 4.7.1]{HA}, this endows $\Fun(S,S)\times_{\Fun(S,S^+)}\Fun(S,S^+)_{/i}$ with the structure of a monoidal $\infty$-category.  Moreover, this description as an endomorphism object allows us to construct actions of the monoidal $\infty$-category  $\Fun(S,S)\times_{\Fun(S,S^+)}\Fun(S,S^+)_{/i}$ on other $\infty$-categories.  We will construct the right action of $\Fun(S,S)\times_{\Fun(S,S^+)}\Fun(S,S^+)_{/i}$ on $\sect(q)$ by first producing an action on a closely related $\infty$-category.

\begin{defn}\label{defn:laxsect}
Let $\sect^+(q)$ denote the $\infty$-category $\Fun(S,\mE^+)\times_{\Fun(S,S^+)} \Fun(S,S^+)_{/i}.$  
\end{defn}

\begin{obs}\label{obs:action}
The $\infty$-category $\sect^+(q)$ can be identified as the $\infty$-category of maps in $\mathcal{M}$ from $(S,i)$ to $(\mE^+,q^+).$  The verification is identical to Lemma \ref{lem:monendcat} so we omit it. 
As a result, $\sect^+(q)$ admits a canonical right action of $\End_{\mathcal{M}}((S,i)) \simeq \Fun(S,S)\times_{\Fun(S,S^+)}\Fun(S,S^+)_{/i}.$  
\end{obs}

\begin{rmk}
Concretely, an object of $\sect^+(q)$ is a pair $(f,\eta)$ where $f:S \to \mE^+$ is a functor together with a natural transformation $\nu: q^+\circ f \to i.$   The action of $(\varphi,\eta)\in\Fun(S,S)\times_{\Fun(S,S^+)}\Fun(S,S^+)_{/i}$ on $(f,\nu)$ produces $(f\circ \varphi, \nu')\in \sect^+(q)$ where $\nu'$ denotes the composite $$q^+ f \varphi \xrightarrow{\nu(\varphi(-))} i\varphi \xrightarrow{\eta} i.$$
\end{rmk}

Note that $\Fun(S,S^+)_{/i}$ has a final object given by $i \simeq i.$  The inclusion of this final object induces a fully faithful embedding $\{ *\} \to \Fun(S,S^+)_{/i}$, from which one obtains a fully faithful embedding $j$ via the diagram
\begin{equation}\label{dia:sectdefn}
\begin{tikzcd}
\sect(q) \arrow[r,"j",hookrightarrow] \arrow[d,equals]& \sect^+(q) \arrow[d,equals]\\
\Fun(S,\mE^+)\times_{\Fun(S,S^+)} \{ * \} \arrow[r,hookrightarrow ] & \Fun(S,\mE^+)\times_{\Fun(S,S^+)} \Fun(S,S^+)_{/i} .\\
\end{tikzcd}
\end{equation}

\begin{prop}\label{prop:sectadj}
The inclusion $j:\sect(q) \to \sect^+(q)$ admits a left adjoint.  
\end{prop}
\begin{proof}

We will apply the theory of marked simplicial sets as developed in \cite[Section 3.1]{HTT} and freely use the notation therein.  

Let $(f,\eta)\in \sect^+(q).$ We will start by defining a section $f'\in \sect(q)$ together with a map $\theta : (f,\eta) \to j(f')$ in $\sect^+(q).$  Concretely, $f'$ will be given by the formula $f'(s) = (\eta_{s})_* f(s)$ for $s\in S$.  

Consider the diagram of marked simplicial sets
\begin{equation*}
\begin{tikzcd}
S^{\flat} \times \{ 0\} \arrow[r]\arrow[d] & (\mE^+)^{\natural} \arrow[d,"q^+"]\\
S^{\flat}\times (\Delta^1)^{\sharp} \arrow[r] & (S^+)^{\sharp}
\end{tikzcd}
\end{equation*}
where the bottom arrow encodes the natural transformation $\eta: q^+ \circ f\to \id$ and the top arrow is defined by $f$.  The opposite of the left vertical arrow is marked anodyne, and $q^+$ is a coCartesian fibration, so we obtain a lift $S\times \Delta^1 \to \mE^+$, which determines a map $\theta$ from $(f,\eta)$ to its restriction to $S \times \{1 \}$, which we define to be the desired section $f'.$  

To show that $j$ has an left adjoint, it suffices to check that for any section $\sigma \in \sect(q)$, $\theta$ induces an equivalence $$\Hom_{\sect(q)}(f',\sigma) \simeq \Hom_{\sect^+(q)}((f,\eta), j(\sigma)).$$ 

We can compute these mapping spaces by expressing $\sect^+(q)$ and $\sect(q)$ as pullbacks as in diagram (\ref{dia:sectdefn}). In what follows, we will abuse notation by regarding the section $\sigma$ as a map $\sigma: S\to \mE^+$ together with an identification of $q^+\sigma$ with $i$, and similarly with $f'$.  Unwinding the definitions, one needs to show that $\theta$ induces an equivalence between the following two spaces:
\begin{enumerate}
\item The fiber of the map of spaces $\Hom_{\Fun(S,\mE^+)}(f',\sigma) \to \Hom_{\Fun(S,S^+)}(q^+\circ f' , q^+\circ \sigma)$ over the canonical element (corresponding to the fact that $f'$ and $\sigma$ are sections) which we will call $\nu_0 : q^+\circ f' \to q^+\circ \sigma$.
\item The fiber of the map of spaces $\Hom_{\Fun(S,\mE^+)}(f,\sigma) \to \Hom_{\Fun(S,S^+)}(q^+\circ f , q^+\circ \sigma)$ over the natural transformation $\nu_1:q^+\circ f\to q^+\circ \sigma$ given by precomposition of $\nu_0$ with the natural transformation $\nu_2: q^+\circ f\to q^+\circ f'$ determined by $\theta$.
\end{enumerate}

The natural transformations $\nu_0,\nu_1,\nu_2$ determine a 2-simplex $\nu: \Delta^2 \to \Fun(S,S^+)$ by sending the edge opposite $k$ to $\nu_k$.  We will consider the following subsets of $\Delta^2$ as marked simplicial sets where the edge $[0,1]$ is marked:
\begin{equation}\label{dia:subDelta2}
\begin{tikzcd}
\lbrack 0,1\rbrack \coprod \{2\} \arrow[r,"i_1"]\arrow[d,"i_0"] & \Lambda^2_1 \arrow[d,"j_1"]\\
\Lambda^2_0 \arrow[r,"j_0"] & \Delta^2.
\end{tikzcd}
\end{equation}
We have a commutative diagram 
\begin{equation}\label{dia:Delta2lift}
\begin{tikzcd}
\lbrack 0,1\rbrack \coprod \{2\} \arrow[r]\arrow[d] & \Fun(S,\mE^+) \arrow[d,"q^+_*"]\\
\Delta^2 \arrow[r] & \Fun(S,S^+)
\end{tikzcd}
\end{equation}
where the top horizontal map sends $[0,1]$ to the map $f\to f'$ induced by $\theta$ and sends $\{ 2\}$ to $\sigma$, and $q^+_*$ is a coCartesian fibration by \cite[Proposition 3.1.2.1]{HTT}.  

The space (1) above is the space of dotted lifts in the diagram
$$
\begin{tikzcd}
\lbrack 0,1\rbrack \coprod \{2\} \arrow[r]\arrow[d,"i_1"] & \Fun(S,\mE^+) \arrow[d,"q^+_*"]\\
\Lambda_1^2 \arrow[ru,dashed]\arrow[r] & \Fun(S,S^+).
\end{tikzcd}
$$
Analogously, the space (2) is the space of dotted lifts in the diagram
$$
\begin{tikzcd}
\lbrack 0,1\rbrack \coprod \{2\} \arrow[r]\arrow[d,"i_0"] & \Fun(S,\mE^+) \arrow[d,"q^+_*"]\\
\Lambda_0^2 \arrow[r]\arrow[ru,dashed] & \Fun(S,S^+).
\end{tikzcd}
$$
However, note that the opposites of the inclusions $j_0$ and $j_1$ (from Diagram (\ref{dia:subDelta2}) above) are marked anodyne.  Hence, by \cite[Proposition 3.1.3.3]{HTT}, both spaces are equivalent to the space of lifts in Diagram (\ref{dia:Delta2lift}) and thus are equivalent, as desired. 
\end{proof}

Applying Lemma \ref{lem:oplax} to the adjunction of Proposition \ref{prop:sectadj}, we find that there is an oplax monoidal functor $$\Fun(\sect^+(q),\sect^+(q)) \to \Fun(\sect(q) , \sect(q)).$$  Composing with the right action of Observation \ref{obs:action}, we obtain an oplax monoidal functor $$(\Fun(S,S)\times_{\Fun(S,S^+)}\Fun(S,S^+)_{/i})^{\mathrm{op}} \to \Fun(\sect(q),\sect(q)).$$  In fact, the oplax structure maps are equivalences, and so we have shown:

\begin{prop}\label{prop:mainaction}
Let $q^+:\mE^+ \to S^+$ be a coCartesian fibration of $\infty$-categories, let $i:S \to S^+$ be a subcategory, and let $q$ denote the restriction of $q^+$ along $i$.  Then there is a natural right action of $\Fun(S,S)\times_{\Fun(S,S^+)}\Fun(S,S^+)_{/i}$ on $\sect(q)$ which, for $\sigma \in \sect(q)$ and $(\varphi \xrightarrow{\eta} \id) \in \Fun(S,S)\times_{\Fun(S,S^+)}\Fun(S,S^+)_{/i}$, can be described on objects by the formula $$\sigma (\varphi \xrightarrow{\eta} \id) =  (s\mapsto \eta_{s*}\sigma(\varphi (s)))\in \sect(q).$$
\end{prop}

\bibliographystyle{plain}
\bibliography{Bibliography}

\end{document}